\documentclass{article}
\usepackage[utf8]{inputenc}

\usepackage{amsmath,amssymb,amsfonts,amscd, graphicx, latexsym, verbatim, multirow, color}
\usepackage{amsthm}
\usepackage{tikz}

\usetikzlibrary{shapes.geometric}

\usepackage{forest}
\usepackage{float}
\usepackage{graphicx}
\usepackage{amsmath,amssymb,amsfonts,amscd, graphicx, latexsym, verbatim, multirow, color, float, enumitem}
\usepackage{caption, subcaption}
\usepackage{mathtools} 
\usepackage{pgf, tikz}
\usetikzlibrary{patterns}
\usetikzlibrary{decorations.shapes}
\newtheorem{theorem}{Theorem}[section]
\newtheorem{proposition}[theorem]{Proposition}

\theoremstyle{definition}
\newtheorem{definition}[theorem]{Definition}
\newtheorem{example}[theorem]{Example}

\newtheorem{remark}[theorem]{Remark}

\makeatletter
\@addtoreset{equation}{section}
\makeatother

\newcommand{\NN}{{\mathbb N}}

\newcommand{\ZZ}{{\mathbb Z}}

\newcommand{\RR}{{\mathbb R}}

\newcommand{\cF}{{\mathcal F}}

\newcommand{\cI}{{\mathcal I}}
\newcommand{\cJ}{{\mathcal J}}

\newcommand{\cP}{{\mathcal P}}

\newcommand{\bb}{{\partial\mathrm{supp\ }}}

\newcommand{\sss}{{\mathrm{supp\ }}}

\usepackage{etoolbox}
\makeatletter
\patchcmd{\@maketitle}
  {\ifx\@empty\@dedicatory}
  {\ifx\@empty\@date \else {\vskip3ex \centering\footnotesize\@date\par\vskip1ex}\fi
   \ifx\@empty\@dedicatory}
  {}{}
\patchcmd{\@adminfootnotes}
  {\ifx\@empty\@date\else \@footnotetext{\@setdate}\fi}
  {}{}{}
\makeatother

\newcommand{\pentla}[4]{
\begin{scope}[xshift=#1cm, yshift=#2cm, rotate=#3, scale=#4]
\draw[] (0,0) -- ({sin(18)}, {sin(72)}) -- ({(-1+sqrt(5))/2}, 0) -- (0,0);
\filldraw (0.125,0.1) circle (1pt);
\end{scope}}

\newcommand{\pentlb}[4]{
\begin{scope}[xshift=#1cm, yshift=#2cm, rotate=#3, scale=#4]
\draw[] (0,0) -- ({sin(18)}, {sin(72)}) -- ({(-1+sqrt(5))/2}, 0) -- (0,0);
\filldraw (0.125,0.1) circle (1pt);
\filldraw (0.25,0.1) circle (1pt);
\end{scope}}

\newcommand{\pentlc}[4]{
\begin{scope}[xshift=#1cm, yshift=#2cm, rotate=#3, scale=#4]
\filldraw[color=gray!40] (0,0) -- ({sin(18)}, {sin(72)}) -- ({(-1+sqrt(5))/2}, 0) -- (0,0);
\draw[] (0,0) -- ({sin(18)}, {sin(72)}) -- ({(-1+sqrt(5))/2}, 0) -- (0,0);
\filldraw (0.125,0.1) circle (1pt);
\end{scope}}
\newcommand{\pentld}[4]{
\begin{scope}[xshift=#1cm, yshift=#2cm, rotate=#3, scale=#4]
\filldraw[color=gray!40][thick] (0,0) -- ({sin(18)}, {sin(72)}) -- ({(-1+sqrt(5))/2}, 0) -- (0,0);
\draw[] (0,0) -- ({sin(18)}, {sin(72)}) -- ({(-1+sqrt(5))/2}, 0) -- (0,0);
\filldraw (0.125,0.1) circle (1pt);
\filldraw (0.25,0.1) circle (1pt);
\end{scope}}

\newcommand{\pentra}[4]{
\begin{scope}[xshift=#1cm, yshift=#2cm, rotate=#3, scale=#4]
\draw[] (0,0) -- (-{sin(18)}, {sin(72)}) -- ({(1-sqrt(5))/2}, 0) -- (0,0);
\filldraw (-0.125,0.1) circle (1pt);
\end{scope}}

\newcommand{\pentrb}[4]{
\begin{scope}[xshift=#1cm, yshift=#2cm, rotate=#3, scale=#4]
\draw[] (0,0) -- (-{sin(18)}, {sin(72)}) -- ({(1-sqrt(5))/2}, 0) -- (0,0);
\filldraw (-0.125,0.1) circle (1pt);
\filldraw (-0.25,0.1) circle (1pt);
\end{scope}}

\newcommand{\pentrc}[4]{
\begin{scope}[xshift=#1cm, yshift=#2cm, rotate=#3, scale=#4]
\filldraw[color=gray!40] (0,0) -- (-{sin(18)}, {sin(72)}) -- ({(1-sqrt(5))/2}, 0) -- (0,0);
\draw[] (0,0) -- (-{sin(18)}, {sin(72)}) -- ({(1-sqrt(5))/2}, 0) -- (0,0);
\filldraw (-0.125,0.1) circle (1pt);
\end{scope}}

\newcommand{\pentrd}[4]{
\begin{scope}[xshift=#1cm, yshift=#2cm, rotate=#3, scale=#4]
\filldraw[color=gray!40] (0,0) -- (-{sin(18)}, {sin(72)}) -- ({(1-sqrt(5))/2}, 0) -- (0,0);
\draw[] (0,0) -- (-{sin(18)}, {sin(72)}) -- ({(1-sqrt(5))/2}, 0) -- (0,0);
\filldraw (-0.125,0.1) circle (1pt);
\filldraw (-0.25,0.1) circle (1pt);
\end{scope}}

\newcommand{\pensla}[4]{
\begin{scope}[xshift=#1cm, yshift=#2cm, rotate=#3, scale=#4]
\draw[] (0,0) -- ({sin(54)}, {sin(36)}) -- ({(1+sqrt(5))/2}, 0) -- (0,0);
\filldraw (0.3,0.1) circle (1pt);
\end{scope}}

\newcommand{\penslb}[4]{
\begin{scope}[xshift=#1cm, yshift=#2cm, rotate=#3, scale=#4]
\filldraw[color=gray!40] (0,0) -- ({sin(54)}, {sin(36)}) -- ({(1+sqrt(5))/2}, 0) -- (0,0);
\draw[] (0,0) -- ({sin(54)}, {sin(36)}) -- ({(1+sqrt(5))/2}, 0) -- (0,0);
\filldraw (0.3,0.1) circle (1pt);
\end{scope}}

\newcommand{\pensra}[4]{
\begin{scope}[xshift=#1cm, yshift=#2cm, rotate=#3, scale=#4]
\draw[] (0,0) -- (-{sin(54)}, {sin(36)}) -- ({-(1+sqrt(5))/2}, 0) -- (0,0);
\filldraw (-0.3,0.1) circle (1pt);
\end{scope}}

\newcommand{\pensrb}[4]{
\begin{scope}[xshift=#1cm, yshift=#2cm, rotate=#3, scale=#4]
\filldraw[color=gray!40] (0,0) -- (-{sin(54)}, {sin(36)}) -- ({-(1+sqrt(5))/2}, 0) -- (0,0);
\draw[] (0,0) -- (-{sin(54)}, {sin(36)}) -- ({-(1+sqrt(5))/2}, 0) -- (0,0);
\filldraw (-0.3,0.1) circle (1pt);
\end{scope}}

\newcommand{\penttla}[4]{
\begin{scope}[xshift=#1cm, yshift=#2cm, rotate=#3, scale=#4]
\pentla{1}{0}{108}{1}
\penslb{0.5}{1.538841768587626701285145288018454912003351071768896213519}{-108}{1}
\end{scope}}

\newcommand{\penttlb}[4]{
\begin{scope}[xshift=#1cm, yshift=#2cm, rotate=#3, scale=#4]
\pentlb{1}{0}{108}{1}
\pensla{0.5}{1.538841768587626701285145288018454912003351071768896213519}{-108}{1}
\end{scope}}

\newcommand{\penttlc}[4]{
\begin{scope}[xshift=#1cm, yshift=#2cm, rotate=#3, scale=#4]
\pentlb{1}{0}{108}{1}
\penslb{0.5}{1.538841768587626701285145288018454912003351071768896213519}{-108}{1}
\end{scope}}

\newcommand{\penttld}[4]{
\begin{scope}[xshift=#1cm, yshift=#2cm, rotate=#3, scale=#4]
\pentla{1}{0}{108}{1}
\pensla{0.5}{1.538841768587626701285145288018454912003351071768896213519}{-108}{1}
\end{scope}}
\newcommand{\penttra}[4]{
\begin{scope}[xshift=#1cm, yshift=#2cm, rotate=#3, scale=#4]
\pentra{-1}{0}{-108}{1}
\pensra{-0.5}{1.538841768587626701285145288018454912003351071768896213519}{108}{1}
\end{scope}}

\newcommand{\penttrb}[4]{
\begin{scope}[xshift=#1cm, yshift=#2cm, rotate=#3, scale=#4]
\pentrb{-1}{0}{-108}{1}
\pensrb{-0.5}{1.538841768587626701285145288018454912003351071768896213519}{108}{1}
\end{scope}}

\newcommand{\penttrc}[4]{
\begin{scope}[xshift=#1cm, yshift=#2cm, rotate=#3, scale=#4]
\pentrb{-1}{0}{-108}{1}
\pensra{-0.5}{1.538841768587626701285145288018454912003351071768896213519}{108}{1}
\end{scope}}

\newcommand{\penttrd}[4]{
\begin{scope}[xshift=#1cm, yshift=#2cm, rotate=#3, scale=#4]
\pentra{-1}{0}{-108}{1}
\pensrb{-0.5}{1.538841768587626701285145288018454912003351071768896213519}{108}{1}
\end{scope}}

\newcommand{\penssla}[4]{
\begin{scope}[xshift=#1cm, yshift=#2cm, rotate=#3, scale=#4]

\penslb{1.309016994374947424102293417182819058860154589902881431067}{0.951056516295153572116439333379382143405698634125750222447}{216}{1}
\pentrb{1.309016994374947424102293417182819058860154589902881431067}{0.951056516295153572116439333379382143405698634125750222447}{144}{1}
\pensra{2.618033988749894848204586834365638117720309179805762862135}{0}{0}{1}
\end{scope}}

\newcommand{\pensslb}[4]{
\begin{scope}[xshift=#1cm, yshift=#2cm, rotate=#3, scale=#4]

\pensla{1.309016994374947424102293417182819058860154589902881431067}{0.951056516295153572116439333379382143405698634125750222447}{216}{1}
\pentra{1.309016994374947424102293417182819058860154589902881431067}{0.951056516295153572116439333379382143405698634125750222447}{144}{1}
\pensrb{2.618033988749894848204586834365638117720309179805762862135}{0}{0}{1}
\end{scope}}

\newcommand{\penssra}[4]{
\begin{scope}[xshift=#1cm, yshift=#2cm, rotate=#3, scale=#4]
\pensrb{-1.309016994374947424102293417182819058860154589902881431067}{0.951056516295153572116439333379382143405698634125750222447}{144}{1}
\pentla{-1.309016994374947424102293417182819058860154589902881431067}{0.951056516295153572116439333379382143405698634125750222447}{216}{1}
\pensla{-2.618033988749894848204586834365638117720309179805762862135}{0}{0}{1}
\end{scope}}

\newcommand{\penssrb}[4]{
\begin{scope}[xshift=#1cm, yshift=#2cm, rotate=#3, scale=#4]
\pensra{-1.309016994374947424102293417182819058860154589902881431067}{0.951056516295153572116439333379382143405698634125750222447}{144}{1}
\pentlb{-1.309016994374947424102293417182819058860154589902881431067}{0.951056516295153572116439333379382143405698634125750222447}{216}{1}
\penslb{-2.618033988749894848204586834365638117720309179805762862135}{0}{0}{1}
\end{scope}}

\newcommand{\tri}[4]{
\begin{scope}[xshift=#1cm, yshift=#2cm, rotate=#3, scale=#4]
\filldraw [white,line width=1.5pt] (5,5) -- ([shift=(300:1cm)]5,5) -- (6,5) -- (5,5);
\filldraw[white,line width=1.5pt] (6,5) -- (4.95,5);
\end{scope}}

\newcommand{\trii}[4]{
\begin{scope}[xshift=#1cm, yshift=#2cm, rotate=#3, scale=#4]
\filldraw [gray!50,line width=1.5pt] (5,5) -- ([shift=(300:1cm)]5,5) -- (6,5) -- (5,5);
\filldraw[gray!50,line width=1.5pt] (6,5) -- (4.95,5);
\end{scope}}

\newcommand{\Addresses}{{% additional braces for segregating \footnotesize
  \bigskip
  \textsc{The Roslin Institute, The University of Edinburgh, Easter Bush Campus, EH25 9RG, Edinburgh, United Kingdom}\par\nopagebreak
  \textit{E-mail address:} \texttt{iozkarac@ed.ac.uk}

}}

\tikzset{
    buffer/.style={
        draw,
        shape border rotate=-90,
        isosceles triangle,
        isosceles triangle apex angle=60,
        fill=red,
        node distance=2cm,
        minimum height=4em
    }
}

\title{Planar Substitutions to Lebesgue type Space-Filling Curves and Relatively Dense Fractal-like Sets in the Plane}
\author{Mustafa \.{I}smail \"{O}zkaraca}
\date{\today} 
\begin{document}

\maketitle

\begin{abstract}
\emph{Lebesgue\ curve} is a space-filling curve that fills the unit square through linear interpolation. In this study, we generalise Lebesgue's construction to generate space-filling curves from any given planar substitution satisfying a mild condition. The generated space-filling curves for some known substitutions are elucidated. %The constructed curves are defined by a linear interpolation method as like the Lebesgue curve. 
Some of those substitutions further induce relatively dense fractal-like sets in the plane, whenever some additional assumptions are met. %One example of such is demonstrated at the end of the paper.%A number of such sets are demonstrated  at the end of the paper.
\end{abstract}

\section{Introduction}\label{Section_Introduction}

A \emph{space\ filling\ curve} of the plane is a continuous mapping defined from the unit interval to a subset of the plane that has a positive area. One of the most common examples of a space-filling curve is given by Lebesgue in \cite{Book_Lebesgue_SFC}. Lebesgue's space-filling curve is formed by linear interpolation over the map $\phi:\Gamma_c\mapsto [0,1]\times [0,1]$ given in \eqref{equation_cantor_map}, where $\Gamma_c$ is the middle-third Cantor set. The map $\phi$ is defined by the ternary representations of the points in the Cantor set, and is a continuous surjection onto the unit square.
\begin{equation}\label{equation_cantor_map}
    \phi\left(0._3\ (2\cdot x_1)(2\cdot x_2)(2\cdot x_3)\dots\right) = \begin{bmatrix}
0._2\ x_1 x_3 x_5\dots, \\
0._2\ x_2 x_4 x_6\dots, 
\end{bmatrix}.
\end{equation}
The geometric interpretation of the Lebesgue's curve is clarified by Sagan \cite{Journal_Sagan_1, Journal_Sagan_2}, by means of approximating polygons; a notion defined by Wunderlich \cite{Journal_Wunderlich}. These approximating polygons are analogous to the iteration steps of the Morton order, first three of which are depicted in Figure \ref{Figure_Morton_Order_First_Three_Iterations}. Using the approximating polygons, Sagan provided another proof in his book \cite{Book_Hans_Sagan}, indicating that the Lebesgue curve is a continuous surjection onto the unit square. This proof can be thought as a geometric construction of the Lebesgue curve.

A \emph{tile} $t$ consists of a subset of $\RR^n$ ($n\in\ZZ^+$) and an assigned colour label. We denote the associated subset by $\sss t$ and the label by $l(t)$. We assume that for every tile $t$, $\sss t$ is homeomorphic to the closed unit disc. A \emph{planar\ substitution} is a map defined over a collection of tiles in $\RR^2$ such that it expands every tile by a fixed factor greater than $1$ and divides each expanded tile into pieces, each of which is a translation of a tile. Throughout the paper we refer planar substitutions as substitutions in short. In this study we introduce an algorithm to produce space-filling curves from substitutions satisfying a weak condition, by mimicking the geometric construction of the Lebesgue's curve. We also prove relatively dense fractal-like sets of the plane can be generated if some further conditions are satisfied. %The main result of this study is Theorem \ref{main_theorem} which is proved in Section \ref{Section_proof_of_the_main_result}. %Space-filling curves generated through this theorem are demonstrated at the end of the paper, on which associated fractal-like sets are illustrated for some substitutions.
%\begin{figure}[H]
%\centering
%\begin{tikzpicture}[scale=0.85]
%\draw[line width=0.5pt, step=0.5cm] (0,0) grid (4,4);
%\draw[line width=0.5pt, step=1cm] (-6,0) grid (-2,4);
%\draw[line width=0.5pt, step=2cm] (-12,0) grid (-8,4);
%\draw (-12,0) -- (-12,4);
%\draw[line width=2pt, ->] (4.5,2) -- (5.5,2);
%\draw[line width=2pt, ->] (-1.5,2) -- (-0.5,2);
%\draw[line width=2pt, ->] (-7.5,2) -- (-6.5,2);
%\draw[line width=2pt, ->, -stealth] (-11,1) -- (-11,3) -- (-9,1) -- (-9,3);
%\end{tikzpicture}
%\caption{Caption}
%\label{fig:my_label}
%\end{figure}

The organisation of this paper is as follows. In Section \ref{Section_Background}, we provide the relevant preliminary definitions and an example of a space-filling curve constructed through a substitution in detail. In Section \ref{Section_proof_of_the_main_result}, we define a space-filling curve generator algorithm (Theorem \ref{main_theorem}) and present examples of space-filling curves formed by this algorithm applied over some of the known substitutions (or their variations). Lastly, 
%we present examples of space-filling curves which are formed by some of the known substitutions (or their variations). In addition, 
we explain and illustrate with an example how substitutions induce fractal-like sets that are relatively dense in $\RR^2$ in Section \ref{Section_fractal}. %A sufficient condition for generalizing this construction is also presented in this section.

\begin{figure}[ht]
\centering
\begin{subfigure}[b]{0.24\textwidth}
\centering
\tikz[remember picture]\node[inner sep=0pt,outer sep=0pt] (a){\includegraphics[width=\linewidth]{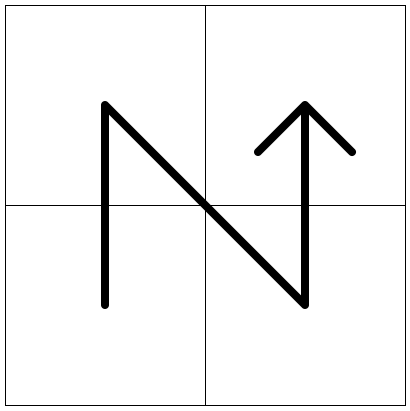}};
\end{subfigure}
\hfill
\begin{subfigure}[b]{0.24\textwidth}
\centering
\tikz[remember picture]\node[inner sep=0pt,outer sep=0pt] (b){\includegraphics[width=\linewidth]{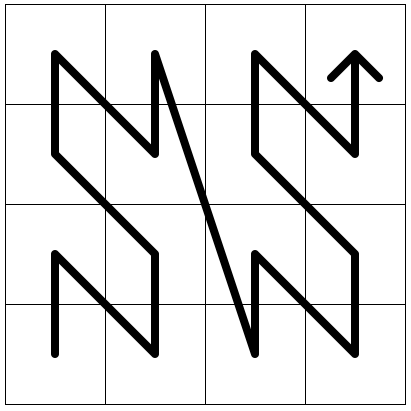}};
\end{subfigure}
\hfill
\begin{subfigure}[b]{0.24\textwidth}
\centering
\tikz[remember picture]\node[inner sep=0pt,outer sep=0pt] (c){\includegraphics[width=\linewidth]{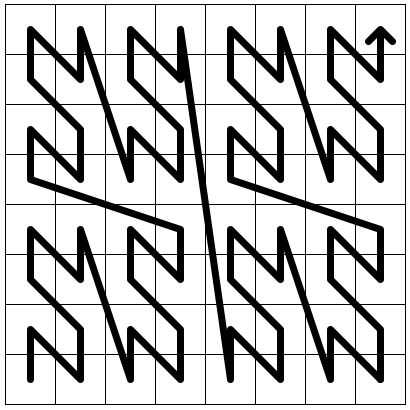}};
\end{subfigure}
\hfill
\tikz[remember picture,overlay]\draw[line width=2pt,-stealth] ([xshift=5pt]a.east) -- ([xshift=30pt]a.east)node[midway,above,text=black,font=\LARGE\bfseries\sffamily] {};
\tikz[remember picture,overlay]\draw[line width=2pt,-stealth] ([xshift=5pt]b.east) -- ([xshift=30pt]b.east)node[midway,above,text=black,font=\LARGE\bfseries\sffamily] {};
\tikz[remember picture,overlay]\draw[line width=2pt,-stealth] ([xshift=5pt]c.east) -- ([xshift=30pt]c.east)node[midway,above,text=black,font=\LARGE\bfseries\sffamily] {};
\caption{First tree iterations of the Morton order}
\label{Figure_Morton_Order_First_Three_Iterations}
\end{figure}
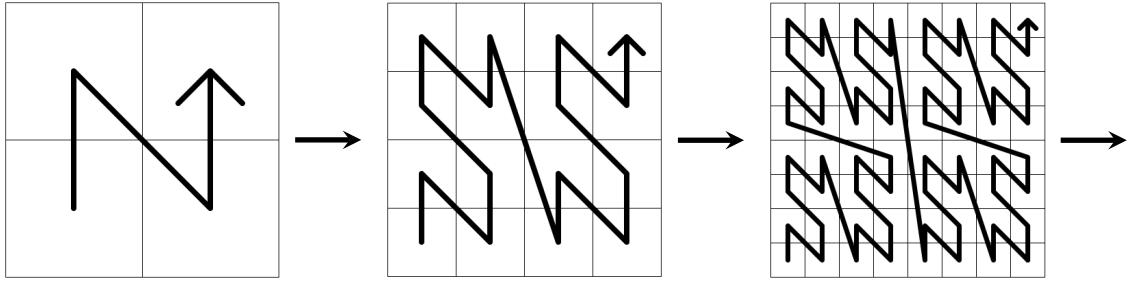

%The organisation of this paper is as follows. In Section \ref{Section_Background} we provide the relevant preliminary definitions, and an example of a space-filling curve constructed through a substitution. In Section \ref{Section_proof_of_the_main_result} we prove the main result, Theorem \ref{main_theorem}, and explain how the theorem can be regarded as an algorithm. Lastly, in Section \ref{Section_Examples} we present examples of space-filling curves which are formed by some of the known substitutions (or their variations), a few of which that satisfy a further condition are demonstrated along with their corresponding relatively dense fractal-like sets. 
\section{Methodology}\label{Section_Background}
In this section we explain how to generalize the geometric construction of Lebesgue's curve with an example in detail. 
\subsection{Substitutions}
Let $\RR^2$ denote Euclidean plane. We define the following:
\begin{enumerate}
    \item A \emph{tile} $t$ consists of a subset of $\RR^2$ that is homeomorphic to the closed unit disc, and a colour label $l(t)$ that distinguishes $t$ from any other identical sets. The associated subset of $t$ is denoted by $\sss t$. We call $\sss t$ the \emph{support} of $t$.
    \item A \emph{patch} $P$ is a finite collection of tiles so that
\begin{enumerate}
\item[(i)] $\bigcup\limits_{t\in P}\sss t$ is homeomorphic to the closed unit disc,
\item[(ii)] $\mathrm{int\ }(\sss t)\cap\mathrm{int\ }(\sss t^\prime)=\emptyset$ for each distinct $t,t^\prime\in P$. 
\end{enumerate}
The \emph{support\ of\ a\ patch} $P$ is the union of supports of its tiles. It is denoted by $\sss P$; i.e. $\sss P=\bigcup\limits_{t\in P}\sss t$. 
\item For a given tile $t$, a vector $x\in\RR^2$ and a non-zero scalar $\lambda\neq0$, we obtain new tiles $t+x$ and $\lambda t$ by the relations $\sss (t+x)=(\sss t) + x$, $l(t+x)=l(t)$, and $\sss (\lambda t)=\lambda\sss t$, $l(\lambda t)=l(t)$, respectively. We say $t+x$ is a \emph{translation} of $t$ and $\lambda t$ is a \emph{scaled\ copy} of $t$. Similarly, for a given patch $P$, a vector $x\in\RR^2$ and a non-zero scalar $\lambda\neq0$, we obtain new patches $P+x$ and $\lambda P$ by the relations $P+x=\{t+x:\ t\in P\}$ and $\{\lambda t:\ t\in P\}$, respectively. We say $P+x$ is a \emph{translation} of $P$ and $\lambda P$ is a \emph{scaled\ copy} of $P$.
\end{enumerate}

\begin{definition} 
Suppose $\cP$ is a given collection of tiles. Let $\cP^*$ denote the set of all patches consisting of tiles that are translations of tiles in $\cP$. A map $\omega:\cP\mapsto\cP^*$ is called a (\emph{planar}) \emph{substitution} if there exists $\lambda>1$ such that $\sss \omega(p)=\lambda\cdot\sss \omega(p)$ for all $p\in\cP$.

We call $\lambda$ the \emph{expansion\ factor} of $\omega$. We say that $\omega$ is a \emph{finite\ substitution} if, in addition, $\cP$ has a finite size.
\end{definition}
\begin{definition}
For a given substitution $\omega:\cP\mapsto\cP^*$, the patch $\omega^n(p)$ for $p\in\cP$ and $n\in\ZZ^+$ is called an \emph{n-supertile} of $\omega$. We say that every tile in $\cP$ is a $0$-supertile of $\omega$; i.e. $\omega^0(p):=p$ for $p\in\cP$.
\end{definition}

An example of a substitution is given in Figure \ref{Figure_Fibonacci_Substitution}. It is defined over the two intervals $[0,(1+\sqrt{5})/2]$ and $[0,1]$ with labels $a,b$, respectively. The interval $[0,(1+\sqrt{5})/2]$ with label $a$ is substituted into a patch consisting of two intervals $[0,(1+\sqrt{5})/2]$ and $[(1+\sqrt{5})/2, (3+\sqrt{5})/2]$ with labels $a,b$, respectively. The interval $[0,1]$ with label $b$ is substituted into a patch consisting of a single interval $[0,(1+\sqrt{5})/2]$ with label $a$. This substitution is called the \emph{Fibonacci\ substitution}. The expansion factor for the Fibonacci substitution is the golden mean $(1+\sqrt{5})/2$. By taking the Cartesian product of two Fibonacci substitutions, we get another substitution that is illustrated in Figure \ref{Figure_Fibonacci_times_Fibonacci_Substitution}. Throughout the document we denote the Fibonacci substitution by $\mu$, the Cartesian product of two Fibonacci substitutions by $\nu$ and their associated domains by $\cP_{\mu},\cP_{\nu}$, respectively. The domain $\cP_{\nu}$ consists of four\footnote{The tiles $p_b$ and $p_c$ are rotations of one another.} tiles $p_a,p_b,p_c,p_d$ such that $p_i$ is the tile with $l(p_i)=i$ for $i\in\{a,b,c,d\}$. 

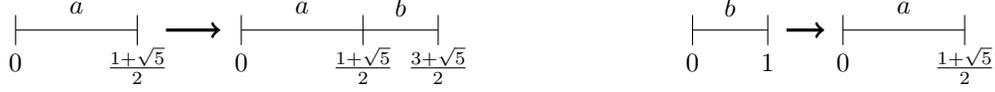
\begin{figure}[ht]
\centering
\begin{tikzpicture}
\draw (0,0) -- (1.6180339887498948482045868343656381177203091798057628621354486227,0);
\draw[very thick, ->] (2,0) -- (2.72,0);
\draw (3,0) -- (5.6180339887498948482045868343656381177203091798057628621354486227,0);
\draw (0,-0.2) -- (0,0.2);
\draw (1.6180339887498948482045868343656381177203091798057628621354486227,-0.2) -- (1.6180339887498948482045868343656381177203091798057628621354486227,0.2);
\draw (3,-0.2) -- (3,0.2);
\draw (5.6180339887498948482045868343656381177203091798057628621354486227,0.2) -- (5.6180339887498948482045868343656381177203091798057628621354486227,-0.2);
\draw (4.6180339887498948482045868343656381177203091798057628621354486227,0.2) -- (4.6180339887498948482045868343656381177203091798057628621354486227,-0.2);

\node[anchor=north] at (0,-0.2) {$0$};
\node[anchor=north] at (1.6180339887498948482045868343656381177203091798057628621354486227,-0.1) {$\frac{1+\sqrt{5}}{2}$};

\node[anchor=north] at (3,-0.2) {$0$};
\node[anchor=north] at (4.6180339887498948482045868343656381177203091798057628621354486227,-0.1) {$\frac{1+\sqrt{5}}{2}$};
\node[anchor=north] at (5.6180339887498948482045868343656381177203091798057628621354486227,-0.1) {$\frac{3+\sqrt{5}}{2}$};

\node at (0.8090169943749474241022934171828190588601545899028814310677243113,0.3) {$a$};
\node at (3.8090169943749474241022934171828190588601545899028814310677243113,0.3) {$a$};
\node at (0.5+4.6180339887498948482045868343656381177203091798057628621354486227,0.3) {$b$};

\draw (7+2,0) -- (8+2,0);
\draw[very thick, ->] (8.25+2,0) -- (8.75+2,0);
\draw (9+2,0) -- (10.6180339887498948482045868343656381177203091798057628621354486227+2,0);
\draw (9+2,-0.2) -- (9+2,0.2);
\draw (10.6180339887498948482045868343656381177203091798057628621354486227+2,0.2) --(10.6180339887498948482045868343656381177203091798057628621354486227+2,-0.2);
\draw (7+2,0.2) -- (7+2,-0.2);
\draw (8+2,0.2) -- (8+2,-0.2);

\node[anchor=north] at (7+2,-0.2) {$0$};
\node[anchor=north] at (8+2,-0.2) {$1$};
\node at (7.5+2,0.3) {$b$};
\node at (9.8090169943749474241022934171828190588601545899028814310677243113+2,0.3) {$a$};

\node[anchor=north] at (9+2,-0.2) {$0$};
\node[anchor=north] at (10.6180339887498948482045868343656381177203091798057628621354486227+2,-0.1) {$\frac{1+\sqrt{5}}{2}$};

\end{tikzpicture}    
\caption{The Fibonacci substitution $\mu$.}
\label{Figure_Fibonacci_Substitution}
\end{figure}

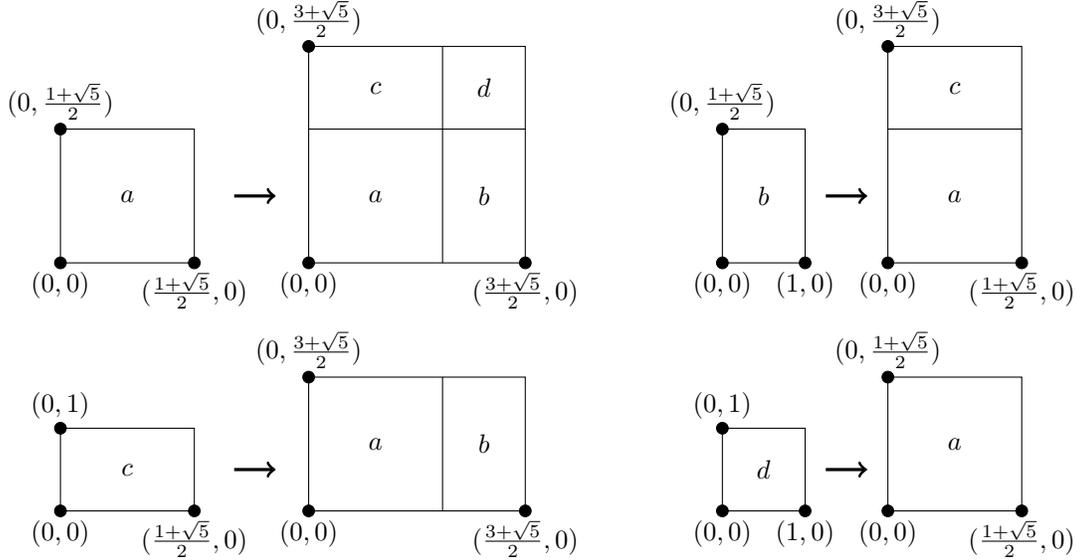
\begin{figure}[H]%[ht]
\centering
\begin{tikzpicture}[scale=1.1]
\draw (0,0) rectangle (1.6180339887498948482045868343656381177203091798057628621354486227,1.6180339887498948482045868343656381177203091798057628621354486227);
\draw[very thick, ->] (2.1,0.8090169943749474241022934171828190588601545899028814310677243113) -- (2.6,0.8090169943749474241022934171828190588601545899028814310677243113);
\draw (3,0) rectangle (5.6180339887498948482045868343656381177203091798057628621354486227,2.6180339887498948482045868343656381177203091798057628621354486227);
\draw (8,0) rectangle (9,1.6180339887498948482045868343656381177203091798057628621354486227);
\draw[very thick, ->] (9.25,0.8090169943749474241022934171828190588601545899028814310677243113) -- (9.75,0.8090169943749474241022934171828190588601545899028814310677243113);
\draw (10,0) rectangle (11.6180339887498948482045868343656381177203091798057628621354486227,2.6180339887498948482045868343656381177203091798057628621354486227);
\draw (10,1.6180339887498948482045868343656381177203091798057628621354486227) -- (11.6180339887498948482045868343656381177203091798057628621354486227,1.6180339887498948482045868343656381177203091798057628621354486227);
\draw (3,1.6180339887498948482045868343656381177203091798057628621354486227) -- (5.6180339887498948482045868343656381177203091798057628621354486227,1.6180339887498948482045868343656381177203091798057628621354486227);
\draw (4.6180339887498948482045868343656381177203091798057628621354486227,2.6180339887498948482045868343656381177203091798057628621354486227) -- (4.6180339887498948482045868343656381177203091798057628621354486227,0);
\draw (0,-3) rectangle (1.6180339887498948482045868343656381177203091798057628621354486227,1-3);
\draw[very thick, ->] (2.1,0.5-3) -- (2.6,0.5-3);
\draw (3,-3) rectangle (5.6180339887498948482045868343656381177203091798057628621354486227,1.6180339887498948482045868343656381177203091798057628621354486227-3);
\draw (4.6180339887498948482045868343656381177203091798057628621354486227,1.6180339887498948482045868343656381177203091798057628621354486227-3) -- (4.6180339887498948482045868343656381177203091798057628621354486227,1.6180339887498948482045868343656381177203091798057628621354486227-3-1.6180339887498948482045868343656381177203091798057628621354486227);
\draw (8,-3) rectangle (9,1-3);
\draw[very thick, ->] (9.25,0.5-3) -- (9.75,0.5-3);
\draw (10,0-3) rectangle (11.6180339887498948482045868343656381177203091798057628621354486227,1.6180339887498948482045868343656381177203091798057628621354486227-3);
\node at (0.8090169943749474241022934171828190588601545899028814310677243113,0.8090169943749474241022934171828190588601545899028814310677243113) {$a$};
\node at (3.8090169943749474241022934171828190588601545899028814310677243113,0.8090169943749474241022934171828190588601545899028814310677243113) {$a$};
\node at (3.8090169943749474241022934171828190588601545899028814310677243113,0.8090169943749474241022934171828190588601545899028814310677243113-3) {$a$};
\node at (3.8090169943749474241022934171828190588601545899028814310677243113+1.3090169943749474241022934171828190588601545899028814310677243113,0.8090169943749474241022934171828190588601545899028814310677243113-3) {$b$};
\node at (3.8090169943749474241022934171828190588601545899028814310677243113+1.3090169943749474241022934171828190588601545899028814310677243113,0.8090169943749474241022934171828190588601545899028814310677243113) {$b$};
\node at (3.8090169943749474241022934171828190588601545899028814310677243113+1.3090169943749474241022934171828190588601545899028814310677243113,0.8090169943749474241022934171828190588601545899028814310677243113+1.3090169943749474241022934171828190588601545899028814310677243113) {$d$};
\node at (3.8090169943749474241022934171828190588601545899028814310677243113,0.8090169943749474241022934171828190588601545899028814310677243113+1.3090169943749474241022934171828190588601545899028814310677243113) {$c$};
\node at (10.8090169943749474241022934171828190588601545899028814310677243113,0.8090169943749474241022934171828190588601545899028814310677243113+1.3090169943749474241022934171828190588601545899028814310677243113) {$c$};
\node at (8.5,0.8090169943749474241022934171828190588601545899028814310677243113) {$b$};
\node at (8.5,0.5-3) {$d$};
\node at (0.8090169943749474241022934171828190588601545899028814310677243113,0.5-3) {$c$};
\node at (10.8090169943749474241022934171828190588601545899028814310677243113,0.8090169943749474241022934171828190588601545899028814310677243113-3) {$a$};
\node at (10.8090169943749474241022934171828190588601545899028814310677243113,0.8090169943749474241022934171828190588601545899028814310677243113) {$a$};
\filldraw[black] (0,0) circle (2pt) node[anchor=north] {$(0,0)$};
\filldraw[black] (3,0) circle (2pt) node[anchor=north] {$(0,0)$} ; 
\filldraw[black] (0,1.6180339887498948482045868343656381177203091798057628621354486227) circle (2pt) node[anchor=south] {$(0,\frac{1+\sqrt{5}}{2})$}; 
\filldraw[black] (3,2.6180339887498948482045868343656381177203091798057628621354486227) circle (2pt) node[anchor=south] {$(0,\frac{3+\sqrt{5}}{2})$}; 
\filldraw[black] (0,-3) circle (2pt) node[anchor=north] {$(0,0)$}; 
\filldraw[black] (3,-3) circle (2pt) node[anchor=north] {$(0,0)$}; 
\filldraw[black] (0,1-3) circle (2pt) node[anchor=south] {$(0,1)$};
\filldraw[black] (3,1.6180339887498948482045868343656381177203091798057628621354486227-3) circle (2pt) node[anchor=south] {$(0,\frac{3+\sqrt{5}}{2})$}; 

\filldraw[black] (3+2.6180339887498948482045868343656381177203091798057628621354486227,0-3) circle (2pt) node[anchor=north] {$(\frac{3+\sqrt{5}}{2},0)$};
\filldraw[black] (1.6180339887498948482045868343656381177203091798057628621354486227,0-3) circle (2pt) node[anchor=north] {$(\frac{1+\sqrt{5}}{2},0)$}; 
\filldraw[black] (1.6180339887498948482045868343656381177203091798057628621354486227,0) circle (2pt) node[anchor=north] {$(\frac{1+\sqrt{5}}{2},0)$};
\filldraw[black] (5.6180339887498948482045868343656381177203091798057628621354486227,0) circle (2pt) node[anchor=north] {$(\frac{3+\sqrt{5}}{2},0)$};
\filldraw[black] (8,0) circle (2pt) node[anchor=north] {$(0,0)$}; 
\filldraw[black] (9,0) circle (2pt) node[anchor=north] {$(1,0)$}; 
\filldraw[black] (8,0-3) circle (2pt) node[anchor=north] {$(0,0)$}; 
\filldraw[black] (9,0-3) circle (2pt) node[anchor=north] {$(1,0)$}; 
\filldraw[black] (8,1.6180339887498948482045868343656381177203091798057628621354486227) circle (2pt) node[anchor=south] {$(0,\frac{1+\sqrt{5}}{2})$}; 
\filldraw[black] (10,0) circle (2pt) node[anchor=north] {$(0,0)$}; 
\filldraw[black]  (10,-3) circle (2pt) node[anchor=north] {$(0,0)$}; 
\filldraw[black]  (10,2.6180339887498948482045868343656381177203091798057628621354486227) circle (2pt) node[anchor=south] {$(0,\frac{3+\sqrt{5}}{2})$}; 
\filldraw[black] (10,1.6180339887498948482045868343656381177203091798057628621354486227-3) circle (2pt) node[anchor=south]  {$(0,\frac{1+\sqrt{5}}{2})$}; 
\filldraw[black] (10+1.6180339887498948482045868343656381177203091798057628621354486227,0) circle (2pt) node[anchor=north]  {$(\frac{1+\sqrt{5}}{2},0)$}; 
\filldraw[black] (10+1.6180339887498948482045868343656381177203091798057628621354486227,0-3) circle (2pt) node[anchor=north]  {$(\frac{1+\sqrt{5}}{2},0)$}; 
\filldraw[black] (8,0-3+1) circle (2pt) node[anchor=south]  {$(0,1)$}; 
\end{tikzpicture}    
\caption{The substitution $\nu$.}
\label{Figure_Fibonacci_times_Fibonacci_Substitution}
\end{figure}
\begin{remark}[Powers of substitutions]
Let $\omega:\cP\mapsto\cP^*$ be a given substitution with an expansion factor $\lambda$. The range of a substitution $\omega$ does not necessarily contain its domain. That is,  $\omega^2$ may be not well-defined. Therefore, we define an extension $\omega^\prime:\cP+\RR^2\mapsto\cP^*$ by $\omega^\prime(p+x)=\omega(p)+\lambda\cdot x$ for $p\in\cP$ and $x\in\RR^2$, and define $(\omega^\prime)^n(p)$ for $p\in\cP$ and $n\geq 2$ recursively as follows:
\[
{\displaystyle (\omega^\prime)^n(p):=\bigcup\limits_{t\in(\omega^\prime)^{n-1}(p)}\omega^\prime(t).}
\]
The powers $(\omega^\prime)^n$ for $n\in\ZZ^+$ are well-defined. We use the maps $\omega$ and $\omega^\prime$ interchangeably, when there is no confusion of taking powers of the substitutions.
\end{remark}
\subsection{Total Order Structures by $\nu$}\label{Section_Total_Order}
A fundamental ingredient of Sagan's geometric construction of the Lebesgue curve is the approximating polygons \cite{Journal_Sagan_1, Journal_Sagan_2}. For that we define total orders over the supertiles of $\nu$. These total order structures are the key feature to form such approximating polygons.

Assume that $\nu(p_a)=\{a_1,a_2,a_3,a_4\}$, $\nu(p_b)=\{b_1,b_2\}$, $\nu(p_c)=\{c_1,c_2\}$ and $\nu(p_d)=\{d_1\}$, as shown in Figure \ref{Figure_1_supertiles_of_Fibonacci_times_Fibonacci_with_tile_names}. The relations given in \eqref{a} - \eqref{d} indicate a total order $\underset{i,1}{\lesssim}$ on the collection $\nu(p_i)$ for each $i\in\{a,b,c,d\}$, respectively.
\begin{align}
a_j\underset{a,1}{\lesssim}a_k & \  \mathrm{if\ } j\leq k \mathrm{\ for\ each\ } j,k\in\{1,2,3,4\}.\label{a} \\ 
b_j\underset{b,1}{\lesssim}b_k &\  \mathrm{if\ } j\leq k \mathrm{\ for\ each\ } j,k\in\{1,2\}. \label{b}\\ 
c_j\underset{c,1}{\lesssim}c_k &\ \mathrm{if\ } j\leq k \mathrm{\ for\ each\ } j,k\in\{1,2\}. \label{c} \\ 
d_j\underset{d,1}{\lesssim}d_k &\ \mathrm{if\ } j\leq k \mathrm{\ for\ each\ } j,k\in\{1\}. \label{d}
\end{align}
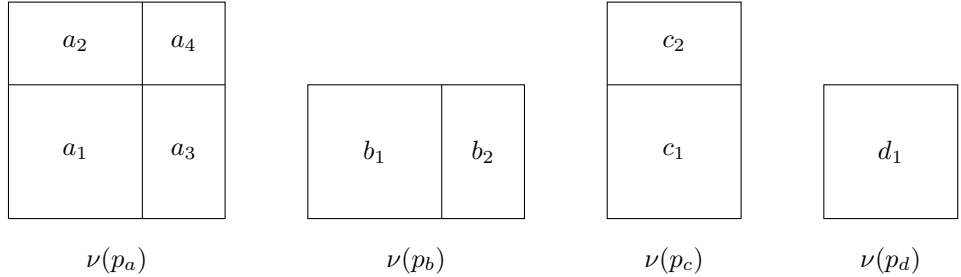
\begin{figure}[ht]
\centering
\begin{tikzpicture}[scale=1.1]
\foreach \index in {0}{
\draw (0,0+\index) rectangle (2.6180339887498948482045868343656381177203091798057628621354486227,2.6180339887498948482045868343656381177203091798057628621354486227+\index);
\draw (1.6180339887498948482045868343656381177203091798057628621354486227,0+\index) -- (1.6180339887498948482045868343656381177203091798057628621354486227,2.6180339887498948482045868343656381177203091798057628621354486227+\index);
\draw (0,1.6180339887498948482045868343656381177203091798057628621354486227+\index) -- (2.6180339887498948482045868343656381177203091798057628621354486227,1.6180339887498948482045868343656381177203091798057628621354486227+\index);

\draw (3.6180339887498948482045868343656381177203091798057628621354486227,0+\index) rectangle (3.6180339887498948482045868343656381177203091798057628621354486227+2.6180339887498948482045868343656381177203091798057628621354486227,1.6180339887498948482045868343656381177203091798057628621354486227+\index);
\draw (3.6180339887498948482045868343656381177203091798057628621354486227+1.6180339887498948482045868343656381177203091798057628621354486227,0+\index) -- (3.6180339887498948482045868343656381177203091798057628621354486227+1.6180339887498948482045868343656381177203091798057628621354486227,1.6180339887498948482045868343656381177203091798057628621354486227+\index);

\draw (3.6180339887498948482045868343656381177203091798057628621354486227+3.6180339887498948482045868343656381177203091798057628621354486227,0+\index) rectangle (3.6180339887498948482045868343656381177203091798057628621354486227+3.6180339887498948482045868343656381177203091798057628621354486227+1.6180339887498948482045868343656381177203091798057628621354486227,2.6180339887498948482045868343656381177203091798057628621354486227+\index);
\draw (3.6180339887498948482045868343656381177203091798057628621354486227+3.6180339887498948482045868343656381177203091798057628621354486227,1.6180339887498948482045868343656381177203091798057628621354486227+\index) -- (3.6180339887498948482045868343656381177203091798057628621354486227+3.6180339887498948482045868343656381177203091798057628621354486227+1.6180339887498948482045868343656381177203091798057628621354486227,1.6180339887498948482045868343656381177203091798057628621354486227+\index);

\draw (3.6180339887498948482045868343656381177203091798057628621354486227+3.6180339887498948482045868343656381177203091798057628621354486227+2.6180339887498948482045868343656381177203091798057628621354486227,0+\index) rectangle (3.6180339887498948482045868343656381177203091798057628621354486227+3.6180339887498948482045868343656381177203091798057628621354486227+2.6180339887498948482045868343656381177203091798057628621354486227+1.6180339887498948482045868343656381177203091798057628621354486227,1.6180339887498948482045868343656381177203091798057628621354486227+\index);
}
\node at (0.8090169943749474241022934171828190588601545899028814310677243113,0.8090169943749474241022934171828190588601545899028814310677243113) {$a_1$};
\node at (0.8090169943749474241022934171828190588601545899028814310677243113,0.5+1.6180339887498948482045868343656381177203091798057628621354486227) {$a_2$};
\node at (0.5+1.6180339887498948482045868343656381177203091798057628621354486227,0.8090169943749474241022934171828190588601545899028814310677243113) {$a_3$};
\node at (0.5+1.6180339887498948482045868343656381177203091798057628621354486227,0.5+1.6180339887498948482045868343656381177203091798057628621354486227) {$a_4$};

\node at (0.8090169943749474241022934171828190588601545899028814310677243113+0.5,-0.5) {$\nu(p_a)$};

\node at (0.8090169943749474241022934171828190588601545899028814310677243113+3.6180339887498948482045868343656381177203091798057628621354486227,0.8090169943749474241022934171828190588601545899028814310677243113) {$b_1$};

\node at (0.5+1.6180339887498948482045868343656381177203091798057628621354486227+3.6180339887498948482045868343656381177203091798057628621354486227,0.8090169943749474241022934171828190588601545899028814310677243113) {$b_2$};

\node at (0.8090169943749474241022934171828190588601545899028814310677243113+0.5+3.6180339887498948482045868343656381177203091798057628621354486227,-0.5) {$\nu(p_b)$};

\node at (3.6180339887498948482045868343656381177203091798057628621354486227+3.6180339887498948482045868343656381177203091798057628621354486227+0.8090169943749474241022934171828190588601545899028814310677243113, 0.8090169943749474241022934171828190588601545899028814310677243113) {$c_1$};

\node at (3.6180339887498948482045868343656381177203091798057628621354486227+3.6180339887498948482045868343656381177203091798057628621354486227+0.8090169943749474241022934171828190588601545899028814310677243113, 0.5+1.6180339887498948482045868343656381177203091798057628621354486227) {$c_2$};

\node at (3.6180339887498948482045868343656381177203091798057628621354486227+3.6180339887498948482045868343656381177203091798057628621354486227+0.8090169943749474241022934171828190588601545899028814310677243113,-0.5) {$\nu(p_c)$};

\node at (3.6180339887498948482045868343656381177203091798057628621354486227+3.6180339887498948482045868343656381177203091798057628621354486227+2.6180339887498948482045868343656381177203091798057628621354486227+0.8090169943749474241022934171828190588601545899028814310677243113,0.8090169943749474241022934171828190588601545899028814310677243113) {$d_1$};

\node at (3.6180339887498948482045868343656381177203091798057628621354486227+3.6180339887498948482045868343656381177203091798057628621354486227+2.6180339887498948482045868343656381177203091798057628621354486227+0.8090169943749474241022934171828190588601545899028814310677243113,-0.5) {$\nu(p_d)$};
\end{tikzpicture}
\caption{The patches $\nu(p_a)$, $\nu(p_b)$, $\nu(p_c)$ and $\nu(p_d)$, from left to right.} 
\label{Figure_1_supertiles_of_Fibonacci_times_Fibonacci_with_tile_names}
\end{figure}

Note that $\nu^{m}(p)=\bigcup\limits_{t\in\nu(p)}\omega^{m-1}(t)$ for each $m>1$ and $p\in\cP_{\nu}$. Therefore, the relations in \eqref{a} - \eqref{d} induce total orders for supertiles of $\nu$ inductively. More precisely, suppose $i\in\{a,b,c,d\}$ and $m>1$ are fixed. Suppose further $\underset{i,m-1}{\lesssim}$ is a total order on the collection $\nu^{m-1}(p_i)$. Define $\underset{i,m}{\lesssim}$ over the patch $\nu^{m}(p_i)$ such that:
\begin{enumerate}
    \item[(i)]  If $x,y\in\nu^{m-1}(t)$ for some $t\in \nu(p_i)$, then $x\underset{i,m}{\lesssim}y$ whenever $x\underset{i,m-1}{\lesssim}y$.
    \item[(ii)]  If $x\in\nu^{m-1}(t_1)$ and $y\in\nu^{m-1}(t_2)$ for some distinct $t_1,t_2\in \nu(p_i)$, then $x\underset{i,m}{\lesssim}y$ whenever $t_1\underset{i,1}{\lesssim}t_2$.
\end{enumerate}
The total orders between the tiles of 1-supertiles and 2-supertiles of $\nu$ are shown in Figure \ref{Figure_Total_Order_in_tiles_in_the_1_and_2_supertiles_of_W_FxF}. The associated order structures are elucidated by the numbers attached to the tiles in the figure.
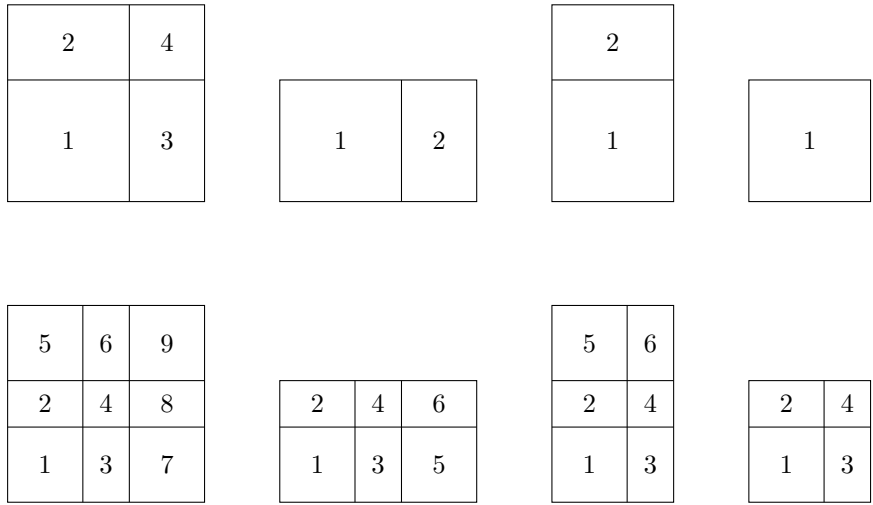
\begin{figure}[H]
\centering
\begin{tikzpicture}[scale=1]%[scale=1.3]
\foreach \index in {0, -4}{
\draw (0,0+\index) rectangle (2.6180339887498948482045868343656381177203091798057628621354486227,2.6180339887498948482045868343656381177203091798057628621354486227+\index);
\draw (1.6180339887498948482045868343656381177203091798057628621354486227,0+\index) -- (1.6180339887498948482045868343656381177203091798057628621354486227,2.6180339887498948482045868343656381177203091798057628621354486227+\index);
\draw (0,1.6180339887498948482045868343656381177203091798057628621354486227+\index) -- (2.6180339887498948482045868343656381177203091798057628621354486227,1.6180339887498948482045868343656381177203091798057628621354486227+\index);

\draw (3.6180339887498948482045868343656381177203091798057628621354486227,0+\index) rectangle (3.6180339887498948482045868343656381177203091798057628621354486227+2.6180339887498948482045868343656381177203091798057628621354486227,1.6180339887498948482045868343656381177203091798057628621354486227+\index);
\draw (3.6180339887498948482045868343656381177203091798057628621354486227+1.6180339887498948482045868343656381177203091798057628621354486227,0+\index) -- (3.6180339887498948482045868343656381177203091798057628621354486227+1.6180339887498948482045868343656381177203091798057628621354486227,1.6180339887498948482045868343656381177203091798057628621354486227+\index);

\draw (3.6180339887498948482045868343656381177203091798057628621354486227+3.6180339887498948482045868343656381177203091798057628621354486227,0+\index) rectangle (3.6180339887498948482045868343656381177203091798057628621354486227+3.6180339887498948482045868343656381177203091798057628621354486227+1.6180339887498948482045868343656381177203091798057628621354486227,2.6180339887498948482045868343656381177203091798057628621354486227+\index);
\draw (3.6180339887498948482045868343656381177203091798057628621354486227+3.6180339887498948482045868343656381177203091798057628621354486227,1.6180339887498948482045868343656381177203091798057628621354486227+\index) -- (3.6180339887498948482045868343656381177203091798057628621354486227+3.6180339887498948482045868343656381177203091798057628621354486227+1.6180339887498948482045868343656381177203091798057628621354486227,1.6180339887498948482045868343656381177203091798057628621354486227+\index);

\draw (3.6180339887498948482045868343656381177203091798057628621354486227+3.6180339887498948482045868343656381177203091798057628621354486227+2.6180339887498948482045868343656381177203091798057628621354486227,0+\index) rectangle (3.6180339887498948482045868343656381177203091798057628621354486227+3.6180339887498948482045868343656381177203091798057628621354486227+2.6180339887498948482045868343656381177203091798057628621354486227+1.6180339887498948482045868343656381177203091798057628621354486227,1.6180339887498948482045868343656381177203091798057628621354486227+\index);
}

\draw (1,-4) -- (1,-4+2.6180339887498948482045868343656381177203091798057628621354486227);
\draw (0,-4+1) -- (0+2.6180339887498948482045868343656381177203091798057628621354486227,-4+1);
\node at (0+0.8090169943749474241022934171828190588601545899028814310677243113,0+0.8090169943749474241022934171828190588601545899028814310677243113) {$1$};
\node at (0.5+1.6180339887498948482045868343656381177203091798057628621354486227,0+0.8090169943749474241022934171828190588601545899028814310677243113) {$3$};
\node at (0+0.8090169943749474241022934171828190588601545899028814310677243113,0.5+1.6180339887498948482045868343656381177203091798057628621354486227) {$2$};
\node at (0.5+1.6180339887498948482045868343656381177203091798057628621354486227,0.5+1.6180339887498948482045868343656381177203091798057628621354486227) {$4$};
\node at (0+0.5,-4+0.5) {$1$};
\node at (0+1+0.3090169943749474241022934171828190588601545899028814310677243113,-4+0.5) {$3$};
\node at (0+0.5,-4+1+0.3090169943749474241022934171828190588601545899028814310677243113) {$2$};
\node at (0+1+0.3090169943749474241022934171828190588601545899028814310677243113,-4+1+0.3090169943749474241022934171828190588601545899028814310677243113) {$4$};
\node at (0+0.5,-4+1.6180339887498948482045868343656381177203091798057628621354486227+0.5) {$5$};
\node at (0+1+0.3090169943749474241022934171828190588601545899028814310677243113,-4+1.6180339887498948482045868343656381177203091798057628621354486227+0.5) {$6$};
\node at (0+1.6180339887498948482045868343656381177203091798057628621354486227+0.5,-4+0.5) {$7$};
\node at (0+1.6180339887498948482045868343656381177203091798057628621354486227+0.5,-4+1+0.3090169943749474241022934171828190588601545899028814310677243113) {$8$};
\node at (0+1.6180339887498948482045868343656381177203091798057628621354486227+0.5,-4+1.6180339887498948482045868343656381177203091798057628621354486227+0.5) {$9$};
\draw (3.6180339887498948482045868343656381177203091798057628621354486227,0+1-4) -- (3.6180339887498948482045868343656381177203091798057628621354486227+2.6180339887498948482045868343656381177203091798057628621354486227,0+1-4); 
\draw (3.6180339887498948482045868343656381177203091798057628621354486227+1,0-4) -- (3.6180339887498948482045868343656381177203091798057628621354486227+1,0-4+1.6180339887498948482045868343656381177203091798057628621354486227); 
\node at (3.6180339887498948482045868343656381177203091798057628621354486227+0.8090169943749474241022934171828190588601545899028814310677243113,0+0.8090169943749474241022934171828190588601545899028814310677243113) {$1$};
\node at (3.6180339887498948482045868343656381177203091798057628621354486227+1.6180339887498948482045868343656381177203091798057628621354486227+0.5,0+0.8090169943749474241022934171828190588601545899028814310677243113) {$2$};
\node at (3.6180339887498948482045868343656381177203091798057628621354486227+0.5,-4+0.5) {$1$};
\node at (3.6180339887498948482045868343656381177203091798057628621354486227+0.5,-4+1+0.3090169943749474241022934171828190588601545899028814310677243113) {$2$};
\node at (3.6180339887498948482045868343656381177203091798057628621354486227+1+0.3090169943749474241022934171828190588601545899028814310677243113,-4+0.5) {$3$};
\node at (3.6180339887498948482045868343656381177203091798057628621354486227+1+0.3090169943749474241022934171828190588601545899028814310677243113,-4+1+0.3090169943749474241022934171828190588601545899028814310677243113) {$4$};
\node at (3.6180339887498948482045868343656381177203091798057628621354486227+1.6180339887498948482045868343656381177203091798057628621354486227+0.5,-4+0.5) {$5$};
\node at (3.6180339887498948482045868343656381177203091798057628621354486227+1.6180339887498948482045868343656381177203091798057628621354486227+0.5,-4+1+0.3090169943749474241022934171828190588601545899028814310677243113) {$6$};
\draw (3.6180339887498948482045868343656381177203091798057628621354486227+3.6180339887498948482045868343656381177203091798057628621354486227,-4+1)--(3.6180339887498948482045868343656381177203091798057628621354486227+3.6180339887498948482045868343656381177203091798057628621354486227+1.6180339887498948482045868343656381177203091798057628621354486227,-4+1);
\draw (3.6180339887498948482045868343656381177203091798057628621354486227+3.6180339887498948482045868343656381177203091798057628621354486227+1,-4)--(3.6180339887498948482045868343656381177203091798057628621354486227+3.6180339887498948482045868343656381177203091798057628621354486227+1,-4+2.6180339887498948482045868343656381177203091798057628621354486227);
\node at (3.6180339887498948482045868343656381177203091798057628621354486227+3.6180339887498948482045868343656381177203091798057628621354486227+0.8090169943749474241022934171828190588601545899028814310677243113,0+0.8090169943749474241022934171828190588601545899028814310677243113) {$1$};
\node at (3.6180339887498948482045868343656381177203091798057628621354486227+3.6180339887498948482045868343656381177203091798057628621354486227+0.8090169943749474241022934171828190588601545899028814310677243113,0+1.6180339887498948482045868343656381177203091798057628621354486227+0.5) {$2$};
\node at (3.6180339887498948482045868343656381177203091798057628621354486227+3.6180339887498948482045868343656381177203091798057628621354486227+0.5,-4+0.5) {$1$};
\node at (3.6180339887498948482045868343656381177203091798057628621354486227+3.6180339887498948482045868343656381177203091798057628621354486227+0.5,-4+1+0.3090169943749474241022934171828190588601545899028814310677243113) {$2$};
\node at (3.6180339887498948482045868343656381177203091798057628621354486227+3.6180339887498948482045868343656381177203091798057628621354486227+1+0.3090169943749474241022934171828190588601545899028814310677243113,-4+0.5) {$3$};
\node at (3.6180339887498948482045868343656381177203091798057628621354486227+3.6180339887498948482045868343656381177203091798057628621354486227+1+0.3090169943749474241022934171828190588601545899028814310677243113,-4+1+0.3090169943749474241022934171828190588601545899028814310677243113) {$4$};
\node at (3.6180339887498948482045868343656381177203091798057628621354486227+3.6180339887498948482045868343656381177203091798057628621354486227+0.5,-4+1.6180339887498948482045868343656381177203091798057628621354486227+0.5) {$5$};
\node at (3.6180339887498948482045868343656381177203091798057628621354486227+3.6180339887498948482045868343656381177203091798057628621354486227+1+0.3090169943749474241022934171828190588601545899028814310677243113,-4+1.6180339887498948482045868343656381177203091798057628621354486227+0.5) {$6$};

\draw (3.6180339887498948482045868343656381177203091798057628621354486227+3.6180339887498948482045868343656381177203091798057628621354486227+2.6180339887498948482045868343656381177203091798057628621354486227+1,-4) -- (3.6180339887498948482045868343656381177203091798057628621354486227+3.6180339887498948482045868343656381177203091798057628621354486227+2.6180339887498948482045868343656381177203091798057628621354486227+1,-4+1.6180339887498948482045868343656381177203091798057628621354486227);
\draw (3.6180339887498948482045868343656381177203091798057628621354486227+3.6180339887498948482045868343656381177203091798057628621354486227+2.6180339887498948482045868343656381177203091798057628621354486227,-4+1) -- (3.6180339887498948482045868343656381177203091798057628621354486227+3.6180339887498948482045868343656381177203091798057628621354486227+2.6180339887498948482045868343656381177203091798057628621354486227+1.6180339887498948482045868343656381177203091798057628621354486227,-4+1);
\node at (3.6180339887498948482045868343656381177203091798057628621354486227+3.6180339887498948482045868343656381177203091798057628621354486227+2.6180339887498948482045868343656381177203091798057628621354486227+0.8090169943749474241022934171828190588601545899028814310677243113,0+0.8090169943749474241022934171828190588601545899028814310677243113) {$1$};
\node at (3.6180339887498948482045868343656381177203091798057628621354486227+3.6180339887498948482045868343656381177203091798057628621354486227+2.6180339887498948482045868343656381177203091798057628621354486227+0.5,-4+0.5) {$1$};
\node at (3.6180339887498948482045868343656381177203091798057628621354486227+3.6180339887498948482045868343656381177203091798057628621354486227+2.6180339887498948482045868343656381177203091798057628621354486227+0.5,-4+1+0.3090169943749474241022934171828190588601545899028814310677243113) {$2$};
\node at (3.6180339887498948482045868343656381177203091798057628621354486227+3.6180339887498948482045868343656381177203091798057628621354486227+2.6180339887498948482045868343656381177203091798057628621354486227+1+0.3090169943749474241022934171828190588601545899028814310677243113,-4+0.5) {$3$};
\node at (3.6180339887498948482045868343656381177203091798057628621354486227+3.6180339887498948482045868343656381177203091798057628621354486227+2.6180339887498948482045868343656381177203091798057628621354486227+1+0.3090169943749474241022934171828190588601545899028814310677243113,-4+1+0.3090169943749474241022934171828190588601545899028814310677243113) {$4$};
%(3.6180339887498948482045868343656381177203091798057628621354486227+3.6180339887498948482045868343656381177203091798057628621354486227+2.6180339887498948482045868343656381177203091798057628621354486227,0+\index)

\end{tikzpicture}
\caption{The total orders between the tiles of 1-supertiles and 2-supertiles of $\nu$.}
\label{Figure_Total_Order_in_tiles_in_the_1_and_2_supertiles_of_W_FxF}
\end{figure}
\subsection{Partitions of the Unit Square by $\nu$}
\label{Section_Fibonacci_SFC_Subsection_Partitions_of_the_Unit_Square}
Consider the sequence $\{\lambda^{-k-1}\cdot\nu^k(p_a):\ k=0,1,2,\dots\}$ of partitions of the unit square, with the convention $\nu^0(p_a):=p_a$. Observe that $\lambda^{-k-1}\cdot\left(\sss\nu^k(p_a)\right)=[0,1]\times[0,1]$ for every $k\in\NN$. In particular, each partition is a scaled copy of a supertile of $\nu$. Thus, the total orders introduced in Section \ref{Section_Total_Order} can be transferred over the rectangle tiles within the partitions. We use these total orders in order to label the tiles appearing in the partitions. 

The partitions $\lambda^{-1}\cdot p_a$, $\lambda^{-2}\cdot\nu(p_a)$ and $\lambda^{-3}\cdot\nu^2(p_a)$ are demonstrated in Figure \ref{Figure_Partitions_of_the_Unit_Sqaure_First_Second_Iterations}. For every $k\in\NN$, $\lambda^{-k-1}\cdot\nu^k(p_a)$, consists of $\cF_{k+2}^2$ rectangle tiles, where $\cF_{k+2}$ is the $(k+2)$-th Fibonacci number. 
Denote these tiles by $\cJ_{k}^{1}, \dots, \cJ_{k}^{\cF_{k+2}^2}$ such that
\begin{enumerate}
    \item[(i)] $\bigcup\limits_{i=1}^{\cF_{k+2}^2}\lambda^{k+1}\cdot \cJ_{k}^{i}=\nu^k(p_a)$,
    \item[(ii)] $\lambda^{k+1}\cdot \cJ_{k}^{i}\underset{a,k}{\lesssim}\lambda^{k+1}\cdot\cJ_{k}^{j}$ if and only if $i\leq j$, for every $i,j\in\{1,\dots,\cF_{k+2}^2\}$. 
\end{enumerate}
The tiles $\cJ_{0}^1$, $\cJ_{1}^1, \dots, \cJ_{1}^4$ and $\cJ_{2}^1, \dots, \cJ_{2}^9$ are demonstrated in Figure \ref{Figure_Labelling_the_Rectangles_in_the_partititons_of_the_unit_square}. Notice that the tiles are labelled with respect to the total order structure presented on the leftmost side of Figure \ref{Figure_Total_Order_in_tiles_in_the_1_and_2_supertiles_of_W_FxF}.

\begin{figure}[H]
\centering
\begin{tikzpicture}[scale=0.5]
\draw(0-10,0) rectangle (6-10,6);
\node[anchor=north] at (0-10,0) {$(0,0)$};
\node[anchor=north] at (6-10,0) {$(1,0)$};
\node[anchor=south] at (0-10,6) {$(0,1)$};
\node[anchor=south] at (6-10,6) {$(1,1)$};
\foreach \index in {0,10}{
\draw (0+\index,0) rectangle (6+\index,6);
\draw (\index+6/1.6180339887498948482045868343656381177203091798057628621354486227,0) -- (\index+6/1.6180339887498948482045868343656381177203091798057628621354486227,6);
\draw (\index,6/1.6180339887498948482045868343656381177203091798057628621354486227) -- (\index+6,6/1.6180339887498948482045868343656381177203091798057628621354486227);
\node[anchor=north] at (0+\index,0) {$(0,0)$};
\node[anchor=north] at (6+\index,0) {$(1,0)$};
\node[anchor=south] at (0+\index,6) {$(0,1)$};
\node[anchor=south] at (6+\index,6) {$(1,1)$};
}
\draw (10+6/2.6180339887498948482045868343656381177203091798057628621354486227,0) -- (10+6/2.6180339887498948482045868343656381177203091798057628621354486227,6);
\draw (10,6/2.6180339887498948482045868343656381177203091798057628621354486227) -- (10+6,6/2.6180339887498948482045868343656381177203091798057628621354486227);
\end{tikzpicture}
\caption{$\lambda^{-1}\cdot p_a$, $\lambda^{-2}\cdot \nu(p_a)$, $\lambda^{-3}\cdot \nu^2(p_a)$, from left to right.}
\label{Figure_Partitions_of_the_Unit_Sqaure_First_Second_Iterations}
\end{figure}
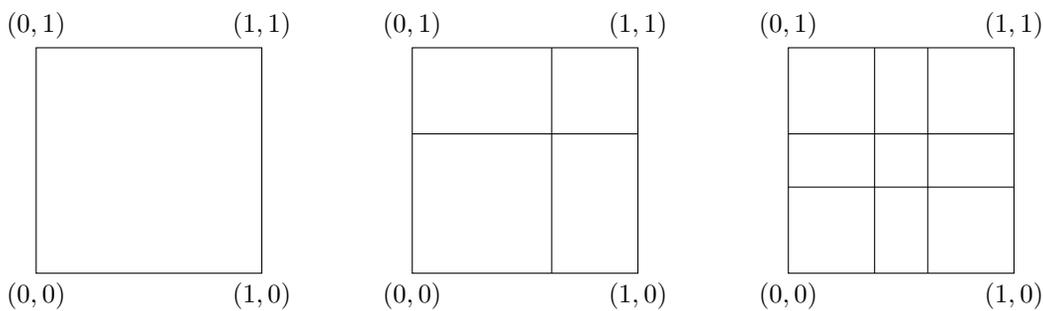
\begin{figure}[H]
\centering
\begin{tikzpicture}[scale=0.5]

\draw(0-10,0) rectangle (6-10,6);
\node[scale=1.5] at (3-10,3) {$\cJ_{0}^1$};
\node[anchor=north] at (0-10,0) {$(0,0)$};
\node[anchor=north] at (6-10,0) {$(1,0)$};
\node[anchor=south] at (0-10,6) {$(0,1)$};
\node[anchor=south] at (6-10,6) {$(1,1)$};

\foreach \index in {0,10}{
\draw (0+\index,0) rectangle (6+\index,6);
\draw (\index+6/1.6180339887498948482045868343656381177203091798057628621354486227,0) -- (\index+6/1.6180339887498948482045868343656381177203091798057628621354486227,6);
\draw (\index,6/1.6180339887498948482045868343656381177203091798057628621354486227) -- (\index+6,6/1.6180339887498948482045868343656381177203091798057628621354486227);
\node[anchor=north] at (0+\index,0) {$(0,0)$};
\node[anchor=north] at (6+\index,0) {$(1,0)$};
\node[anchor=south] at (0+\index,6) {$(0,1)$};
\node[anchor=south] at (6+\index,6) {$(1,1)$};
}
\draw (10+6/2.6180339887498948482045868343656381177203091798057628621354486227,0) -- (10+6/2.6180339887498948482045868343656381177203091798057628621354486227,6);
\draw (10,6/2.6180339887498948482045868343656381177203091798057628621354486227) -- (10+6,6/2.6180339887498948482045868343656381177203091798057628621354486227);

\node[scale=1.5] at (3/1.6180339887498948482045868343656381177203091798057628621354486227,3/1.6180339887498948482045868343656381177203091798057628621354486227) {$\cJ_{1}^1$};
\node[scale=1.5] at (3/1.6180339887498948482045868343656381177203091798057628621354486227,3*1.6180339887498948482045868343656381177203091798057628621354486227) {$\cJ_{1}^2$};
\node[scale=1.5] at (3*1.6180339887498948482045868343656381177203091798057628621354486227,3/1.6180339887498948482045868343656381177203091798057628621354486227) {$\cJ_{1}^3$};
\node[scale=1.5] at (3*1.6180339887498948482045868343656381177203091798057628621354486227,3*1.6180339887498948482045868343656381177203091798057628621354486227) {$\cJ_{1}^4$};

\node[scale=1] at (10+3/2.6180339887498948482045868343656381177203091798057628621354486227,3/2.6180339887498948482045868343656381177203091798057628621354486227) {$\cJ_{2}^1$};
\node[scale=1] at (10+3/2.6180339887498948482045868343656381177203091798057628621354486227,3) {$\cJ_{2}^2$};
\node[scale=1] at (10+3,3/2.6180339887498948482045868343656381177203091798057628621354486227) {$\cJ_{2}^3$};
\node[scale=1] at (10+3,3) {$\cJ_{2}^4$};
\node[scale=1] at (10+3/2.6180339887498948482045868343656381177203091798057628621354486227,3*1.6180339887498948482045868343656381177203091798057628621354486227) {$\cJ_{2}^5$};
\node[scale=1] at (10+3,3*1.6180339887498948482045868343656381177203091798057628621354486227) {$\cJ_{2}^6$};
\node[scale=1] at (10+3*1.6180339887498948482045868343656381177203091798057628621354486227,3/2.6180339887498948482045868343656381177203091798057628621354486227) {$\cJ_{2}^7$};
\node[scale=1] at (10+3*1.6180339887498948482045868343656381177203091798057628621354486227,3) {$\cJ_{2}^8$};
\node[scale=1] at (10+3*1.6180339887498948482045868343656381177203091798057628621354486227,3*1.6180339887498948482045868343656381177203091798057628621354486227) {$\cJ_{2}^9$};
\end{tikzpicture}
\caption{}
\label{Figure_Labelling_the_Rectangles_in_the_partititons_of_the_unit_square}
\end{figure}

\subsection{A Construction of a Cantor Set by $\nu$}
\label{Section_Fibonacci_SFC_Subsection_A_construction_of_the_Cantor_Set}
Consider the iteration rules shown in Figure \ref{Figure_The_Function_Psi}. The rules are defined for $4$ different interval types of arbitrary length and demonstrate how to subdivide each interval type. Start with the interval $S_0=[0,1]$ attached with label $a$. Iterating it according to the given rules generates $4$ subintervals with distinct labels $a,b,c,d$, respectively. Let $S_1$ denote the union of those $4$ subintervals. For each generated subinterval, apply the rules in Figure \ref{Figure_The_Function_Psi} and denote the union of those subintervals as $S_2$, and so on so forth. Observe that the set $\Gamma=\bigcap\limits_{i=0}^{\infty}S_i$ is a cantor set, when we ignore the labels attached to $S_i$'s. The first steps of the construction of $\Gamma$ is shown in Figure \ref{Figure_Construction_Steps_of_the_Cantor_Set}.

\begin{figure}[H]
\centering
\begin{tikzpicture}[scale=1.15]
\centering
\foreach \index in {0, -2, -4, -6}{
\draw (0,0+\index)--(1,0+\index);
\draw (0,-0.2+\index) -- (0,0.2+\index);
\node at (0,-0.5+\index) {$x$};
\draw (1,-0.2+\index) -- (1,0.2+\index);
\node at (1,-0.5+\index) {$y$};
\draw[very thick, ->] (1.25,0+\index) -- (1.75,0+\index);
}

\draw (2,0) -- (2+9/7,0);
\draw (2+2*9/7,0) -- (2+3*9/7,0);
\draw (2+4*9/7,0) -- (2+5*9/7,0);
\draw (2+6*9/7,0) -- (2+7*9/7,0);
\foreach \ind in {0,1,2,3,4,5,6,7}{
\draw (2+\ind*9/7,0.2) -- (2+\ind*9/7,-0.2);
}
\node[anchor=south] at (0.5,0) {$a$};
\node at (2,-0.5) {$x$};
\node at (11,-0.5) {$y$};
\node at (2+9/7,-0.5) {$\frac{6x+y}{7}$};
\node at (2+18/7,-0.5) {$\frac{5x+2y}{7}$};
\node at (2+27/7,-0.5) {$\frac{4x+3y}{7}$};
\node at (2+36/7,-0.5) {$\frac{3x+4y}{7}$};
\node at (2+45/7,-0.5) {$\frac{2x+5y}{7}$};
\node at (2+54/7,-0.5) {$\frac{x+6y}{7}$};
\node[anchor=south] at (2+9/14,0) {$a$};
\node[anchor=south] at (2+9/14+18/7,0) {$b$};
\node[anchor=south] at (2+9/14+36/7,0) {$c$};
\node[anchor=south] at (2+9/14+54/7,0) {$d$};

\draw (2,-2) -- (2+3,-2);
\draw (2+6,-2) -- (11,-2);
\foreach \ind in {0,1,2,3}{
\draw (2+\ind*3,-2+0.2) -- (2+\ind*3,-2+-0.2);
}
\node[anchor=south] at (0.5,-2) {$b$};
\node at (2,-2.5) {$x$};
\node at (11,-2.5) {$y$};
\node at (2+3,-2.5) {$\frac{2x+y}{3}$};
\node at (2+6,-2.5) {$\frac{x+2y}{3}$};
\node[anchor=south] at (2+1.5,-2) {$a$};
\node[anchor=south] at (2+7.5,-2) {$c$};

\draw (2,-4) -- (2+3,-4);
\draw (2+6,-4) -- (11,-4);
\foreach \ind in {0,1,2,3}{
\draw (2+\ind*3,-4+0.2) -- (2+\ind*3,-4+-0.2);
}
\node[anchor=south] at (0.5,-4) {$c$};
\node at (2,-4.5) {$x$};
\node at (11,-4.5) {$y$};
\node at (2+3,-4.5) {$\frac{2x+y}{3}$};
\node at (2+6,-4.5) {$\frac{x+2y}{3}$};
\node[anchor=south] at (2+1.5,-4) {$a$};
\node[anchor=south] at (2+7.5,-4) {$b$};

\draw (2+4.5,-6) -- (11,-6);
\foreach \ind in {1,2}{
\draw (2+\ind*4.5,-6+0.2) -- (2+\ind*4.5,-6+-0.2);
}
\node[anchor=south] at (0.5,-6) {$d$};
\node at (11,-6.5) {$y$};
%\node at (2+4.5,-6.5) {$\frac{2x+y}{3}$};
\node at (2+4.5,-6.5) {$\frac{x+y}{2}$};
\node[anchor=south] at (2+4.5+2.25,-6) {$a$};

\end{tikzpicture}
\caption{}
\label{Figure_The_Function_Psi}
\end{figure}

%More precisely, let $\cQ=\{a,b,c,d\}$ be a given set of labels and let $S=\{[x,y]:\ x,y\in\RR\ \mathrm{with}\ x<y\}$ denote the set of non-trivial closed intervals in $\RR$. Define
%\begin{align*}
%S^* & = \left\{\bigcup\limits_{j=1}^n [x_j, y_j]:\ n\in\ZZ^+\ \mathrm{and}\  x_1<y_1<x_2<y_2<\dots<x_n<y_n\right\},\\
%R & = \left\{\bigcup\limits_{j=1}^{n}\{(I_j,p_j)\}:\ \bigcup\limits_{j=1}^n I_j\in R\ \mathrm{and}\ p_j\in\cQ\ \mathrm{for\ all}\ j\in\{1,\dots,n\}\right\}.
%\end{align*}
%The rules in Figure \ref{Figure_The_Function_Psi} can be characterised as a map $\Psi:S\times \cQ\mapsto S^*$ by:
%\begin{align*}
%\Psi([x,y],a) & =\left\{([x,x_1],a),([x_2,x_3],b),([x_4,x_5],c),([x_6,y],d)\right\}\\ \nonumber
%& where\ x_i=x+\frac{y-x}{7}\cdot i\ \mathrm{for\ each\ }i=1,\dots,6,   \\ 
%\Psi([x,y],b) & =\left\{\left(\left[x,x+\frac{y-x}{3}\right], a\right),\left(\left[x+\frac{2(y-x)}{3},y\right], c\right)\right\},\\
%\Psi([x,y],c) & =\left\{\left(\left[x,x+\frac{y-x}{3}\right], a\right),\left(\left[x+\frac{2(y-x)}{3},y\right], b\right)\right\},\\
%\Psi([x,y],d) & =\left\{\left(\left[x+\frac{y-x}{2},y\right], a\right)\right\},
%\end{align*}
%where $x,y\in\RR$ with $x<y$. Extend $\Psi$ to $R$ by $\Psi\left(\bigcup\limits_{j=1}^n\left\{(I_j,p_j)\right\}\right):=\bigcup\limits_{j=1}^n\Psi(I_i,p_i)$.
%Define also $\Phi:R\mapsto S^*$ by $\Phi\left(\bigcup\limits_{j=1}^n\left\{(I_j,p_j)\right\}\right)=\bigcup\limits_{j=1}^n I_j$ for $\bigcup\limits_{j=1}^n\left\{(I_j,p_j)\right\}\in R$. Observe that $\Gamma=\lim\limits_{n\to\infty}\Phi\left(\Psi^n\left([0,1],a\right)\right)$ is a Cantor set.

\begin{figure}[h]
\centering

\begin{tikzpicture}
\draw (0,0) -- (14,0);
\draw(0,0.2) -- (0,-0.2);
\draw(14,0.2) -- (14,-0.2);

\draw(0,0-2) -- (2,0-2);
\draw(4,0-2) -- (6,0-2);
\draw(8,0-2) -- (10,0-2);
\draw(12,0-2) -- (14,0-2);

\draw(0,0-4) -- (2/7,0-4);
\draw(4/7,0-4) -- (6/7,0-4);
\draw(8/7,0-4) -- (10/7,0-4);
\draw(12/7,0-4) -- (14/7,0-4); 

\draw(4,0-4) -- (4+2/3,0-4);
\draw(4+4/3,0-4) -- (4+2,0-4);

\draw(8,0-4) -- (8+2/3,0-4);
\draw(8+4/3,0-4) -- (10,0-4);

\draw(13,0-4) -- (14,0-4);

\foreach \in in {0,2,4,6,8,10,12,14}{
\draw(\in,0.2-2) -- (\in,-0.2-2);
}

\foreach \inx in {0,2/7,4/7,6/7,8/7,10/7,12/7,2,4,4+2/3,4+4/3,6,8,8+2/3,8+4/3,10,13,14}{
\draw(\inx,0.2-4) -- (\inx,-0.2-4);
}
\filldraw (7,-5) circle (1pt) node{};
\filldraw (7,-5.5) circle (1pt) node{};
\filldraw (7,-6) circle (1pt) node{};

\foreach \innn in {0, -2, -4}{
\node at (0,\innn+-0.5) {$0$};
\node at (14,\innn+-0.5) {$1$};
}

\end{tikzpicture}
\caption{Construction steps of $\Gamma$}
\label{Figure_Construction_Steps_of_the_Cantor_Set}
\end{figure}
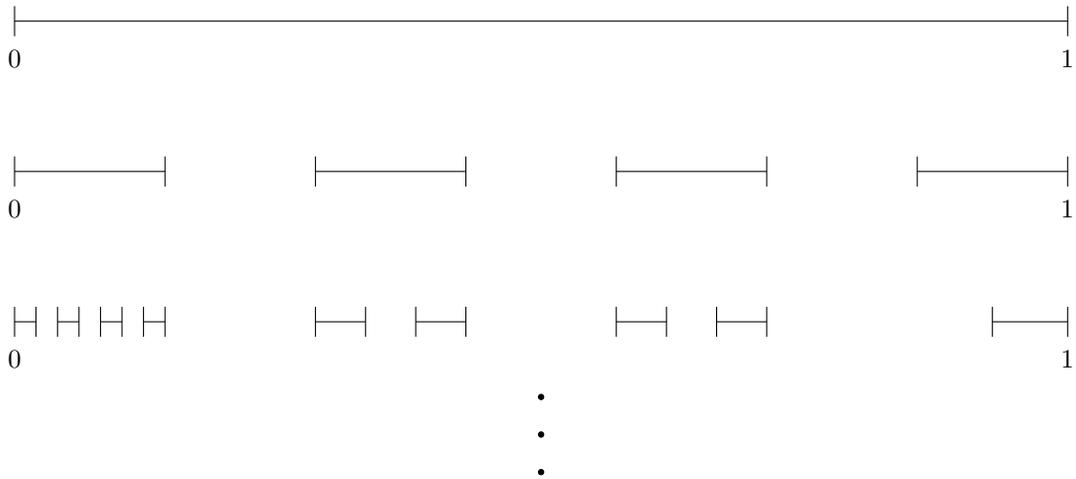

\subsection{A Space-Filling Curve by $\nu$}

Notice that there exist $\cF_{k+2}^2$ many disjoint intervals in the n-th step\footnote{With the convention that [0,1] is the 0-th step.} of the construction of $\Gamma$, for every $n\in\NN$. Denote these intervals by $\cI_{k}^{1}, \dots, \cI_{k}^{\cF_{k+2}^2}$, from left to right respectively. The intervals $\cI_{0}^1$, $\cI_{1}^1, \dots, \cI_{1}^4$ and $\cI_{2}^1, \dots, \cI_{2}^9$ are demonstrated in Figure \ref{Figure_Construction_Steps_of_the_Cantor_Set_0}.

\begin{figure}[ht]
\centering

\begin{tikzpicture}
\draw (0,0) -- (14,0);
\draw(0,0.2) -- (0,-0.2);
\draw(14,0.2) -- (14,-0.2);

\draw(0,0-2) -- (2,0-2);
\draw(4,0-2) -- (6,0-2);
\draw(8,0-2) -- (10,0-2);
\draw(12,0-2) -- (14,0-2);

\draw(0,0-4) -- (2/7,0-4);
\draw(4/7,0-4) -- (6/7,0-4);
\draw(8/7,0-4) -- (10/7,0-4);
\draw(12/7,0-4) -- (14/7,0-4); 

\draw(4,0-4) -- (4+2/3,0-4);
\draw(4+4/3,0-4) -- (4+2,0-4);

\draw(8,0-4) -- (8+2/3,0-4);
\draw(8+4/3,0-4) -- (10,0-4);

\draw(13,0-4) -- (14,0-4);

\foreach \in in {0,2,4,6,8,10,12,14}{
\draw(\in,0.2-2) -- (\in,-0.2-2);
}

\foreach \inx in {0,2/7,4/7,6/7,8/7,10/7,12/7,2,4,4+2/3,4+4/3,6,8,8+2/3,8+4/3,10,13,14}{
\draw(\inx,0.2-4) -- (\inx,-0.2-4);
}

\foreach \innn in {0, -2, -4}{
\node at (0,\innn+-0.5) {$0$};
\node at (14,\innn+-0.5) {$1$};
}

\node at (7,0.5) {$\cI_0^1$};

\node at (1,0.5-2) {$\cI_1^1$};
\node at (5,0.5-2) {$\cI_1^2$};
\node at (9,0.5-2) {$\cI_1^3$};
\node at (13,0.5-2) {$\cI_1^4$};

\node[scale=0.7] at (1/7,0.5-4) {$\cI_2^1$};
\node[scale=0.7] at (5/7,0.5-4) {$\cI_2^2$};
\node[scale=0.7] at (9/7,0.5-4) {$\cI_2^3$};
\node[scale=0.7] at (13/7,0.5-4) {$\cI_2^4$};

\node[scale=0.7] at (4+1/3,0.5-4) {$\cI_2^5$};
\node[scale=0.7] at (4+5/3,0.5-4) {$\cI_2^6$};

\node[scale=0.7] at (8+1/3,0.5-4) {$\cI_2^7$};
\node[scale=0.7] at (8+5/3,0.5-4) {$\cI_2^8$};

\node[scale=0.7] at (13.5,0.5-4) {$\cI_2^9$};

\end{tikzpicture}
\caption{}
\label{Figure_Construction_Steps_of_the_Cantor_Set_0}
\end{figure}

%\subsection{A Space-Filling Curve by $\nu$}
There exist bijective correspondences between the intervals appearing in the construction steps of $\Gamma$, first three of which are illustrated in Figure \ref{Figure_Construction_Steps_of_the_Cantor_Set_0}, and the rectangles in the partitions of the unit square demonstrated in Section \ref{Section_Fibonacci_SFC_Subsection_Partitions_of_the_Unit_Square}. Precisely, for each $n\in\NN$, $f_n:\{\cI^k_n:\ k=1,\dots,\cF_{n+2}^2\}\mapsto \{\cJ^k_n:\ k=1,\dots,\cF_{n+2}^2\}$ defined by $f_n(\cI^k_n)=\cJ^k_n$ for $k=1,\dots,\cF_{n+2}^2$, is a bijection. Using this bijective correspondence, we define a space filling curve as follows.

Let $x\in\Gamma$ be fixed. For every $n\in\NN$, there exists unique $k_n\in\{1,\dots,\cF_{n+2}^2\}$ such that $x\in\cI_{n}^{k_n}$. In fact, $\{x\}=\bigcap\limits_{n=1}^{\infty}\cI_n^{k_n}$ by Cantor's intersection theorem. Similarly, $\bigcap\limits_{n=1}^{\infty}\cJ_{n}^{k_n}=\{y\}$ for some unique $y\in[0,1]\times[0,1]$ due to Cantor's intersection theorem. That is, for each fixed $x\in\Gamma$, there is a unique $y\in[0,1]\times[0,1]$. This induces a map $f:\Gamma\mapsto[0,1]\times[0,1]$ such that $f(x)=y$ where $\{x\}=\bigcap\limits_{n=1}^{\infty}\cI_n^{k_n}$ and $\{y\}=\bigcap\limits_{n=1}^{\infty}\cJ_{n}^{k_n}$.

\paragraph{$f$\ is\ surjective:} For any $y\in[0,1]\times[0,1]$, choose a sequence $m_n$ so that $\{y\}=\bigcap\limits_{n=1}^{\infty}\cJ_{n}^{m_n}$. Note that such sequence exists but not necessarily unique. Notice also that $\bigcap\limits_{n=1}^{\infty}\cI_{n}^{m_n}$ is a unique point. In fact, $f(\bigcap\limits_{n=1}^{\infty}\cI_{n}^{m_n})=y$.

\paragraph{$f$\ is\ continuous:} %More precisely, for each $n\in\NN$, define $f_n:\{\cI^k_n:\ k=1,\dots,\cF_{n+2}^2\}\mapsto \{\cJ^k_n:\ k=1,\dots,\cF_{n+2}^2\}$ by $f_n(\cI^k_n)=\cJ^k_n$ for $k=1,\dots,\cF_{n+2}^2$.  For every $x\in\Gamma$ and $n\in\NN$, there exists $k_n\in\{1,\dots,\cF_{n+2}^2\}$ such that $x\in\cI_{n}^{k_n}$. In particular, $\{x\}=\bigcap\limits_{n=1}^{\infty}\cI_n^{k_n}$ by Cantor's intersection theorem. By the same token, there exists $y\in[0,1]\times[0,1]$ such that $\{y\}=\bigcap\limits_{n=1}^{\infty}\cJ_{n}^{k_n}$. This process induces a surjection $f:\Gamma\mapsto[0,1]\times[0,1]$ such that $f(x)=y$ where $x$ and $y$ are as defined above. Next we prove that $f$ is continuous.
For each $n\in\ZZ^+$, define the following:
\begin{align*}
g_n & = \max\limits_{j\in\{1,\dots,\cF_{n+2}^2\}}\left\{l(\cI^j_n):\ l(\cI^j_n)\ \mathrm{denotes\ the\ length\ of\ }\cI^j_n\right\}, \\
h_n & = \max\limits_{j\in\{1,\dots,\cF_{n+2}^2\}}\left\{diam(\cJ^j_n):\ diam(\cJ^j_n)\ \mathrm{denotes\ the\ diameter\ of\ }\cJ^j_n\right\}.
\end{align*}
We have $g_n\downarrow0$ and $h_n\downarrow0$ as $n\to\infty$. Let $x\in\Gamma$ be fixed and let $\epsilon>0$ be given. Choose $N\in\ZZ^+$ sufficiently large so that $h_N<\epsilon$. Set $\delta=g_N/2$. If $y\in\Gamma$ with $|x-y|<\delta$ then $x,y\in\cI_{N}^{j_0}$ for some $j_0\in\{1,\dots,\cF_{N+2}^2\}$. That is, $f(x),f(y)\in\cJ_N^{j_0}$ and $||f(x)-f(y)||<\epsilon$. Thus, $f$ is continuous.

Since $f:\Gamma\mapsto[0,1]\times[0,1]$ is a continuous surjection, it extends to a space-filling curve $F:[0,1]\mapsto[0,1]\times[0,1]$ by linear interpolation (See \cite{Journal_Sagan_1} for details of such process).

\paragraph{A Geometrisation of $F$}
For each $n\in\ZZ^+$, puncture the centre of every rectangle in $\lambda^{-n-1}\cdot\nu^n(p_a)$. Denote these points by $x(\cJ^j_n)$ for $j=1,\dots,\cF_{n+2}^2$. Join the punctures $x(\cJ^1_n), x(\cJ^2_n), \dots, x(\cJ^{\cF_{n+2}^2}_n)$ with straight lines, respectively. The constructed directed curve is called the \emph{n-th\ approximant\ of\ $F$}. First four approximants of $F$ are shown in Figure \ref{ff}.

\begin{figure}[ht]
\centering
\begin{subfigure}[b]{0.18\textwidth}
\centering
\tikz[remember picture]\node[inner sep=0pt,outer sep=0pt] (a){\includegraphics[width=\linewidth]{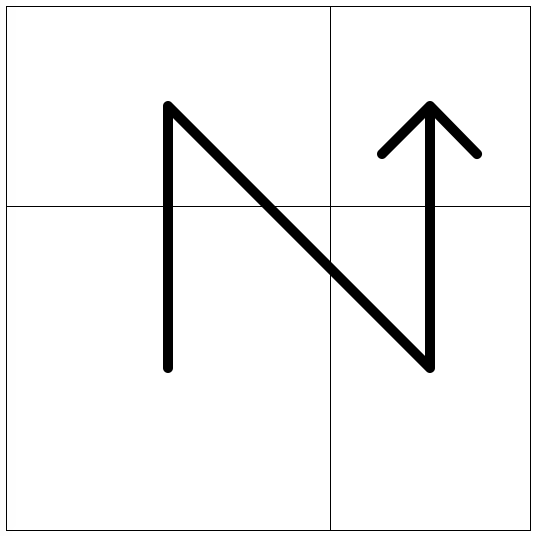}};
\end{subfigure}
\hfill
\begin{subfigure}[b]{0.18\textwidth}
\centering
\tikz[remember picture]\node[inner sep=0pt,outer sep=0pt] (b){\includegraphics[width=\linewidth]{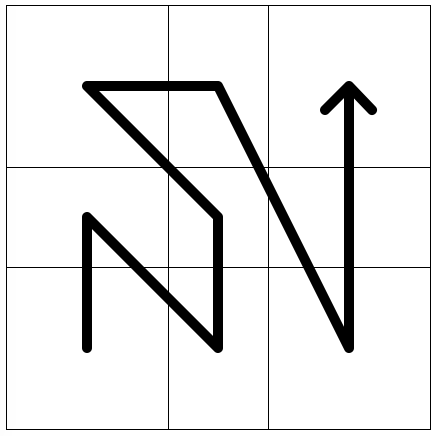}};
\end{subfigure}
\hfill
\begin{subfigure}[b]{0.18\textwidth}
\centering
\tikz[remember picture]\node[inner sep=0pt,outer sep=0pt] (c){\includegraphics[width=\linewidth]{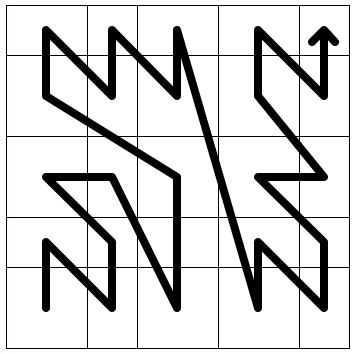}};
\end{subfigure}
\hfill
\begin{subfigure}[b]{0.18\textwidth}
\centering
\tikz[remember picture]\node[inner sep=0pt,outer sep=0pt] (d){\includegraphics[width=\linewidth]{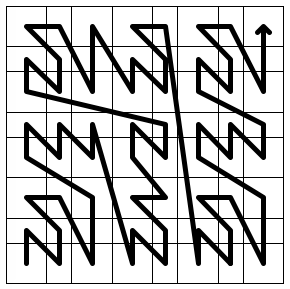}};
\end{subfigure}
\tikz[remember picture,overlay]\draw[line width=2pt,-stealth] ([xshift=5pt]a.east) -- ([xshift=30pt]a.east)node[midway,above,text=black,font=\LARGE\bfseries\sffamily] {};
\tikz[remember picture,overlay]\draw[line width=2pt,-stealth] ([xshift=5pt]b.east) -- ([xshift=30pt]b.east)node[midway,above,text=black,font=\LARGE\bfseries\sffamily] {};
\tikz[remember picture,overlay]\draw[line width=2pt,-stealth] ([xshift=5pt]c.east) -- ([xshift=30pt]c.east)node[midway,above,text=black,font=\LARGE\bfseries\sffamily] {};
\caption{First four approximants of $F$}
\label{ff}
\end{figure}

\begin{figure}[H]
\centering
\includegraphics[width=\textwidth]{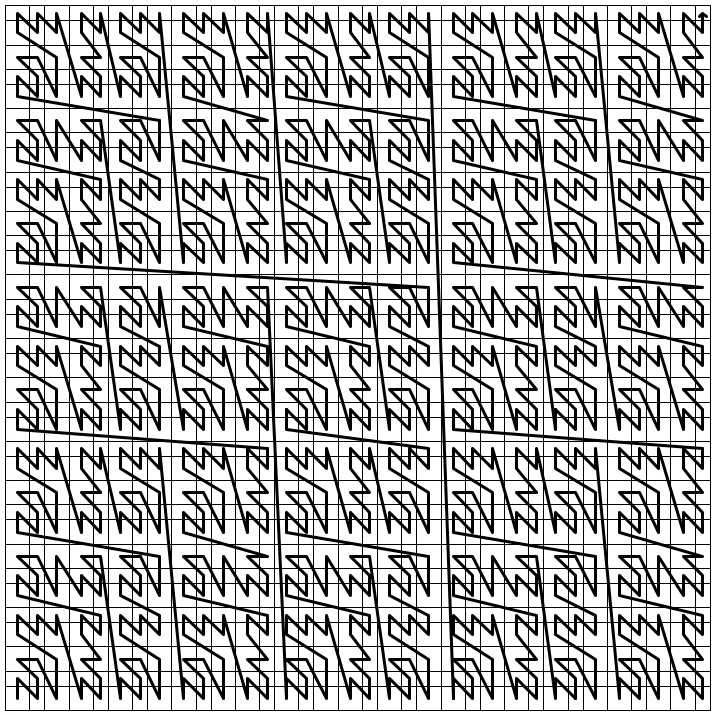}
\caption{7th approximant of $F$}
\label{ff_7th}
\end{figure}

\section{A Space-Filling Curve Generator Algorithm}\label{Section_proof_of_the_main_result}

In this section we generalise the argument presented in Section \ref{Section_Background}.
%The space-filling curve generated in Section \ref{Section_Background} can be generalized by the following theorem.

\begin{theorem}\label{main_theorem}
Let $\cP$ be a given finite collection of tiles in $\RR^2$. Suppose $\omega$ is a substitution defined over $\cP$ such that $\max\{diam(t):\ t\in\lambda^{-n}\omega^n(p),\ p\in\cP\}\downarrow 0$ as $n\to\infty$, where $diam(t)$ denotes the diameter of $\sss t$ and $\lambda$ denotes the expansion factor of $\omega$. Then for each $p\in\cP$ there exists a Cantor set $\Gamma_{\omega,p}\subseteq[0,1]$ and a continuous surjection $f_{\omega,p}:\Gamma_{\omega,p}\mapsto \sss p$ such that $f_{\omega,p}$ extends to a space-filling curve $F_{\omega,p}:[0,1]\mapsto\sss p$ by linear interpolation.
\end{theorem}

\begin{proof}%[Proof of Theorem \ref{main_theorem}]
Let $\cP$ be a finite collection of tiles. Suppose $\omega:\cP\mapsto\cP^*$ is a substitution with an expansion factor $\lambda>1$ such that $\max\{diam(t):\ t\in\lambda^{-n}\omega^n(q),\ q\in\cP\}\downarrow 0$ as $n\to\infty$. Assume without loss of generality %$(0,0)\in V_p$ and 
$|\omega(q)|>1$ for all $q\in\cP$.

Define a bijection $g_q:\omega(q)\mapsto\{1,\dots,|\omega(q)|\}$, for every $q\in\cP$. For $q\in\cP$, $g_q$ defines a total order $\underset{q,1}{\lesssim}$ between the tiles of $\omega(q)$ such that $x\underset{q,1}{\lesssim}y$ whenever $g_q(x)\leq g_q(y)$, for $x,y\in\omega(q)$. Extend these total orders to the supertiles of $\omega$ inductively such that
\begin{enumerate}
    \item[(i)]  If $x,y\in\omega^{n-1}(t)$ for some $t\in \omega(q)$ and $n\in\ZZ^+\backslash\{1\}$, then $x\underset{q,n}{\lesssim}y$ whenever $x\underset{q,n-1}{\lesssim}y$.
    \item[(ii)]  If $x\in\omega^{n-1}(t_1)$ and $y\in\omega^{n-1}(t_2)$ for some distinct $t_1,t_2\in \omega(q)$ and $n\in\ZZ^+\backslash\{1\}$, then $x\underset{q,n}{\lesssim}y$ whenever $t_1\underset{q,1}{\lesssim}t_2$.
\end{enumerate}
The total order $\underset{q,n}{\lesssim}$ for $q\in\cP$ and $n\in\ZZ^+$, can be transferred over the scaled patch $\lambda^{-n}\cdot \omega^n(q)$. Precisely, for each $q\in\cP$ and $n\in\ZZ^+$, label the scaled tiles in $\lambda^{-n}\cdot \omega^n(q)$ by $\cJ^1_{q,n}, \cJ^2_{q,n},\dots,\cJ^{|\omega^n(q)|}_{q,n}$ such that $\cJ_{q,n}^i\underset{q,n}{\lesssim}\cJ_{q,n}^j$ if and only if $i\leq j$. Next we construct a Cantor set. 

For each $q\in\cP$, define a subdivision rule by partitioning a given interval of random length into $2.|w(q)|-1$ many equal length subintervals and removing the even indexed subintervals as shown in Figure \ref{label_subdivision__} and Figure \ref{label_subdivision__2}, respectively.

\begin{figure}[h]
\centering
\begin{tikzpicture}
\draw (0,0) -- (2,0);
\draw (4,0) -- (14,0);
\draw(0,0.2) -- (0,-0.2);
\draw(2,0.2) -- (2,-0.2);

\draw[very thick, ->] (2.5,0) -- (3.5,0);

\foreach \inx in {0,1,2,3,14,15}{
\draw(4+\inx*2/3,0.2) -- (4+\inx*2/3,-0.2);
}

\filldraw (7,0.5) circle (1pt) node{};
\filldraw (8,0.5) circle (1pt) node{};
\filldraw (9,0.5) circle (1pt) node{};

\foreach \innn in {0}{
\node at (1,\innn+0.5) {$I$};
\node at (4+1/3,\innn+0.5) {$I_1$};
\node at (4+1/3+2/3,\innn+0.5) {$I_2$};
\node at (4+1/3+4/3,\innn+0.5) {$I_3$};
\node at (14+1/3,\innn+0.5) {$I_{2.|w(q)|-1}$};
\node at (0,\innn-0.5) {$x$};
\node at (2,\innn-0.5) {$y$};
\node at (4,\innn-0.5) {$x$};
\node at (14,\innn-0.5) {$y$};
}

\end{tikzpicture}
\caption{$I$ is dissected into $2.|w(q)|-1$ many closed subintervals of length $(2.|w(q)|-1)^{-1}$.}
\label{label_subdivision__}
\end{figure}
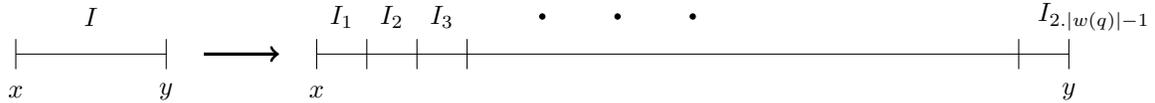

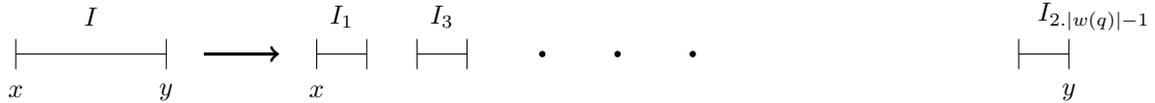
\begin{figure}[h]
\centering
\begin{tikzpicture}
\draw (0,0) -- (2,0);

\draw (4,0) -- (4+2/3,0);
\draw (4+4/3,0) -- (4+2,0);

\draw (14-2/3,0) -- (14,0);

\draw(0,0.2) -- (0,-0.2);
\draw(2,0.2) -- (2,-0.2);

\draw[very thick, ->] (2.5,0) -- (3.5,0);

\foreach \inx in {0,1,2,3,14,15}{
\draw(4+\inx*2/3,0.2) -- (4+\inx*2/3,-0.2);
}

\filldraw (7,0) circle (1pt) node{};
\filldraw (8,0) circle (1pt) node{};
\filldraw (9,0) circle (1pt) node{};

\foreach \innn in {0}{
\node at (1,\innn+0.5) {$I$};
\node at (4+1/3,\innn+0.5) {$I_1$};
\node at (4+1/3+4/3,\innn+0.5) {$I_3$};
\node at (14+1/3,\innn+0.5) {$I_{2.|w(q)|-1}$};
\node at (0,\innn-0.5) {$x$};
\node at (2,\innn-0.5) {$y$};
\node at (4,\innn-0.5) {$x$};
\node at (14,\innn-0.5) {$y$};
}

\end{tikzpicture}
\caption{Even number indexed subintervals are removed.}
\label{label_subdivision__2}
\end{figure}
Next, add labels to intervals $I,I_1,I_3,\dots,I_{2.|\omega(q)|-1}$ using $g_q$ as follows:
\[
l(I)=l(q) \quad\mathrm{and}\quad l(I_{2i-1})=l(g_q^{-1}(i))\ \mathrm{for\ }i\in\{1,\dots,|\omega(q)|\}.
\]
This process defines a subdivision rule of intervals with labels. Let $p\in\cP$ be fixed. Start with the interval $[0,1]$ with label $l(p)$. Applying the defined subdivision rules inductively induces a Cantor set $\Gamma_{\omega,p}$, as explained in Section \ref{Section_Fibonacci_SFC_Subsection_A_construction_of_the_Cantor_Set}.
%\newpage
%Define the following:
%\begin{align*}
%S & =\left\{[x,y]:\ x,y\in\RR\ \mathrm{with}\ x<y\right\},\\
%S^* & = \left\{\bigcup\limits_{j=1}^n [x_j, y_j]:\ n\in\ZZ^+\ \mathrm{and}\  x_1<y_1<x_2<y_2<\dots<x_n<y_n\right\},\\
%R & = \left\{\bigcup\limits_{j=1}^{n}\{(I_j,q_j)\}:\ \bigcup\limits_{j=1}^n I_j\in R\ \mathrm{and}\ q_j\in\cP\ \mathrm{for\ all}\ j\in\{1,\dots,n\}\right\}.
%\end{align*}
%Consider the map $\Psi:S\times \cP\mapsto S^*$ defined by $\Psi([x,y],q) = \bigcup\limits_{j=0}^{|\omega(q)|-1} [x_{2j}, x_{2j+1}]$
%where $x_k=x+\frac{y-x}{2|\omega(q)|-1}\cdot k$ for $k=0,1,\dots, 2|\omega(q)|-1$. Extend $\Psi$ to $R$ by $\Psi\left(\bigcup\limits_{j=1}^n\left\{(I_j,q_j)\right\}\right):=\bigcup\limits_{j=1}^n\Psi(I_i,q_i)$. Define also $\Phi:R\mapsto S^*$ by $\Phi\left(\bigcup\limits_{j=1}^n\left\{(I_j,q_j)\right\}\right):=\bigcup\limits_{j=1}^n I_j$. Observe that $\Gamma_{\omega,p}=\lim\limits_{n\to\infty}\Phi\left(\Psi^n\left([0,1],p\right)\right)$ is a Cantor set.

Denote the intervals appearing in the n-th step\footnote{With the convention that [0,1] is the 0-th step.} of the construction of $\Gamma_{\omega,p}$ by $\cI_{p,n}^1,\dots,\cI_{p,n}^{|\omega^n(p)|}$, from left to right respectively. For each $n\in\ZZ^+$, $f_{p,n}:\{\cI^k_{p,n}:\ k=1,\dots,|\omega^n(p)|\}\mapsto \{\cJ^k_{p,n}:\ k=1,\dots,|\omega^n(p)|\}$ defined by $f_{p,n}(\cI^k_{p,n})=\cJ^k_{p,n}$ for $k=1,\dots,|\omega^n(p)|$ is a well-defined bijection.  For each $x\in\Gamma_{\omega,p}$ and $n\in\NN$, there exists $k_{p,n}\in\{1,\dots,|\omega^n(p)|\}$ such that $x\in\cI_{p,n}^{k_{p,n}}$. In particular, $\{x\}=\bigcap\limits_{n=1}^{\infty}\cI_{p,n}^{k_{p,n}}$ by Cantor's intersection theorem. By the same token, there exists $y\in\sss p$ such that $\{y\}=\bigcap\limits_{n=1}^{\infty}\cJ_{p,n}^{k_{p,n}}$. This process induces a surjection $f_{\omega,p}:\Gamma_{\omega,p}\mapsto\sss p$ such that $f_{\omega,p}(x)=y$ where $x$ and $y$ are as defined above. Next we prove that $f_{\omega,p}$ is continuous.

For each $n\in\ZZ^+$, define the following:
\begin{align*}
g_{p,n} & = \max\limits_{j\in\{1,\dots,|\omega^n(p)|\}}\left\{l(\cI^j_{p,n}):\ l(\cI^j_{p,n})\ \mathrm{denotes\ the\ length\ of\ }\cI^j_{p,n}\right\}, \\
h_{p,n} & = \max\limits_{j\in\{1,\dots,|\omega^n(p)|\}}\left\{diam(\cJ^j_{p,n}):\ diam(\cJ^j_{p,n})\ \mathrm{denotes\ the\ diameter\ of\ }\cJ^j_{p,n}\right\}.
\end{align*}
We have that $g_{p,n}\downarrow0$ and $h_{p,n}\downarrow 0$ as $n\to\infty$. Choose $x\in\Gamma_{\omega,p}$ and $\epsilon>0$. Pick $N_p\in\ZZ^+$ sufficiently large so that $h_{p,N_p}<\epsilon$. Set $\delta=\frac{g_{p,N_p}}{2}$. If $y\in\Gamma_{\omega,p}$ with $|x-y|<\delta$ then $x,y\in\cI_{p,N_p}^{j_0}$ for some $j_0\in\{1,\dots,|\omega^{N_p}(p)|\}$. That is, $f(x),f(y)\in\cJ_{p,N_p}^{j_0}$ and $||f(x)-f(y)||<\epsilon$. Thus, $f_{\omega,p}$ is continuous, and extends to a space-filling curve $F_{\omega,p}:[0,1]\mapsto\sss p$ by linear interpolation.
\end{proof}

%\begin{corollary}
%Let $\cP$ be a given finite collection of tiles in $\RR$. Suppose further $\omega:\cP\mapsto\cP^*$ is a substitution such that $\lim\limits_{n\to\infty}\max\{l(t):\ t\in\lambda^{-n}\omega^n(p),\ p\in\cP\}=0$, where $l(t)$ is the length of the interval $\sss t$ and $\lambda$ is the expansion factor of $\omega$. Then for each $p\in\cP$ there exists a Cantor set $\Gamma_{\omega,p}\subseteq[0,1]$ and a continuous surjection $f_{\omega,p}:\Gamma_{\omega,p}\mapsto \sss \omega(p)\times\sss \omega(p)$ such that $f_{\omega,p}$ extends to a space-filling curve $F_{\omega,p}:[0,1]\mapsto\sss \omega(p)\times\sss \omega(p)$ by linear interpolation.
%\end{corollary}

Theorem \ref{main_theorem} can be regarded as an algorithm. Precisely,
for every finite substitution $\omega:\cP\mapsto\cP^*$ satisfying the condition in Theorem \ref{main_theorem}, and a tile $p\in\cP$, the following steps form a space-filling curve.

$\mathbf{Step - 1:}$ Choose $k\in\ZZ^+$ such that $|\omega^k(q)|>1$ for every $q\in\cP$. 

\begin{remark}
Replace $\omega$ with $\omega^k$ for the following steps. We assume without loss of generality $k=1$ for the following steps. 
\end{remark}

$\mathbf{Step - 2:}$ Define a bijection $g_q:\omega(q)\mapsto\{1,\dots,|\omega(q)|\}$, for all $q\in\cP$. 

$\mathbf{Step - 3:}$ The maps $\{g_q:\ q\in\cP\}$ indicate total orders over the supertiles of $\omega$. Label the scaled tiles in $\lambda^{-n}\cdot\omega^n(p)$ for each $p\in\cP$ and $n\in\ZZ^+$, according to the generated total order structures.

$\mathbf{Step - 4:}$ Construct a Cantor set $\Gamma_{\omega,p}$.

$\mathbf{Step - 5:}$ Define a bijection between the intervals appearing in the n-th construction step of $\Gamma_{\omega,p}$ and the scaled tiles in the collection $\lambda^{-n}\cdot\omega^n(p)$, for each $n\in\ZZ^+$. 

$\mathbf{Step - 6:}$ Construct a continuous surjection $f_{\omega,p}:\Gamma_{\omega,p}\mapsto\sss p$ using the bijective correspondences described in Step - 5.

$\mathbf{Step - 7:}$ Construct a space-filling curve $F_{\omega,p}:[0,1]\mapsto\sss p$ by linear interpolation over $f_{\omega,p}$.
\subsection{Space-Filling Curve Examples}\label{subsection_SFC}
%\subsection{Examples}\label{Section_Examples}
In Section \ref{subsection_SFC} we apply the algorithm induced from Theorem \ref{main_theorem} using some of the known substitutions. The substitutions provided in this section, as well as a vast collection of other substitutions, can be found at \cite{Web_Tiling_Encyclopedia}. The generated space-filling curves are elucidated by their associated approximants (Definition \ref{d_nth-approximant-of-F}). %Furthermore, in Section \ref{Section_fractal}, we demonstrate how relatively dense fractal-like sets are constructed if some additional assumptions are met for substitutions.

\begin{definition}\label{d_nth-approximant-of-F}
Let $F_{\omega,p}$ be a space filling curve constructed by the algorithm in Section \ref{Section_proof_of_the_main_result}. With the same notations in the proof of Theorem \ref{main_theorem}, for each $n\in\ZZ^+$ and $j\in\{1,\dots,|\omega^n(p)|\}$, denote the centre of $\cJ^j_{p,n}$ by $x(\cJ^j_{p.n})$. The directed curve formed by joining the points $x(\cJ^1_{p.n}), x(\cJ^2_{p.n}), \dots, x(\cJ^{|\omega^n(p)|}_{p.n})$ successively is called the \emph{n-th\ approximant\ of\ $F_{\omega,p}$}.
\end{definition}

\begin{example}[Thue-Morse]
Consider the substitution given in Figure \ref{2DTMsubrule}. The substitution is called \emph{2-dimensional\ Thue-Morse\ substitution} (2DTM in short). It is defined over two unit squares with labels $A,B$. The expansion factor for this substitution is $2$. Choose a square tile with label $A$ to input in the algorithm. Define an order structure over the 1-supertiles of 2DTM through the curves depicted in Figure \ref{2DTMsubrule_Lebesgue}, according to which tile is visited first by the curves. The associated orders are described by the numbers attached to the tiles in Figure \ref{2DTMsubrule_with_curves}\footnote{For the rest of the examples we explain total order structures through directed curves alone. The associated total orders are defined according to which tile is visited first}. Then the space-filling curve generated by the algorithm is nothing but the Lebesgue curve.

On the other hand, define another order structure over the 1-supertiles of 2DTM by the curves shown in Figure \ref{2DTMsubrule_Lebesgue_2DTM}. Let $F_{tm}^A$ denote the space-filling curve formed by the algorithm (by inputting a tile with label $A$). First four approximants of $F_{tm}^A$ are shown in Figure \ref{tm}.
\begin{figure}[H]
\centering
\begin{tikzpicture}
\centering
\foreach \in in {0, 6}{
\draw (0+\in,0.5) rectangle (1+\in,1.5);
\draw[very thick, ->] (1.25+\in,1) -- (1.75+\in,1);
\draw[step=1cm] (2+\in,0) grid (4+\in,2);
}

\node at (0.5,1) {$A$};
\node at (0.5+6,1) {$B$};

\node at (2.5,0.5) {$A$};
\node at (3.5,0.5) {$B$};
\node at (2.5,1.5) {$B$};
\node at (3.5,1.5) {$A$};

\node at (2.5+6,0.5) {$B$};
\node at (3.5+6,0.5) {$A$};
\node at (2.5+6,1.5) {$A$};
\node at (3.5+6,1.5) {$B$};

\end{tikzpicture}
\caption{2-dimensional Thue-Morse substitution}
\label{2DTMsubrule}
\end{figure}
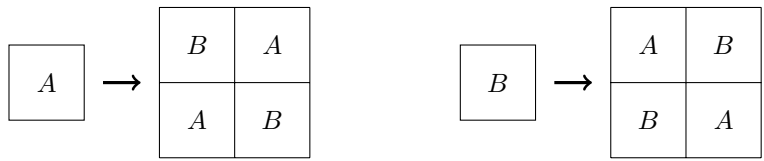

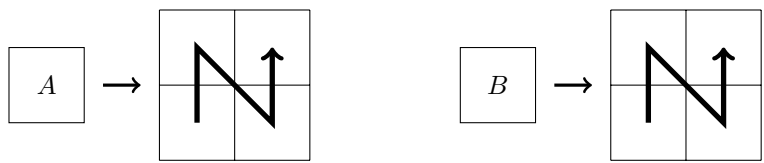
\begin{figure}[H]
\centering
\begin{tikzpicture}
\centering
\foreach \in in {0, 6}{
\draw (0+\in,0.5) rectangle (1+\in,1.5);
\draw[very thick, ->] (1.25+\in,1) -- (1.75+\in,1);
\draw[step=1cm] (2+\in,0) grid (4+\in,2);
\draw[line width=2pt, ->, color=black] (2.5+\in,0.5) -- (2.5+\in,1.5) -- (3.5+\in,0.5) -- (3.5+\in,1.5);
}

\node at (0.5,1) {$A$};
\node at (0.5+6,1) {$B$};

\end{tikzpicture}
\caption{Total orders defined through curves for the 1-supertiles of 2DTM}
\label{2DTMsubrule_Lebesgue}
\end{figure}

\begin{figure}[H]
\centering
\begin{tikzpicture}
\centering
\foreach \in in {0, 6}{
\draw (0+\in,0.5) rectangle (1+\in,1.5);
\draw[very thick, ->] (1.25+\in,1) -- (1.75+\in,1);
\draw[step=1cm] (2+\in,0) grid (4+\in,2);
}

\node at (0.5,1) {$A$};
\node at (0.5+6,1) {$B$};

\node at (2.5,0.5) {$1$};
\node at (3.5,0.5) {$3$};
\node at (2.5,1.5) {$2$};
\node at (3.5,1.5) {$4$};

\node at (2.5+6,0.5) {$1$};
\node at (3.5+6,0.5) {$3$};
\node at (2.5+6,1.5) {$2$};
\node at (3.5+6,1.5) {$4$};

\end{tikzpicture}
\caption{Total orders over 1-supertiles of 2DTM}
\label{2DTMsubrule_with_curves}
\end{figure}
%\end{example}
%\begin{example}[2DTM curve]
%Define another order structure over the 1-supertiles of 2DTM by the curves shown in Figure \ref{2DTMsubrule_Lebesgue_2DTM}. There are two space-filling curves that can be produced by the algorithm. Choose the tile with the label $A$. Let $F_{tm}^A$ denote the space-filling curve formed by the algorithm from this tile. First four approximants of $F_{tm}^A$ are shown in Figure \ref{tm}. 
\begin{figure}[H]
\centering
\begin{tikzpicture}
\centering
\foreach \in in {0, 6}{
\draw (0+\in,0.5) rectangle (1+\in,1.5);
\draw[very thick, ->] (1.25+\in,1) -- (1.75+\in,1);
\draw[step=1cm] (2+\in,0) grid (4+\in,2);
}

\node at (0.5,1) {$A$};
\node at (0.5+6,1) {$B$};

\draw[line width=2pt, ->, color=black] (2.5,0.5) -- (2.5,1.5) -- (3.5,0.5) -- (3.5,1.5);
\draw[line width=2pt, ->, color=black] (2.5+6,0.5) -- (3.5+6,0.5) -- (2.5+6,1.5) -- (3.5+6,1.5);
\end{tikzpicture}
\caption{Total orders over 1-supertiles of 2DTM}
\label{2DTMsubrule_Lebesgue_2DTM}
\end{figure}

\begin{figure}[H]
\centering
\begin{subfigure}[b]{0.18\textwidth}
\centering
\tikz[remember picture]\node[inner sep=0pt,outer sep=0pt] (a){\includegraphics[width=\linewidth]{Figures/TM__1.png}};
\end{subfigure}
\hfill
\begin{subfigure}[b]{0.18\textwidth}
\centering
\tikz[remember picture]\node[inner sep=0pt,outer sep=0pt] (b){\includegraphics[width=\linewidth]{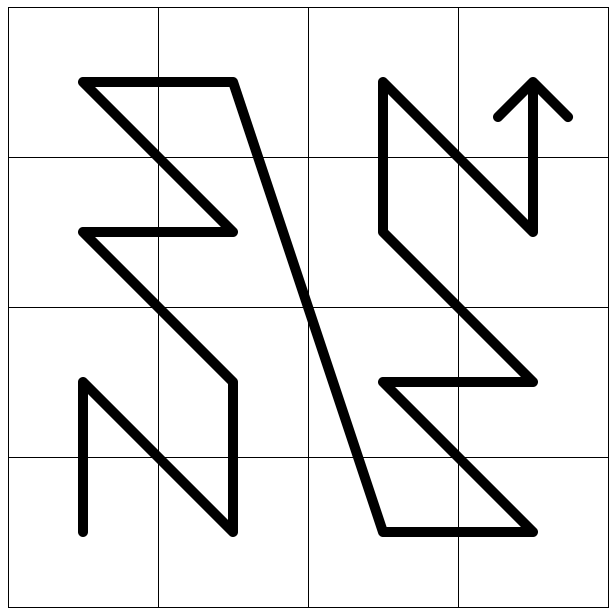}};
\end{subfigure}
\hfill
\begin{subfigure}[b]{0.18\textwidth}
\centering
\tikz[remember picture]\node[inner sep=0pt,outer sep=0pt] (c){\includegraphics[width=\linewidth]{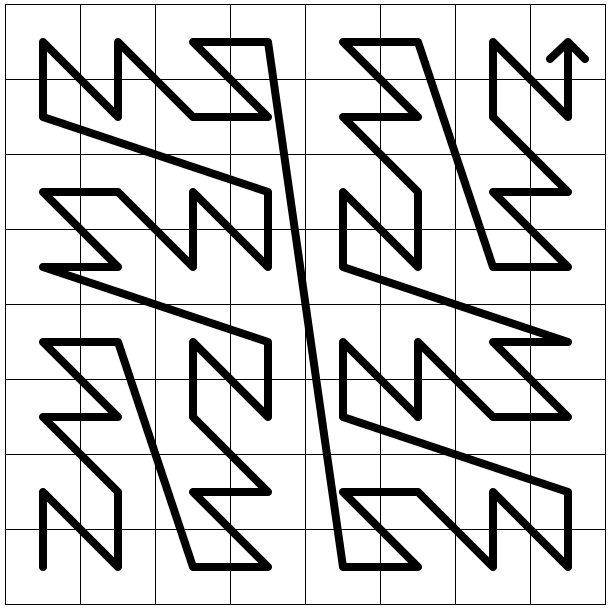}};
\end{subfigure}
\hfill
\begin{subfigure}[b]{0.18\textwidth}
\centering
\tikz[remember picture]\node[inner sep=0pt,outer sep=0pt] (d){\includegraphics[width=\linewidth]{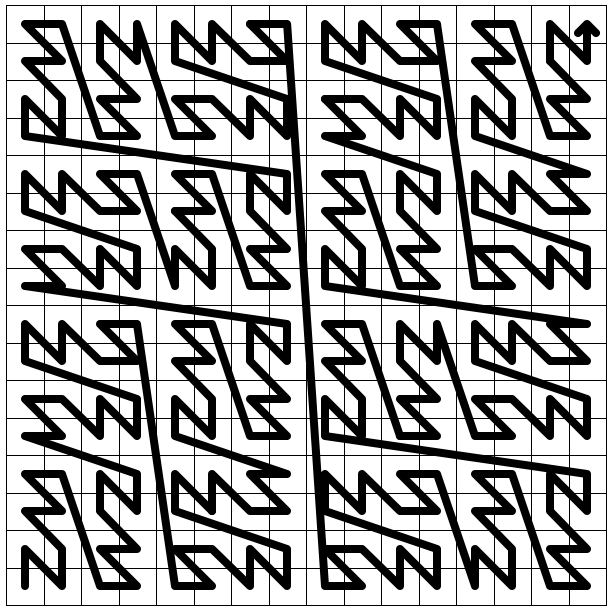}};
\end{subfigure}
\tikz[remember picture,overlay]\draw[line width=2pt,-stealth] ([xshift=5pt]a.east) -- ([xshift=30pt]a.east)node[midway,above,text=black,font=\LARGE\bfseries\sffamily] {};
\tikz[remember picture,overlay]\draw[line width=2pt,-stealth] ([xshift=5pt]b.east) -- ([xshift=30pt]b.east)node[midway,above,text=black,font=\LARGE\bfseries\sffamily] {};
\tikz[remember picture,overlay]\draw[line width=2pt,-stealth] ([xshift=5pt]c.east) -- ([xshift=30pt]c.east)node[midway,above,text=black,font=\LARGE\bfseries\sffamily] {};
\caption{First four approximants of $F_{tm}^A$}
\label{tm}
\end{figure}
\begin{figure}[H]
\centering
\begin{subfigure}[b]{0.45\textwidth}
\centering
\includegraphics[width=\textwidth]{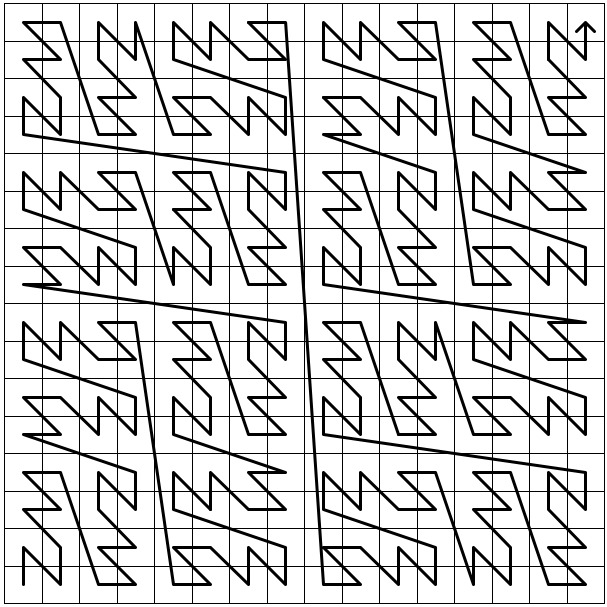}
\end{subfigure}
\hfill
\begin{subfigure}[b]{0.45\textwidth}
\centering
\includegraphics[width=\textwidth]{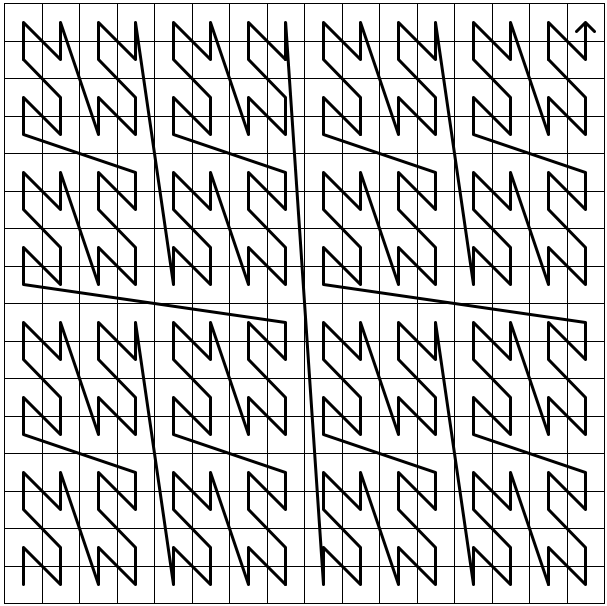}
\end{subfigure}
\caption{4th approximants of $F_{tm}^A$ and the Lebesgue curve, respectively}
\label{Lebesgue_2DTM_comparison}
\end{figure}
\end{example}

\begin{example}[Equithirds-variant]
Consider the substitution depicted in Figure \ref{Equithirds}. The substitution is a variation of \emph{Equithirds\ substitution} \cite{Web_Tiling_Encyclopedia}. It is defined over four tiles and their rotations, which are constructed by two different shapes, an equilateral triangle with side length $1$ and an isosceles triangle with side lengths $1,1,\sqrt{3}$. Its expansion factor is $\sqrt{3}$. 
The curves shown in Figure \ref{Equithirds_Orders} describe total orders over its 1-supertiles. Let $F^i_{eq}$ denote the space-filling curve produced by the algorithm from the tile with label $i$ for $i\in\{A^+,A^-,B^+,B^-\}$. First four approximants of $F^{A^+}_{eq}$ are shown in Figure \ref{equi2} and first four approximants of $F^{B^+}_{eq}$ are shown in Figure \ref{equi}. Observe that $F^{B^+}_{eq}(0)=F^{B^+}_{eq}(1)$. So, for illustration purposes, we modify the approximants of $F^{B^+}_{eq}$ to be closed curves. We connect the end points of its approximant curves with a straight line and fill the associated closed regions as demonstrated in Figure \ref{equi_filled_iterations} and Figure \ref{equi_filled_iterations__2}. 
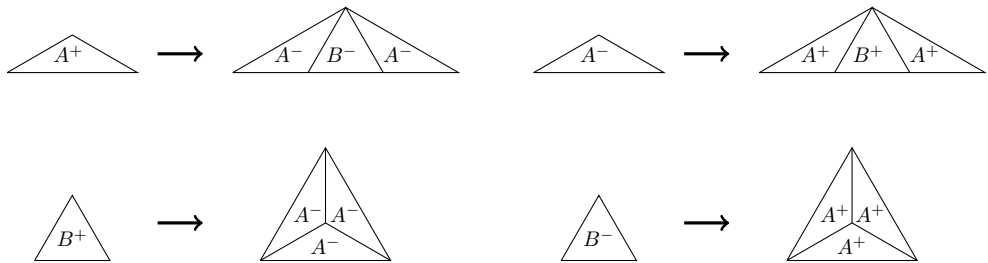
\begin{figure}[ht]
\centering
\begin{tikzpicture}
\centering
\foreach \inn in {0,7}{
\draw (0+\inn,0) -- (1.7320508075688772935274463415058723669428052538103806280558069794+\inn,0) -- (1.7320508075688772935274463415058723669428052538103806280558069794/2+\inn,0.5) -- (0+\inn,0);
\draw[very thick, ->] (2+\inn,0.25) -- (2.6+\inn,0.25);
\draw (3+\inn,0) -- (6+\inn,0) -- (3+3/2+\inn,1.7320508075688772935274463415058723669428052538103806280558069794/2) -- (3+\inn,0);
\draw (4+\inn,0) -- (3+3/2+\inn,1.7320508075688772935274463415058723669428052538103806280558069794/2) -- (5+\inn,0);
}
\node[scale=0.8] at (1.7320508075688772935274463415058723669428052538103806280558069794/2-0.05,0.25) {$A^+$};
\node[scale=0.8] at (1.7320508075688772935274463415058723669428052538103806280558069794/2-0.05+7,0.25) {$A^-$};

\node[scale=0.8] at (4.5-0.05,0.25) {$B^-$};
\node[scale=0.8] at (3.8-0.05,0.25) {$A^-$};
\node[scale=0.8] at (5.25-0.05,0.25) {$A^-$};
\node[scale=0.8] at (4.5-0.05+7,0.25) {$B^+$};
\node[scale=0.8] at (3.8-0.05+7,0.25) {$A^+$};
\node[scale=0.8] at (5.25-0.05+7,0.25) {$A^+$};

\foreach \inn in {0.7320508075688772935274463415058723669428052538103806280558069794/2,7+0.7320508075688772935274463415058723669428052538103806280558069794/2}{
\draw (0+\inn,-2.5) -- (1+\inn,-2.5) -- (0.5+\inn,1.7320508075688772935274463415058723669428052538103806280558069794/2-2.5) -- (0+\inn,0-2.5);
\draw[very thick, ->] (2-0.7320508075688772935274463415058723669428052538103806280558069794/2+\inn,0.5-2.5) -- (2.6-0.7320508075688772935274463415058723669428052538103806280558069794/2+\inn,0.5-2.5);
\draw (3+\inn,-2.5) -- (3+1.7320508075688772935274463415058723669428052538103806280558069794+\inn,-2.5) -- (3+1.7320508075688772935274463415058723669428052538103806280558069794/2+\inn,1.5-2.5) -- (3+\inn,-2.5);

\draw (3+\inn,-2.5) -- (3+1.7320508075688772935274463415058723669428052538103806280558069794/2+\inn,0.5-2.5);
\draw (3+1.7320508075688772935274463415058723669428052538103806280558069794+\inn,-2.5) -- (3+1.7320508075688772935274463415058723669428052538103806280558069794/2+\inn,0.5-2.5) -- (3+1.7320508075688772935274463415058723669428052538103806280558069794/2+\inn,1.5-2.5);
}

\node[scale=0.8] at (0.7320508075688772935274463415058723669428052538103806280558069794/2+0.5,0.3-2.5) {$B^+$};
\node[scale=0.8] at (0.7320508075688772935274463415058723669428052538103806280558069794/2+0.5+7,0.3-2.5) {$B^-$};

\node[scale=0.8] at (3+1.7320508075688772935274463415058723669428052538103806280558069794/2+0.7320508075688772935274463415058723669428052538103806280558069794/2,0.2-2.5) {$A^-$};
\node[scale=0.8] at (3+1.7320508075688772935274463415058723669428052538103806280558069794/2+0.7320508075688772935274463415058723669428052538103806280558069794/2+7,0.2-2.5) {$A^+$};

\node[scale=0.8] at (3+1.7320508075688772935274463415058723669428052538103806280558069794/2+0.14,0.63-2.5) {$A^-$};
\node[scale=0.8] at (3+1.7320508075688772935274463415058723669428052538103806280558069794/2+0.63,0.63-2.5) {$A^-$};
\node[scale=0.8] at (7+3+1.7320508075688772935274463415058723669428052538103806280558069794/2+0.14,0.63-2.5) {$A^+$};
\node[scale=0.8] at (7+3+1.7320508075688772935274463415058723669428052538103806280558069794/2+0.61,0.63-2.5) {$A^+$};
\end{tikzpicture}
\caption{Equithirds-variant}
\label{Equithirds}
\end{figure}
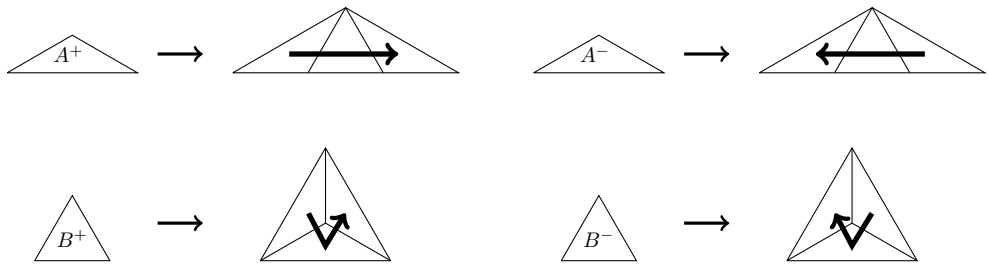
\begin{figure}[H]
\centering
\begin{tikzpicture}
\centering
\foreach \inn in {0,7}{
\draw (0+\inn,0) -- (1.7320508075688772935274463415058723669428052538103806280558069794+\inn,0) -- (1.7320508075688772935274463415058723669428052538103806280558069794/2+\inn,0.5) -- (0+\inn,0);
\draw[very thick, ->] (2+\inn,0.25) -- (2.6+\inn,0.25);
\draw (3+\inn,0) -- (6+\inn,0) -- (3+3/2+\inn,1.7320508075688772935274463415058723669428052538103806280558069794/2) -- (3+\inn,0);
\draw (4+\inn,0) -- (3+3/2+\inn,1.7320508075688772935274463415058723669428052538103806280558069794/2) -- (5+\inn,0);
}
\foreach \inn in {0.7320508075688772935274463415058723669428052538103806280558069794/2,7+0.7320508075688772935274463415058723669428052538103806280558069794/2}{
\draw (0+\inn,-2.5) -- (1+\inn,-2.5) -- (0.5+\inn,1.7320508075688772935274463415058723669428052538103806280558069794/2-2.5) -- (0+\inn,0-2.5);
\draw[very thick, ->] (2-0.7320508075688772935274463415058723669428052538103806280558069794/2+\inn,0.5-2.5) -- (2.6-0.7320508075688772935274463415058723669428052538103806280558069794/2+\inn,0.5-2.5);
\draw (3+\inn,-2.5) -- (3+1.7320508075688772935274463415058723669428052538103806280558069794+\inn,-2.5) -- (3+1.7320508075688772935274463415058723669428052538103806280558069794/2+\inn,1.5-2.5) -- (3+\inn,-2.5);

\draw (3+\inn,-2.5) -- (3+1.7320508075688772935274463415058723669428052538103806280558069794/2+\inn,0.5-2.5);
\draw (3+1.7320508075688772935274463415058723669428052538103806280558069794+\inn,-2.5) -- (3+1.7320508075688772935274463415058723669428052538103806280558069794/2+\inn,0.5-2.5) -- (3+1.7320508075688772935274463415058723669428052538103806280558069794/2+\inn,1.5-2.5);
}

\draw[line width=2pt, ->] (3+1.7320508075688772935274463415058723669428052538103806280558069794/2+0.14,0.63-2.5) -- (3+1.7320508075688772935274463415058723669428052538103806280558069794/2+0.7320508075688772935274463415058723669428052538103806280558069794/2,0.2-2.5) -- (3+1.7320508075688772935274463415058723669428052538103806280558069794/2+0.63,0.63-2.5);

\draw[line width=2pt, <-] (7+3+1.7320508075688772935274463415058723669428052538103806280558069794/2+0.14,0.63-2.5) -- (7+3+1.7320508075688772935274463415058723669428052538103806280558069794/2+0.7320508075688772935274463415058723669428052538103806280558069794/2,0.2-2.5) -- (7+3+1.7320508075688772935274463415058723669428052538103806280558069794/2+0.63,0.63-2.5);

\draw[line width=2pt, ->] (3.8-0.05,0.25) -- (4.5-0.05,0.25) -- (5.25-0.05,0.25);
\draw[line width=2pt, <-] (7+3.8-0.05,0.25) -- (7+4.5-0.05,0.25) -- (7+5.25-0.05,0.25);

\node[scale=0.8] at (1.7320508075688772935274463415058723669428052538103806280558069794/2-0.05,0.25) {$A^+$};
\node[scale=0.8] at (1.7320508075688772935274463415058723669428052538103806280558069794/2-0.05+7,0.25) {$A^-$};
\node[scale=0.8] at (0.7320508075688772935274463415058723669428052538103806280558069794/2+0.5,0.3-2.5) {$B^+$};
\node[scale=0.8] at (0.7320508075688772935274463415058723669428052538103806280558069794/2+0.5+7,0.3-2.5) {$B^-$};

\end{tikzpicture}
\caption{Total orders over the 1-supertiles of Equithirds-variant}
\label{Equithirds_Orders}
\end{figure}
\begin{figure}[ht]
\centering
\begin{subfigure}[b]{0.18\textwidth}
\centering
\tikz[remember picture]\node[inner sep=0pt,outer sep=0pt] (a){\includegraphics[width=\linewidth]{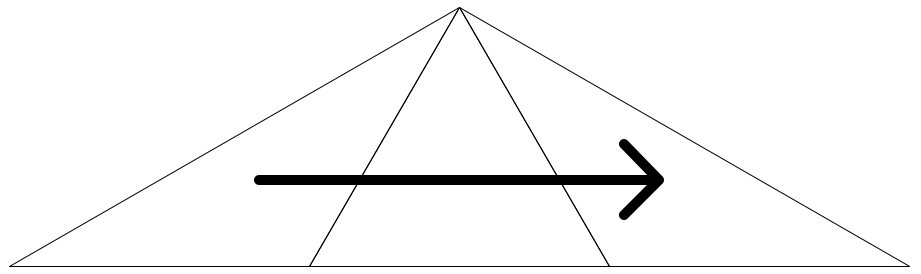}};
\end{subfigure}
\hfill
\begin{subfigure}[b]{0.18\textwidth}
\centering
\tikz[remember picture]\node[inner sep=0pt,outer sep=0pt] (b){\includegraphics[width=\linewidth]{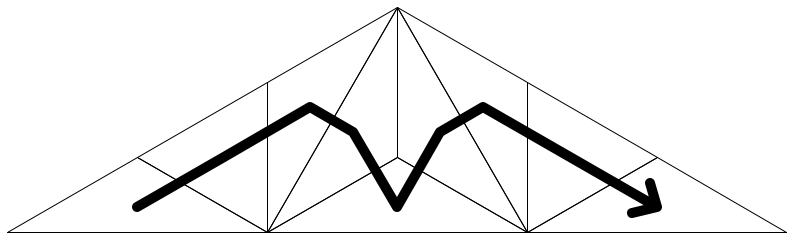}};
\end{subfigure}
\hfill
\begin{subfigure}[b]{0.18\textwidth}
\centering
\tikz[remember picture]\node[inner sep=0pt,outer sep=0pt] (c){\includegraphics[width=\linewidth]{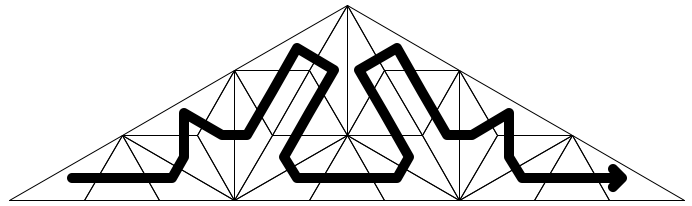}};
\end{subfigure}
\hfill
\begin{subfigure}[b]{0.18\textwidth}
\centering
\tikz[remember picture]\node[inner sep=0pt,outer sep=0pt] (d){\includegraphics[width=\linewidth]{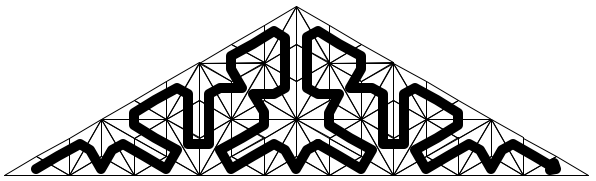}};
\end{subfigure}
\tikz[remember picture,overlay]\draw[line width=2pt,-stealth] ([xshift=5pt]a.east) -- ([xshift=30pt]a.east)node[midway,above,text=black,font=\LARGE\bfseries\sffamily] {};
\tikz[remember picture,overlay]\draw[line width=2pt,-stealth] ([xshift=5pt]b.east) -- ([xshift=30pt]b.east)node[midway,above,text=black,font=\LARGE\bfseries\sffamily] {};
\tikz[remember picture,overlay]\draw[line width=2pt,-stealth] ([xshift=5pt]c.east) -- ([xshift=30pt]c.east)node[midway,above,text=black,font=\LARGE\bfseries\sffamily] {};
\caption{First four approximants of $F^{A^+}_{eq}$}
\label{equi2}
\end{figure}
\begin{figure}[ht]
\centering
\begin{subfigure}[b]{0.18\textwidth}
\centering
\tikz[remember picture]\node[inner sep=0pt,outer sep=0pt] (a){\includegraphics[width=\linewidth]{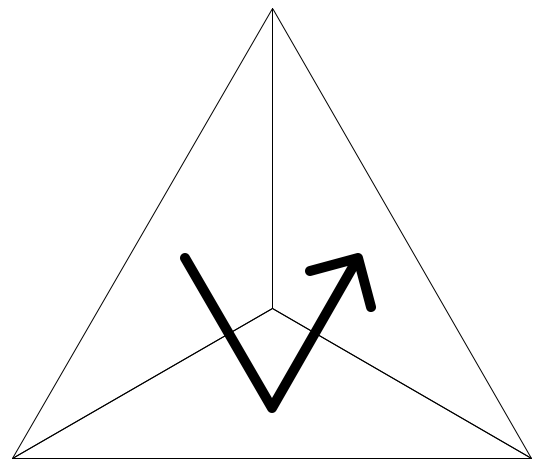}};
\end{subfigure}
\hfill
\begin{subfigure}[b]{0.18\textwidth}
\centering
\tikz[remember picture]\node[inner sep=0pt,outer sep=0pt] (b){\includegraphics[width=\linewidth]{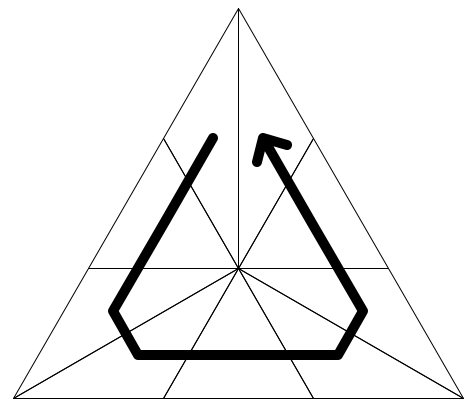}};
\end{subfigure}
\hfill
\begin{subfigure}[b]{0.18\textwidth}
\centering
\tikz[remember picture]\node[inner sep=0pt,outer sep=0pt] (c){\includegraphics[width=\linewidth]{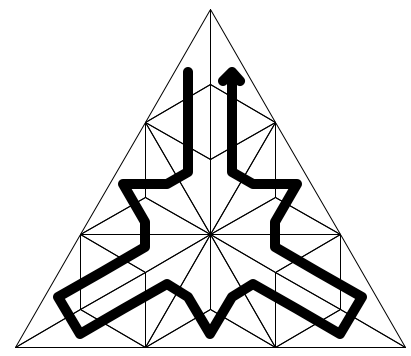}};
\end{subfigure}
\hfill
\begin{subfigure}[b]{0.18\textwidth}
\centering
\tikz[remember picture]\node[inner sep=0pt,outer sep=0pt] (d){\includegraphics[width=\linewidth]{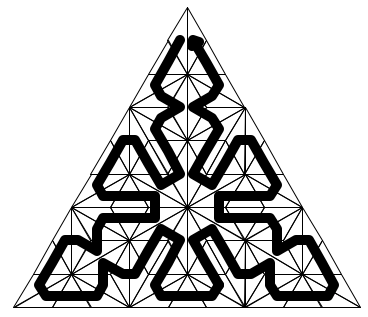}};
\end{subfigure}
\tikz[remember picture,overlay]\draw[line width=2pt,-stealth] ([xshift=5pt]a.east) -- ([xshift=30pt]a.east)node[midway,above,text=black,font=\LARGE\bfseries\sffamily] {};
\tikz[remember picture,overlay]\draw[line width=2pt,-stealth] ([xshift=5pt]b.east) -- ([xshift=30pt]b.east)node[midway,above,text=black,font=\LARGE\bfseries\sffamily] {};
\tikz[remember picture,overlay]\draw[line width=2pt,-stealth] ([xshift=5pt]c.east) -- ([xshift=30pt]c.east)node[midway,above,text=black,font=\LARGE\bfseries\sffamily] {};
\caption{First four approximants of $F^{B^+}_{eq}$}
\label{equi}
\end{figure}
\begin{figure}[H]
\centering
\begin{subfigure}[b]{0.24\textwidth}
\centering
\tikz[remember picture]\node[inner sep=0pt,outer sep=0pt] (a){\includegraphics[width=\linewidth]{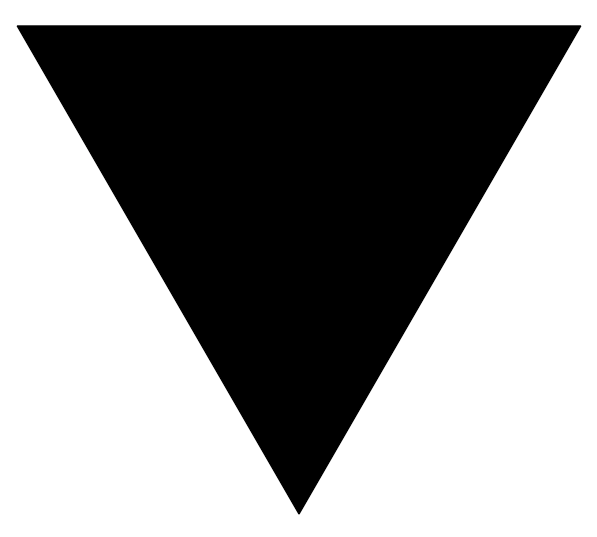}};
\end{subfigure}
\hfill
\begin{subfigure}[b]{0.24\textwidth}
\centering
\tikz[remember picture]\node[inner sep=0pt,outer sep=0pt] (b){\includegraphics[width=\linewidth]{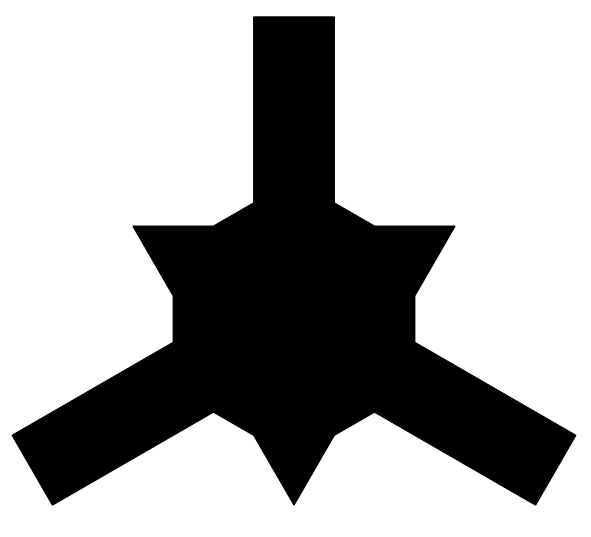}};
\end{subfigure}
\hfill
\begin{subfigure}[b]{0.24\textwidth}
\centering
\tikz[remember picture]\node[inner sep=0pt,outer sep=0pt] (c){\includegraphics[width=\linewidth]{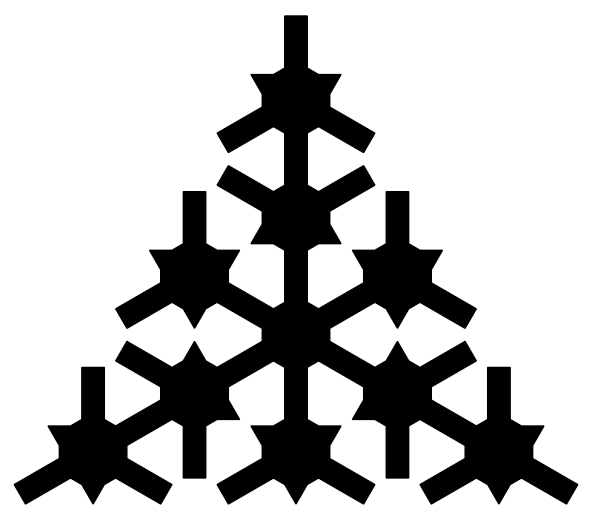}};
\end{subfigure}
\hfill
\tikz[remember picture,overlay]\draw[line width=2pt,-stealth] ([xshift=5pt]a.east) -- ([xshift=30pt]a.east)node[midway,above,text=black,font=\LARGE\bfseries\sffamily] {};
\tikz[remember picture,overlay]\draw[line width=2pt,-stealth] ([xshift=5pt]b.east) -- ([xshift=30pt]b.east)node[midway,above,text=black,font=\LARGE\bfseries\sffamily] {};
\tikz[remember picture,overlay]\draw[line width=2pt,-stealth] ([xshift=5pt]c.east) -- ([xshift=30pt]c.east)node[midway,above,text=black,font=\LARGE\bfseries\sffamily] {};
\begin{subfigure}[b]{\textwidth}
\centering
%\tikz[remember picture]\node[inner sep=0pt,outer sep=0pt] (d){\includegraphics[width=\linewidth]{Figures/EQ_filled_7.png}};
\tikz[remember picture]\node[inner sep=0pt,outer sep=0pt] (d){\includegraphics[scale=0.6]{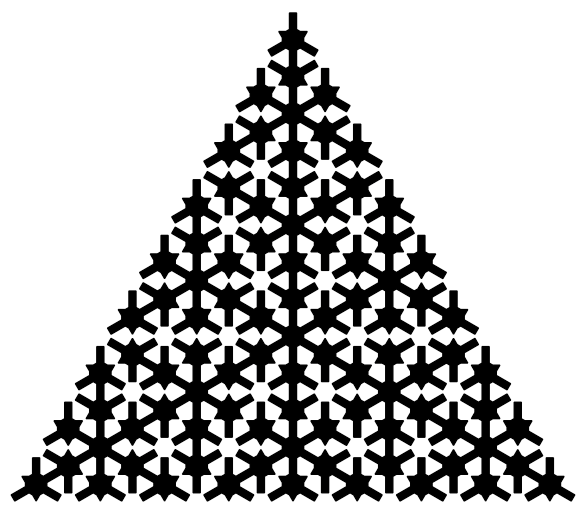}};
\end{subfigure}
\caption{1st, 3rd, 5th and 7th approximants of $F^{B^+}_{eq}$ where the end points of the approximants are joined with a line and the associated closed regions are filled. The 7th approximant is scaled up for illustration purposes.}
\label{equi_filled_iterations}
\end{figure}
\begin{figure}[H]
\centering
\begin{subfigure}[b]{0.24\textwidth}
\centering
\tikz[remember picture]\node[inner sep=0pt,outer sep=0pt] (a){\includegraphics[width=\linewidth]{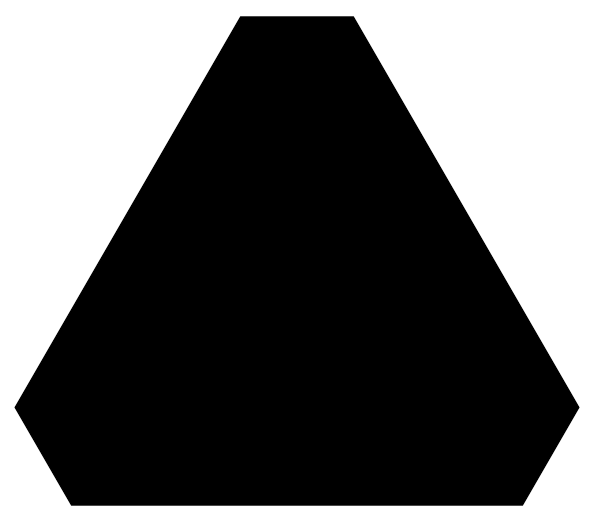}};
\end{subfigure}
\hfill
\begin{subfigure}[b]{0.24\textwidth}
\centering
\tikz[remember picture]\node[inner sep=0pt,outer sep=0pt] (b){\includegraphics[width=\linewidth]{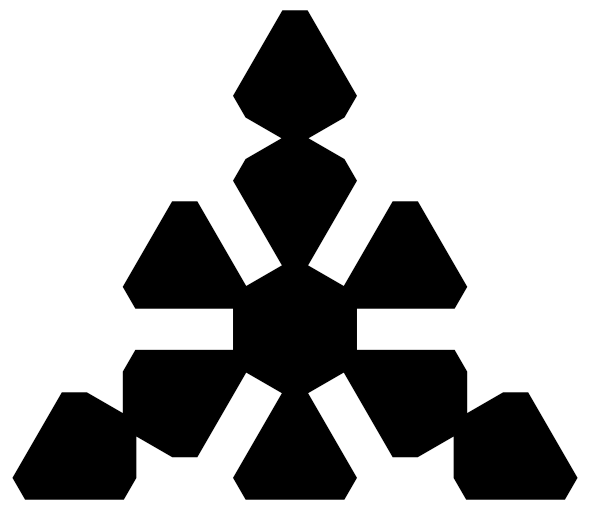}};
\end{subfigure}
\hfill
\begin{subfigure}[b]{0.24\textwidth}
\centering
\tikz[remember picture]\node[inner sep=0pt,outer sep=0pt] (c){\includegraphics[width=\linewidth]{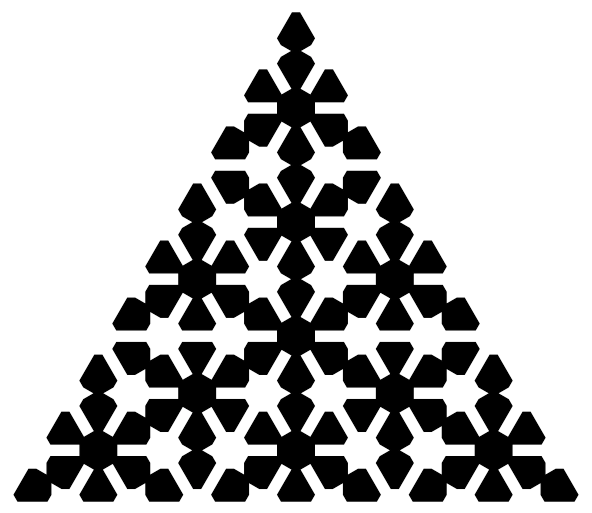}};
\end{subfigure}
\hfill
\tikz[remember picture,overlay]\draw[line width=2pt,-stealth] ([xshift=5pt]a.east) -- ([xshift=30pt]a.east)node[midway,above,text=black,font=\LARGE\bfseries\sffamily] {};
\tikz[remember picture,overlay]\draw[line width=2pt,-stealth] ([xshift=5pt]b.east) -- ([xshift=30pt]b.east)node[midway,above,text=black,font=\LARGE\bfseries\sffamily] {};
\tikz[remember picture,overlay]\draw[line width=2pt,-stealth] ([xshift=5pt]c.east) -- ([xshift=30pt]c.east)node[midway,above,text=black,font=\LARGE\bfseries\sffamily] {};
\begin{subfigure}[b]{\textwidth}
\centering
%\tikz[remember picture]\node[inner sep=0pt,outer sep=0pt] (d){\includegraphics[width=\linewidth]{EQ_filled_8.png}};
\tikz[remember picture]\node[inner sep=0pt,outer sep=0pt] (d){\includegraphics[scale=0.6]{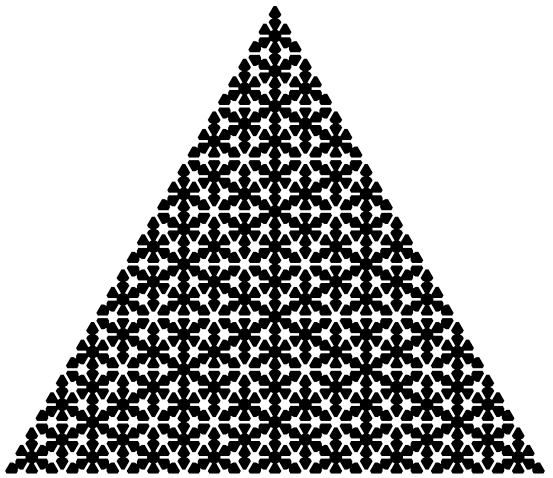}};
\end{subfigure}
\caption{2nd, 4th, 6th and 8th approximants of $F^{B^+}_{eq}$ where the end points of the approximants are joined with a line and the associated closed regions are filled. The 8th approximant is scaled up for illustration purposes.}
\label{equi_filled_iterations__2}
\end{figure}

%\begin{figure}[H]
%\centering
%\begin{subfigure}[b]{0.18\textwidth}
%\centering
%\tikz[remember picture]\node[inner sep=0pt,outer sep=0pt] (a){\includegraphics[width=\linewidth]{Figures/EQ_filled_2.png}};
%\end{subfigure}
%\hfill
%\begin{subfigure}[b]{0.18\textwidth}
%\centering
%\tikz[remember picture]\node[inner sep=0pt,outer sep=0pt] (b){\includegraphics[width=\linewidth]{Figures/EQ_filled_4.png}};
%\end{subfigure}
%\hfill
%\begin{subfigure}[b]{0.18\textwidth}
%\centering
%\tikz[remember picture]\node[inner sep=0pt,outer sep=0pt] (c){\includegraphics[width=\linewidth]{Figures/EQ_filled_6.png}};
%\end{subfigure}
%\hfill
%\begin{subfigure}[b]{0.18\textwidth}
%\centering
%\tikz[remember picture]\node[inner sep=0pt,outer sep=0pt] (d){\includegraphics[width=\linewidth]{Figures/EQ_filled_8.png}};
%\end{subfigure}
%\tikz[remember picture,overlay]\draw[line width=2pt,-stealth] ([xshift=5pt]a.east) -- ([xshift=30pt]a.east)node[midway,above,text=black,font=\LARGE\bfseries\sffamily] {};
%\tikz[remember picture,overlay]\draw[line width=2pt,-stealth] ([xshift=5pt]b.east) -- ([xshift=30pt]b.east)node[midway,above,text=black,font=\LARGE\bfseries\sffamily] {};
%\tikz[remember picture,overlay]\draw[line width=2pt,-stealth] ([xshift=5pt]c.east) -- ([xshift=30pt]c.east)node[midway,above,text=black,font=\LARGE\bfseries\sffamily] {};
%\caption{2nd, 4th, 6th and 8th approximants of $F^{B^+}_{eq}$ where the end points of the approximants are joined with a line and the associated closed regions are filled.}
%\label{equi_filled_iterations__2}
%\end{figure}

Next we describe the geometry of approximants of $F^{A^+}_{eq}$. Notice that $F^{A^+}_{eq}(0)\neq F^{A^+}_{eq}(1)$. Let $p_{A^+}^r$ denote the rotated version of $p_{A^+}$ by $\pi$ such that their longer edges merge. Also, denote
the space filling curves generated by these two tiles by $F^{A^+}_{eq}, F^{A^+_r}_{eq}$. Since $F^{A^+}_{eq}(0)=F^{A^+_r}_{eq}(1)$ and $F^{A^+}_{eq}(1)=F^{A^+_r}_{eq}(0)$, we can concatenate $F^{A^+}_{eq}$ with $F^{A^+_r}_{eq}$ in order to define another space-filling curve $F^{A}_{eq}$ so that $F^{A}_{eq}(0)=F^{A}_{eq}(1)$. The geometry of approximants of $F^{A^+}_{eq}$ will be visualised through approximants of $F^{A}_{eq}$ since we can modify the approximants of  $F^{A}_{eq}$ to be closed curves as shown in Figure \ref{equi2_0}. %First four approximants of $F_{A^+}^*$ are given in Figure \ref{equi2_0}. 
The associated filled regions of the first 8 approximant curves are shown in Figure \ref{equi_filled_iterations_second} and Figure \ref{equi_filled_iterations__2_second}.

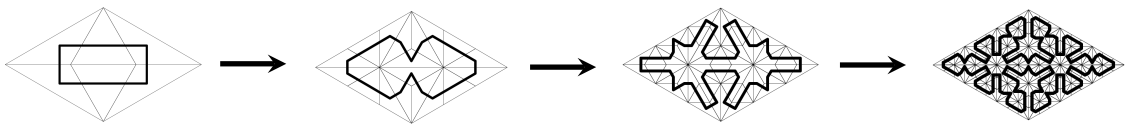
\begin{figure}[H]
\centering
\begin{subfigure}[b]{0.18\textwidth}
\centering
\tikz[remember picture]\node[inner sep=0pt,outer sep=0pt] (a){\includegraphics[width=\linewidth]{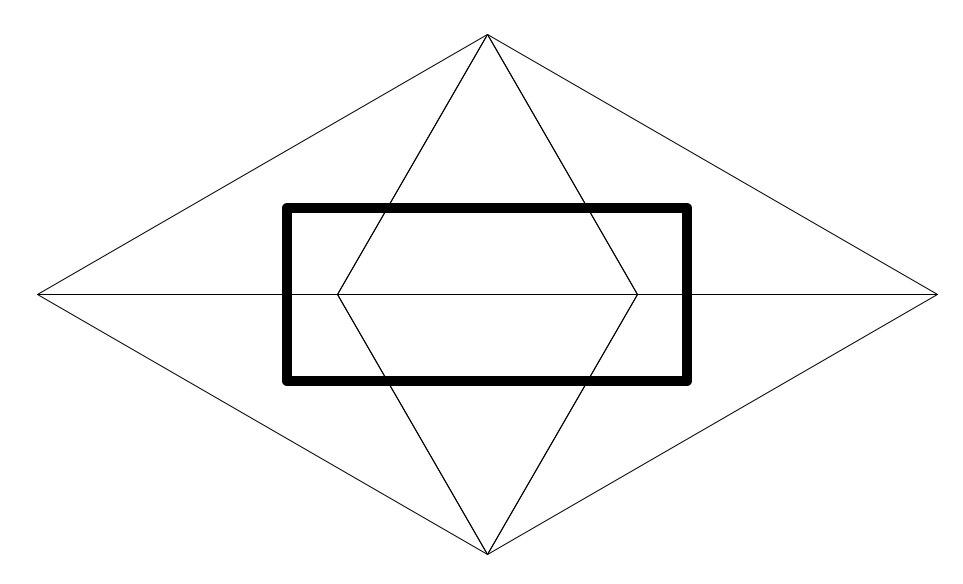}};
\end{subfigure}
\hfill
\begin{subfigure}[b]{0.18\textwidth}
\centering
\tikz[remember picture]\node[inner sep=0pt,outer sep=0pt] (b){\includegraphics[width=\linewidth]{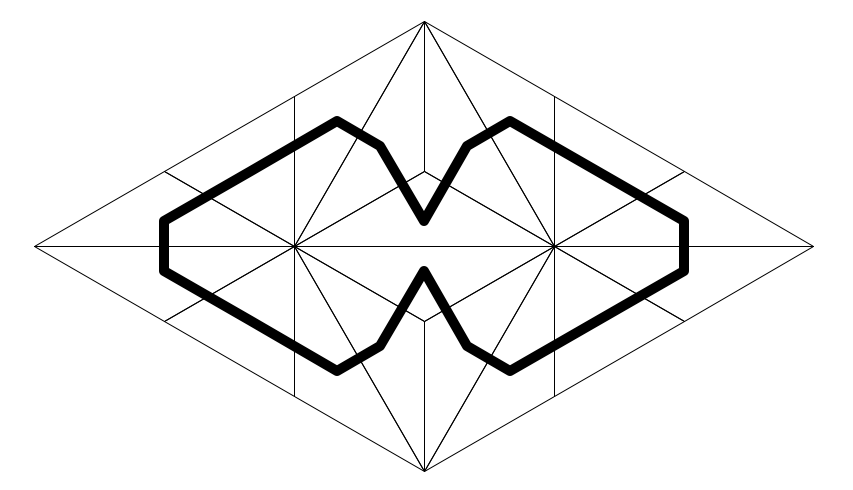}};
\end{subfigure}
\hfill
\begin{subfigure}[b]{0.18\textwidth}
\centering
\tikz[remember picture]\node[inner sep=0pt,outer sep=0pt] (c){\includegraphics[width=\linewidth]{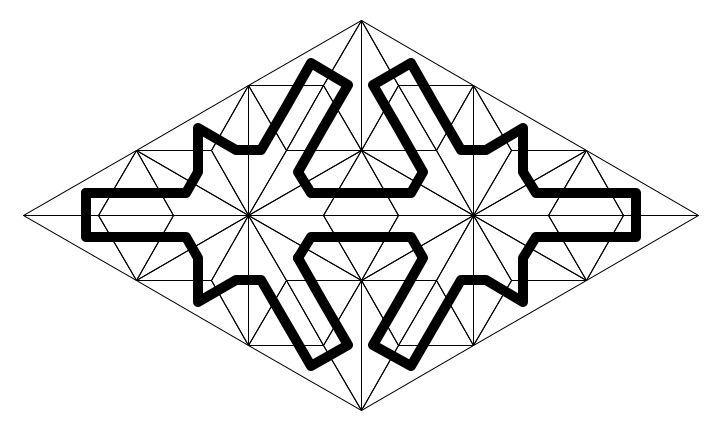}};
\end{subfigure}
\hfill
\begin{subfigure}[b]{0.18\textwidth}
\centering
\tikz[remember picture]\node[inner sep=0pt,outer sep=0pt] (d){\includegraphics[width=\linewidth]{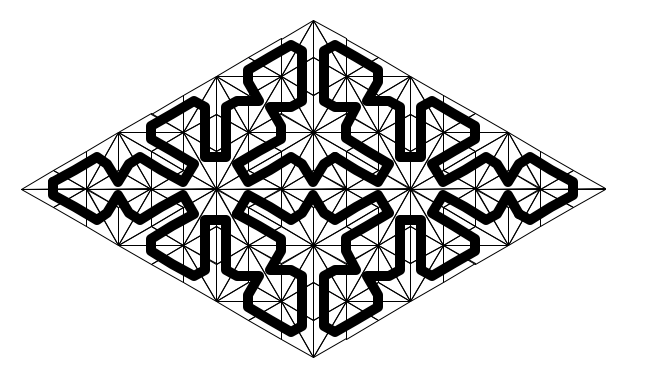}};
\end{subfigure}
\tikz[remember picture,overlay]\draw[line width=2pt,-stealth] ([xshift=5pt]a.east) -- ([xshift=30pt]a.east)node[midway,above,text=black,font=\LARGE\bfseries\sffamily] {};
\tikz[remember picture,overlay]\draw[line width=2pt,-stealth] ([xshift=5pt]b.east) -- ([xshift=30pt]b.east)node[midway,above,text=black,font=\LARGE\bfseries\sffamily] {};
\tikz[remember picture,overlay]\draw[line width=2pt,-stealth] ([xshift=5pt]c.east) -- ([xshift=30pt]c.east)node[midway,above,text=black,font=\LARGE\bfseries\sffamily] {};
\caption{First four approximants of $F^{A}_{eq}$ where their end points are joined with a line.}
\label{equi2_0}
\end{figure}

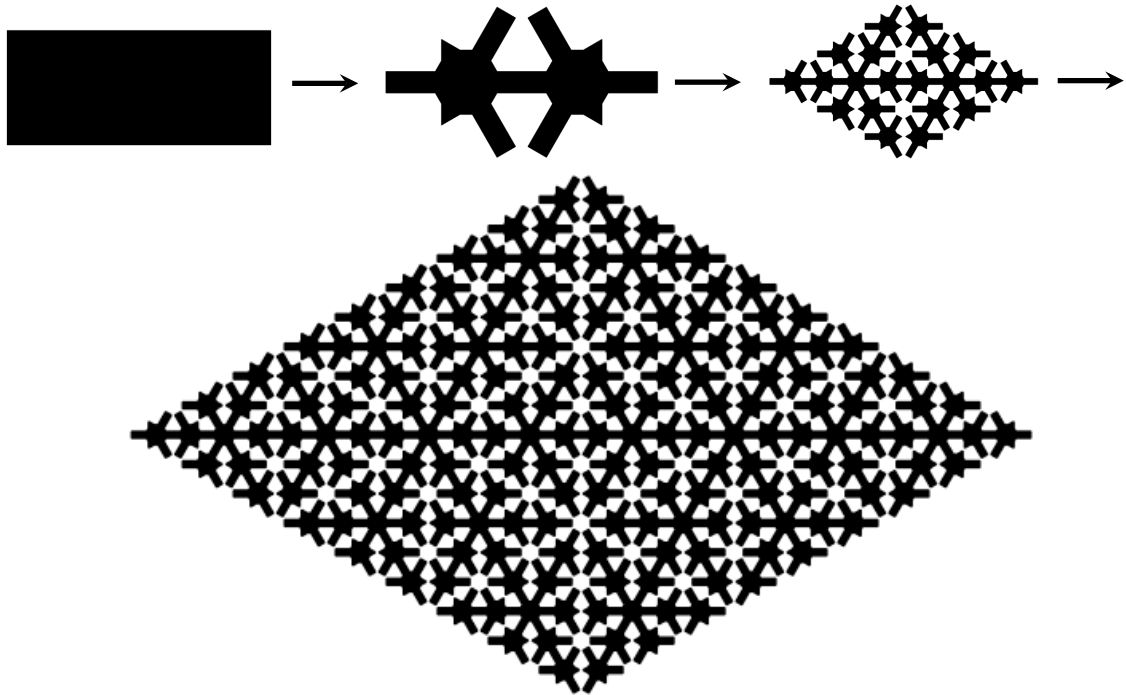
\begin{figure}[H]
\centering
\begin{subfigure}[b]{0.24\textwidth}
\centering
\tikz[remember picture]\node[inner sep=0pt,outer sep=0pt] (a){\includegraphics[width=\linewidth]{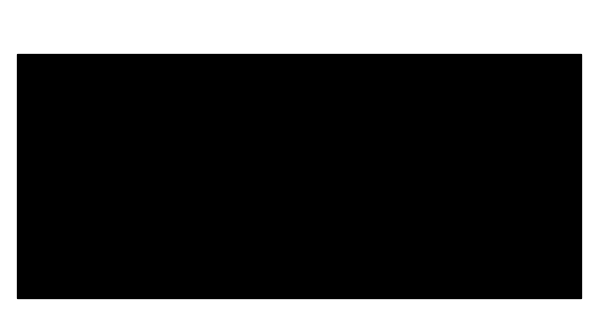}};
\end{subfigure}
\hfill
\begin{subfigure}[b]{0.24\textwidth}
\centering
\tikz[remember picture]\node[inner sep=0pt,outer sep=0pt] (b){\includegraphics[width=\linewidth]{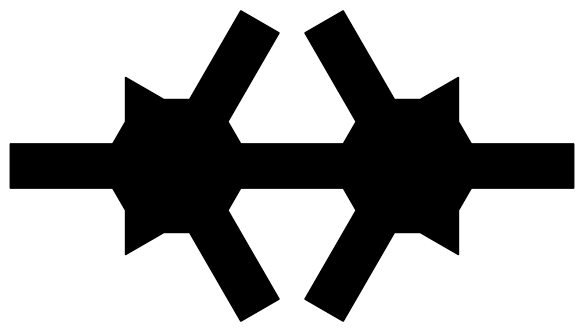}};
\end{subfigure}
\hfill
\begin{subfigure}[b]{0.24\textwidth}
\centering
\tikz[remember picture]\node[inner sep=0pt,outer sep=0pt] (c){\includegraphics[width=\linewidth]{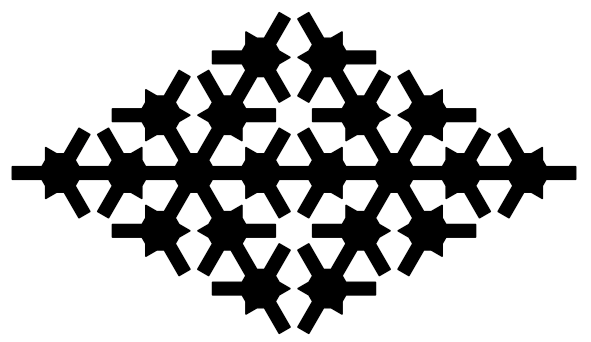}};
\end{subfigure}
\hfill
\tikz[remember picture,overlay]\draw[line width=2pt,-stealth] ([xshift=5pt]a.east) -- ([xshift=30pt]a.east)node[midway,below,text=black,font=\LARGE\bfseries\sffamily] {};
\tikz[remember picture,overlay]\draw[line width=2pt,-stealth] ([xshift=5pt]b.east) -- ([xshift=30pt]b.east)node[midway,above,text=black,font=\LARGE\bfseries\sffamily] {};
\tikz[remember picture,overlay]\draw[line width=2pt,-stealth] ([xshift=5pt]c.east) -- ([xshift=30pt]c.east)node[midway,above,text=black,font=\LARGE\bfseries\sffamily] {};
\begin{subfigure}[b]{\textwidth}
\centering
%\tikz[remember picture]\node[inner sep=0pt,outer sep=0pt] (d){\includegraphics[width=\linewidth]{EQ_Sec_7_filled.png}};
\tikz[remember picture]\node[inner sep=0pt,outer sep=0pt] (d){\includegraphics[scale=0.6]{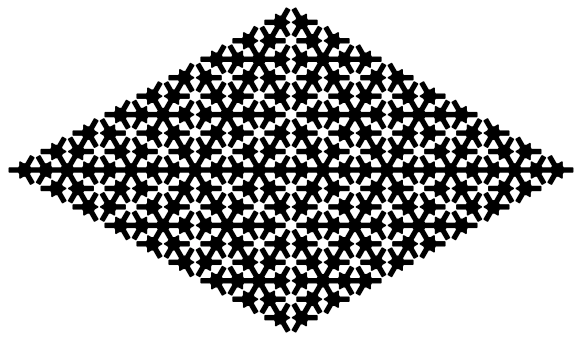}};
\end{subfigure}
\caption{Filled versions of 1st, 3rd, 5th and 7th approximants of $F^{A}_{eq}$. The 7th approximant is scaled up for demonstration purposes.}
\label{equi_filled_iterations_second}
\end{figure}

%\begin{figure}[H]
%\centering
%\begin{subfigure}[b]{0.18\textwidth}
%\centering
%\tikz[remember picture]\node[inner sep=0pt,outer sep=0pt] (a){\includegraphics[width=\linewidth]{Figures/EQ_Sec_1_filled.png}};
%\end{subfigure}
%\hfill
%\begin{subfigure}[b]{0.18\textwidth}
%\centering
%\tikz[remember picture]\node[inner sep=0pt,outer sep=0pt] (b){\includegraphics[width=\linewidth]{Figures/EQ_Sec_3_filled.png}};
%\end{subfigure}
%\hfill
%\begin{subfigure}[b]{0.18\textwidth}
%\centering
%\tikz[remember picture]\node[inner sep=0pt,outer sep=0pt] (c){\includegraphics[width=\linewidth]{Figures/EQ_Sec_5_filled.png}};
%\end{subfigure}
%\hfill
%\begin{subfigure}[b]{0.18\textwidth}
%\centering
%\tikz[remember picture]\node[inner sep=0pt,outer sep=0pt] (d){\includegraphics[width=\linewidth]{Figures/EQ_Sec_7_filled.png}};
%\end{subfigure}
%\tikz[remember picture,overlay]\draw[line width=2pt,-stealth] ([xshift=5pt]a.east) -- ([xshift=30pt]a.east)node[midway,below,text=black,font=\LARGE\bfseries\sffamily] {};
%\tikz[remember picture,overlay]\draw[line width=2pt,-stealth] ([xshift=5pt]b.east) -- ([xshift=30pt]b.east)node[midway,above,text=black,font=\LARGE\bfseries\sffamily] {};
%\tikz[remember picture,overlay]\draw[line width=2pt,-stealth] ([xshift=5pt]c.east) -- ([xshift=30pt]c.east)node[midway,above,text=black,font=\LARGE\bfseries\sffamily] {};
%\caption{1st, 3rd, 5th and 7th approximants of $F^{A}_{eq}$ where the end points of the approximants are joined with a line and the associated closed regions are filled.}
%\label{equi_filled_iterations_second}
%\end{figure}

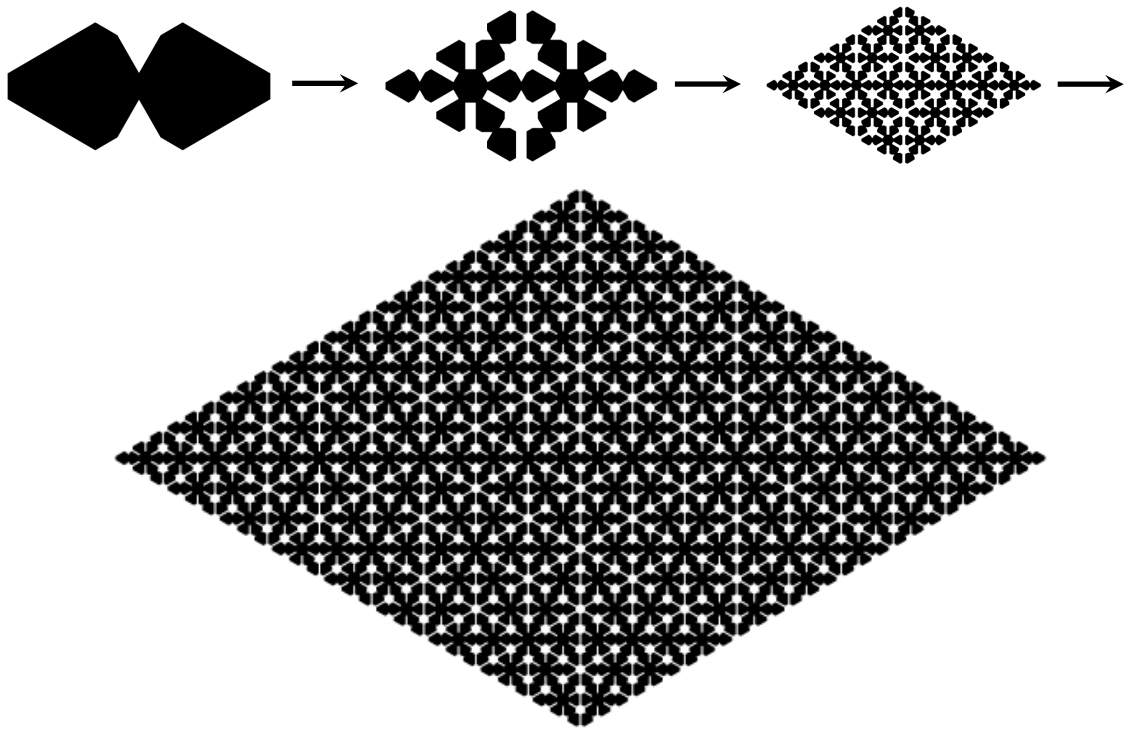
\begin{figure}[H]
\centering
\begin{subfigure}[b]{0.24\textwidth}
\centering
\tikz[remember picture]\node[inner sep=0pt,outer sep=0pt] (a){\includegraphics[width=\linewidth]{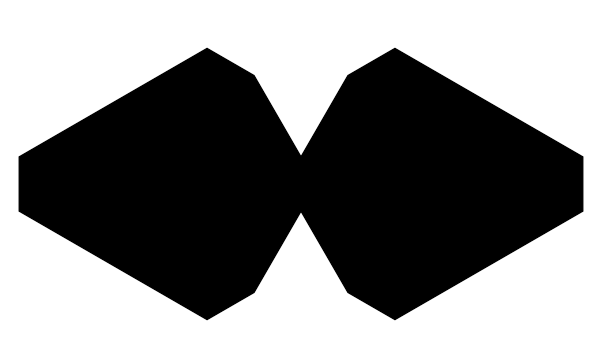}};
\end{subfigure}
\hfill
\begin{subfigure}[b]{0.24\textwidth}
\centering
\tikz[remember picture]\node[inner sep=0pt,outer sep=0pt] (b){\includegraphics[width=\linewidth]{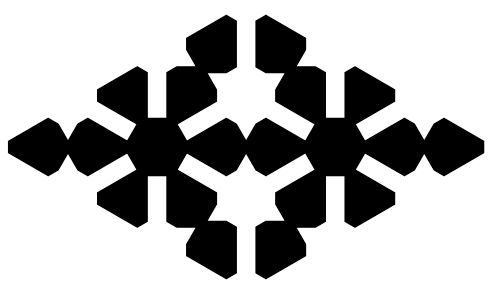}};
\end{subfigure}
\hfill
\begin{subfigure}[b]{0.24\textwidth}
\centering
\tikz[remember picture]\node[inner sep=0pt,outer sep=0pt] (c){\includegraphics[width=\linewidth]{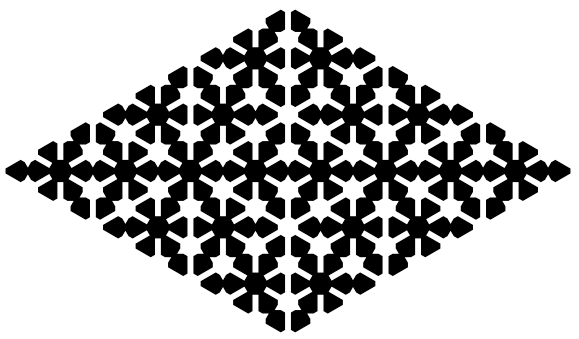}};
\end{subfigure}
\hfill
\tikz[remember picture,overlay]\draw[line width=2pt,-stealth] ([xshift=5pt]a.east) -- ([xshift=30pt]a.east)node[midway,below,text=black,font=\LARGE\bfseries\sffamily] {};
\tikz[remember picture,overlay]\draw[line width=2pt,-stealth] ([xshift=5pt]b.east) -- ([xshift=30pt]b.east)node[midway,above,text=black,font=\LARGE\bfseries\sffamily] {};
\tikz[remember picture,overlay]\draw[line width=2pt,-stealth] ([xshift=5pt]c.east) -- ([xshift=30pt]c.east)node[midway,above,text=black,font=\LARGE\bfseries\sffamily] {};
\begin{subfigure}[b]{\textwidth}
\centering
%\tikz[remember picture]\node[inner sep=0pt,outer sep=0pt] (d){\includegraphics[width=\linewidth]{EQ_Sec_8_filled.png}};
\tikz[remember picture]\node[inner sep=0pt,outer sep=0pt] (d){\includegraphics[scale=0.62]{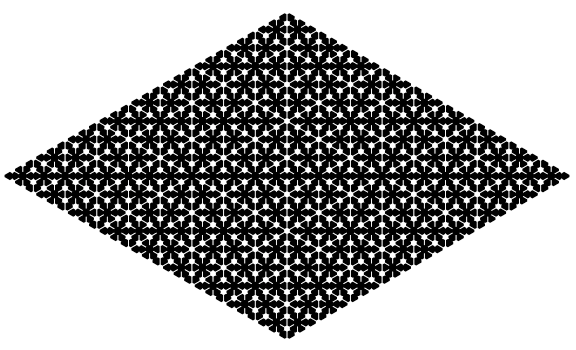}};
\end{subfigure}
\caption{Filled versions of 2nd, 4th, 6th and 8th approximants of $F^{A}_{eq}$. The 8th approximant is scaled up for demonstration purposes.}
\label{equi_filled_iterations__2_second}
\end{figure}

%\begin{figure}[H]
%\centering
%\begin{subfigure}[b]{0.18\textwidth}
%\centering
%\tikz[remember picture]\node[inner sep=0pt,outer sep=0pt] (a){\includegraphics[width=\linewidth]{Figures/EQ_Sec_2_filled.png}};
%\end{subfigure}
%\hfill
%\begin{subfigure}[b]{0.18\textwidth}
%\centering
%\tikz[remember picture]\node[inner sep=0pt,outer sep=0pt] (b){\includegraphics[width=\linewidth]{Figures/EQ_Sec_4_filled.png}};
%\end{subfigure}
%\hfill
%\begin{subfigure}[b]{0.18\textwidth}
%\centering
%\tikz[remember picture]\node[inner sep=0pt,outer sep=0pt] (c){\includegraphics[width=\linewidth]{EQ_Sec_6_filled.png}};
%\end{subfigure}
%\hfill
%\begin{subfigure}[b]{0.18\textwidth}
%\centering
%\tikz[remember picture]\node[inner sep=0pt,outer sep=0pt] (d){\includegraphics[width=\linewidth]{EQ_Sec_8_filled.png}};
%\end{subfigure}
%\tikz[remember picture,overlay]\draw[line width=2pt,-stealth] ([xshift=5pt]a.east) -- ([xshift=30pt]a.east)node[midway,below,text=black,font=\LARGE\bfseries\sffamily] {};
%\tikz[remember picture,overlay]\draw[line width=2pt,-stealth] ([xshift=5pt]b.east) -- ([xshift=30pt]b.east)node[midway,above,text=black,font=\LARGE\bfseries\sffamily] {};
%\tikz[remember picture,overlay]\draw[line width=2pt,-stealth] ([xshift=5pt]c.east) -- ([xshift=30pt]c.east)node[midway,above,text=black,font=\LARGE\bfseries\sffamily] {};
%\caption{2nd, 4th, 6th and 8th approximants of $F^{A}_{eq}$ where the end points of the approximants are joined with a line and the associated closed regions are filled.}
%\label{equi_filled_iterations__2_second}
%\end{figure}
\end{example}

For the following two substitution examples, we will use the algorithm (Theorem \ref{main_theorem}) to create space filling curves over patches instead of tiles, similar to the construction of $F_{eq}^{A}$, for demonstration purposes.

\begin{example}[Pinwheel-variant]
The substitution in Figure \ref{pinwheel_variant_substitution} is defined over four tiles and their rotations. It is a variation of \emph{Pinwheel\ substitution} \cite{Web_Tiling_Encyclopedia}. Every tile in the substitution is a right triangle with side lengths $1,2,\sqrt{5}$. The expansion factor for this substitution is $\sqrt{5}$. The curves in Figure \ref{pinwheel_variant_substitution_order} define total orders over 1-supertiles of this substitution. Let $R_1,R_2$ denote the regions, rhombus and rectangle, shown in Figure \ref{pinwheel_rhombus_and_rectangle_prototiles}. 
Using the algorithm together with the defined total orders, we can produce two space-filling curves $F_{pin}^1, F_{pin}^2$ over $R_1,R_2$, respectively. Figure \ref{pinwheel_rhombus_orders} indicates the first two approximants of $F_{pin}^1$ and $F_{pin}^2$. Observe that $F^i_{pin}(0)=F^i_{pin}(1)$ for $i=1,2$. 

Figure \ref{pinwheel_rhombus_filled_approximants} and Figure \ref{pinwheel_rhombus_filled_approximants2} demonstrates the first 5 approximants of $F_{pin}^1$, whereas Figure \ref{pinwheel_rectangles_filled_approximants} and Figure \ref{pinwheel_rectangles_filled_approximants2} illustrates the first 6 approximants of $F_{pin}^2$.

\begin{figure}[H]
\centering
\begin{tikzpicture}[scale=0.9]
\filldraw[thick, color=gray!40] (0,0) -- (2,0) -- (2,4) -- (0,0);
\draw (0,0) -- (2,0) -- (2,4) -- (0,0);
\draw (4,0) -- (6,0) -- (6,4) -- (4,0);

\filldraw[thick, color=gray!40] (8,0) -- (10,0) -- (8,4) -- (8,0);
\draw (8,0) -- (10,0) -- (8,4) -- (8,0);
\draw (12,0) -- (14,0) -- (12,4) -- (12,0);

\coordinate (A)  at (0.4472135954999579392818347337462552470881236719223051448541794490, 0.8944271909999158785636694674925104941762473438446102897083588981);
\coordinate (B)  at (1.3416407864998738178455042012387657412643710157669154345625383472, 2.6832815729997476356910084024775314825287420315338308691250766944);
\coordinate (C)  at (2, 2.2360679774997896964091736687312762354406183596115257242708972454);
\coordinate (D)  at (1.2236067977499789696409173668731276235440618359611525724270897245, 0.4472135954999579392818347337462552470881236719223051448541794490);

\coordinate (AA)  at (4+0.4472135954999579392818347337462552470881236719223051448541794490, 0.8944271909999158785636694674925104941762473438446102897083588981);
\coordinate (BB)  at (4+1.3416407864998738178455042012387657412643710157669154345625383472, 2.6832815729997476356910084024775314825287420315338308691250766944);
\coordinate (CC)  at (4+2, 2.2360679774997896964091736687312762354406183596115257242708972454);
\coordinate (DD)  at (4+1.2236067977499789696409173668731276235440618359611525724270897245, 0.4472135954999579392818347337462552470881236719223051448541794490);

\coordinate (XA)  at (8+1.5527864045000420607181652662537447529118763280776948551458205509, 0.8944271909999158785636694674925104941762473438446102897083588981);
\coordinate (XB)  at (8+0.6583592135001261821544957987612342587356289842330845654374616527, 2.6832815729997476356910084024775314825287420315338308691250766944);
\coordinate (XC)  at (8, 2.2360679774997896964091736687312762354406183596115257242708972454);
\coordinate (XD)  at (8+0.8944271909999158785636694674925104941762473438446102897083588981-0.05, 0.4472135954999579392818347337462552470881236719223051448541794490+0.05);

\coordinate (XXA)  at (4+8+1.5527864045000420607181652662537447529118763280776948551458205509, 0.8944271909999158785636694674925104941762473438446102897083588981);
\coordinate (XXB)  at (4+8+0.6583592135001261821544957987612342587356289842330845654374616527, 2.6832815729997476356910084024775314825287420315338308691250766944);
\coordinate (XXC)  at (4+8, 2.2360679774997896964091736687312762354406183596115257242708972454);
\coordinate (XXD)  at (4+8+0.8944271909999158785636694674925104941762473438446102897083588981-0.05, 0.4472135954999579392818347337462552470881236719223051448541794490+0.05);

\filldraw[thick, color=gray!40] (-1,3) -- (0,3) -- (0,5) -- (-1,3);
\draw (-1,3) -- (0,3) -- (0,5) -- (-1,3);
\draw (4-1,3) -- (4,3) -- (4,5) -- (4-1,3);

\filldraw[thick, color=gray!40] (14-4,3) -- (15-4,3) -- (14-4,5) -- (14-4,3);
\draw (14,3) -- (15,3) -- (14,5) -- (14,3);
\draw (14-4,3) -- (15-4,3) -- (14-4,5) -- (14-4,3);

\draw (2,0) -- (A) -- (C) -- (B);
\draw (D) -- (C);
\draw (2+4,0) -- (AA) -- (CC) -- (BB);
\draw (DD) -- (CC);
\draw (8,0) -- (XA) -- (XC) -- (XD);
\draw (XC) -- (XB);
\draw (8+4,0) -- (XXA) -- (XXC) -- (XXD);
\draw (XXC) -- (XXB);

\foreach \inn in {0,4}{
\draw[very thick, ->] (\inn+0.25,2.75) -- (\inn+0.75,2.25);
\draw[very thick, ->] (\inn+9.75,2.75) -- (\inn+9.25,2.25);
}

\end{tikzpicture}
\caption{Pinwheel-variant substitution}
\label{pinwheel_variant_substitution}
\end{figure}
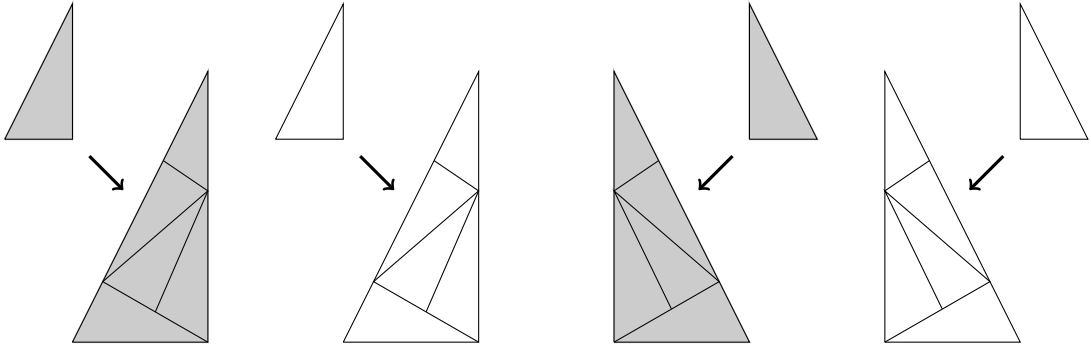

\begin{figure}[H]
\centering
\begin{tikzpicture}[scale=1]%[scale=0.9]
\filldraw[thick, color=gray!40] (0,0) -- (2,0) -- (2,4) -- (0,0);
\draw (0,0) -- (2,0) -- (2,4) -- (0,0);
\draw (4,0) -- (6,0) -- (6,4) -- (4,0);

\filldraw[thick, color=gray!40] (8,0) -- (10,0) -- (8,4) -- (8,0);
\draw (8,0) -- (10,0) -- (8,4) -- (8,0);
\draw (12,0) -- (14,0) -- (12,4) -- (12,0);

\coordinate (A)  at (0.4472135954999579392818347337462552470881236719223051448541794490, 0.8944271909999158785636694674925104941762473438446102897083588981);
\coordinate (B)  at (1.3416407864998738178455042012387657412643710157669154345625383472, 2.6832815729997476356910084024775314825287420315338308691250766944);
\coordinate (C)  at (2, 2.2360679774997896964091736687312762354406183596115257242708972454);
\coordinate (D)  at (1.2236067977499789696409173668731276235440618359611525724270897245, 0.4472135954999579392818347337462552470881236719223051448541794490);

\coordinate (AA)  at (4+0.4472135954999579392818347337462552470881236719223051448541794490, 0.8944271909999158785636694674925104941762473438446102897083588981);
\coordinate (BB)  at (4+1.3416407864998738178455042012387657412643710157669154345625383472, 2.6832815729997476356910084024775314825287420315338308691250766944);
\coordinate (CC)  at (4+2, 2.2360679774997896964091736687312762354406183596115257242708972454);
\coordinate (DD)  at (4+1.2236067977499789696409173668731276235440618359611525724270897245, 0.4472135954999579392818347337462552470881236719223051448541794490);

\coordinate (XA)  at (8+1.5527864045000420607181652662537447529118763280776948551458205509, 0.8944271909999158785636694674925104941762473438446102897083588981);
\coordinate (XB)  at (8+0.6583592135001261821544957987612342587356289842330845654374616527, 2.6832815729997476356910084024775314825287420315338308691250766944);
\coordinate (XC)  at (8, 2.2360679774997896964091736687312762354406183596115257242708972454);
\coordinate (XD)  at (8+0.8944271909999158785636694674925104941762473438446102897083588981-0.05, 0.4472135954999579392818347337462552470881236719223051448541794490+0.05);

\coordinate (XXA)  at (4+8+1.5527864045000420607181652662537447529118763280776948551458205509, 0.8944271909999158785636694674925104941762473438446102897083588981);
\coordinate (XXB)  at (4+8+0.6583592135001261821544957987612342587356289842330845654374616527, 2.6832815729997476356910084024775314825287420315338308691250766944);
\coordinate (XXC)  at (4+8, 2.2360679774997896964091736687312762354406183596115257242708972454);
\coordinate (XXD)  at (4+8+0.8944271909999158785636694674925104941762473438446102897083588981-0.05, 0.4472135954999579392818347337462552470881236719223051448541794490+0.05);

\draw (2,0) -- (A) -- (C) -- (B);
\draw (D) -- (C);
\draw (2+4,0) -- (AA) -- (CC) -- (BB);
\draw (DD) -- (CC);
\draw (8,0) -- (XA) -- (XC) -- (XD);
\draw (XC) -- (XB);
\draw (8+4,0) -- (XXA) -- (XXC) -- (XXD);
\draw (XXC) -- (XXB);

\draw[line width=1pt,<-] (2-1.341615029966348,0.29825827776825253) -- (2-0.29806514130091566, 0.8944529391925291) -- (2-0.8943241815476379, 1.1926468381383357) -- (2-0.8942598027251931+0.1, 1.9380028278579613) -- (2-0.29787200483357956+0.1, 3.130520908351406);
\draw[line width=1pt,->] (4+2-1.341615029966348,0.29825827776825253) -- (4+2-0.29806514130091566, 0.8944529391925291) -- (4+2-0.8943241815476379, 1.1926468381383357) -- (4+2-0.8942598027251931+0.1, 1.9380028278579613) -- (4+2-0.29787200483357956+0.1, 3.130520908351406);

\draw[line width=1pt, <-] (8+0.29787200483357956-0.1, 3.130520908351406) -- (8+0.8942598027251929, 1.9380028278579615-0.1) -- (8+0.8943241815476379, 1.1926468381383357) -- (8+0.29806514130091566, 0.8944529391925291) -- (8+1.341615029966348, 0.2982582777682524);
\draw[line width=1pt, ->] (4+8+0.29787200483357956-0.1, 3.130520908351406) -- (4+8+0.8942598027251929-0.1, 1.9380028278579615) -- (4+8+0.8943241815476379, 1.1926468381383357) -- (4+8+0.29806514130091566, 0.8944529391925291) -- (4+8+1.341615029966348, 0.2982582777682524);
\end{tikzpicture}
\caption{Total orders over 1-supertiles of the Pinwheel-variant substitution}
\label{pinwheel_variant_substitution_order}
\end{figure}
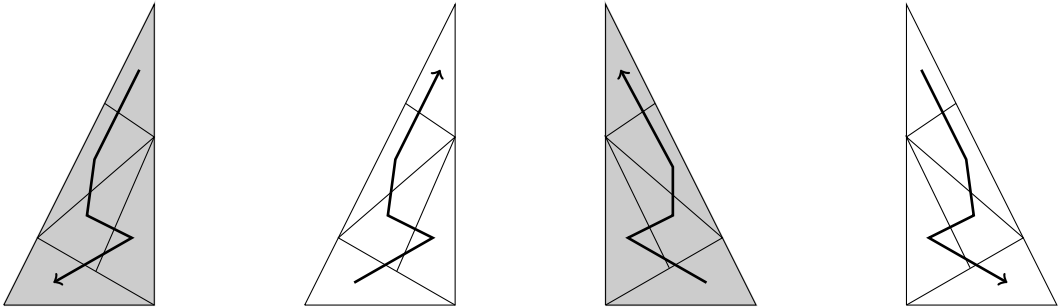

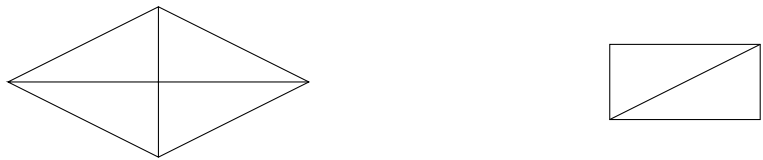
\begin{figure}[H]
\centering
\begin{tikzpicture}
\draw (2,0) -- (-2,0) -- (0,-1) -- (0,1) -- (-2,0);
\draw (0,1) -- (2,0) -- (0,-1);

\draw (6,-0.5) rectangle (8,0.5);
\draw (6,-0.5) -- (8,0.5);

\end{tikzpicture}
\caption{The regions $R_1$ and $R_2$, from left to right respectively}
\label{pinwheel_rhombus_and_rectangle_prototiles}
\end{figure}

\begin{figure}[H]
\centering
\begin{subfigure}[b]{0.18\textwidth}
\centering
\tikz[remember picture]\node[inner sep=0pt,outer sep=0pt] (a){\includegraphics[width=\linewidth]{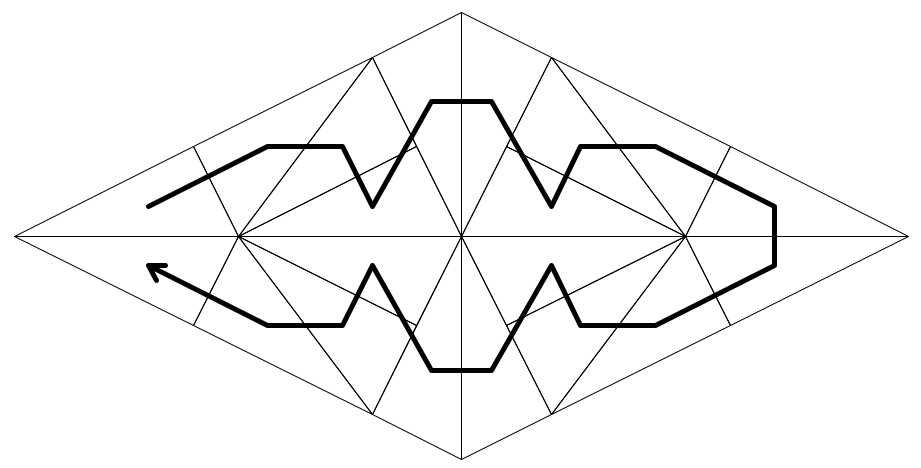}};
\end{subfigure}
\hfill
\begin{subfigure}[b]{0.18\textwidth}
\centering
\tikz[remember picture]\node[inner sep=0pt,outer sep=0pt] (b){\includegraphics[width=\linewidth]{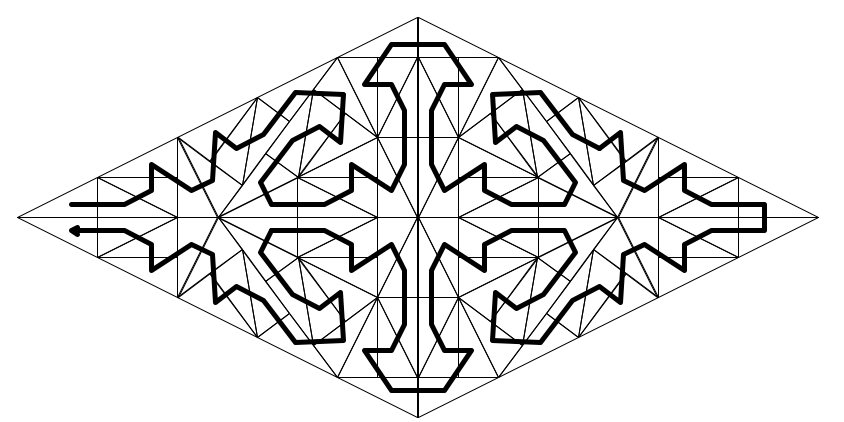}};
\end{subfigure}
\hfill  
\begin{subfigure}[b]{0.18\textwidth}
\centering
\tikz[remember picture]\node[inner sep=0pt,outer sep=0pt] (c){\includegraphics[width=\linewidth]{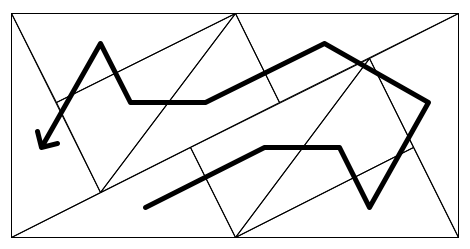}};
\end{subfigure}
\hfill
\begin{subfigure}[b]{0.18\textwidth}
\centering
\tikz[remember picture]\node[inner sep=0pt,outer sep=0pt] (d){\includegraphics[width=\linewidth]{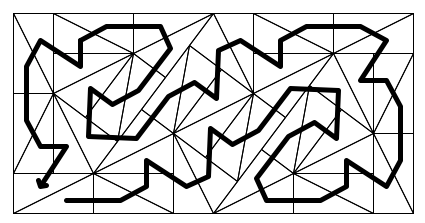}};
\end{subfigure}
\tikz[remember picture,overlay]\draw[line width=2pt,-stealth] ([xshift=5pt]a.east) -- ([xshift=35pt]a.east)node[midway,above,text=black,font=\LARGE\bfseries\sffamily] {};
\tikz[remember picture,overlay]\draw[line width=2pt,-stealth] ([xshift=5pt]c.east) -- ([xshift=35pt]c.east)node[midway,above,text=black,font=\LARGE\bfseries\sffamily] {};
\caption{First two approximants of $F^1_{pin}$ and $F^2_{pin}$}
\label{pinwheel_rhombus_orders}
\end{figure}

\begin{figure}[H]
\centering
\begin{subfigure}[b]{0.4\textwidth}
\centering
\tikz[remember picture]\node[inner sep=0pt,outer sep=0pt] (a){\includegraphics[width=\linewidth]{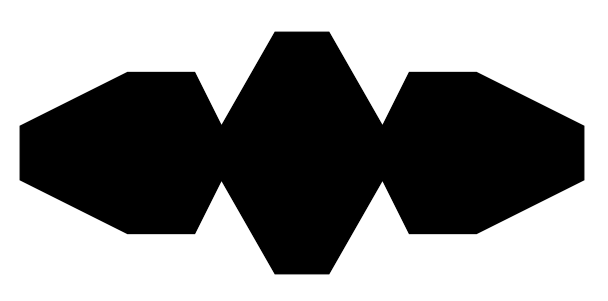}};
\end{subfigure}
\hfill
\begin{subfigure}[b]{0.4\textwidth}
\centering
\tikz[remember picture]\node[inner sep=0pt,outer sep=0pt] (b){\includegraphics[width=\linewidth]{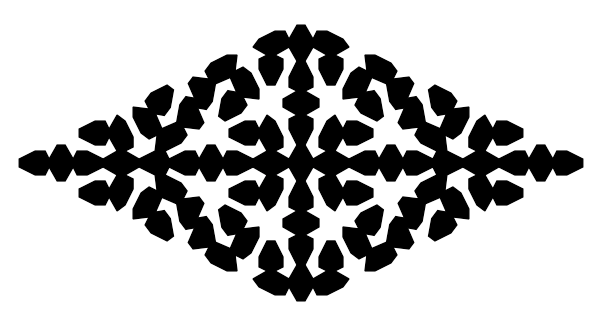}};
\end{subfigure}
\hfill
\tikz[remember picture,overlay]\draw[line width=2pt,-stealth] ([xshift=10pt]a.east) -- ([xshift=30pt]a.east)node[midway,above,text=black,font=\LARGE\bfseries\sffamily] {};
\tikz[remember picture,overlay]\draw[line width=2pt,-stealth] ([xshift=10pt]b.east) -- ([xshift=30pt]b.east)node[midway,below,text=black,font=\LARGE\bfseries\sffamily] {};
\begin{subfigure}[b]{\textwidth}
\centering
\tikz[remember picture]\node[inner sep=0pt,outer sep=0pt] (c){\includegraphics[width=\linewidth]{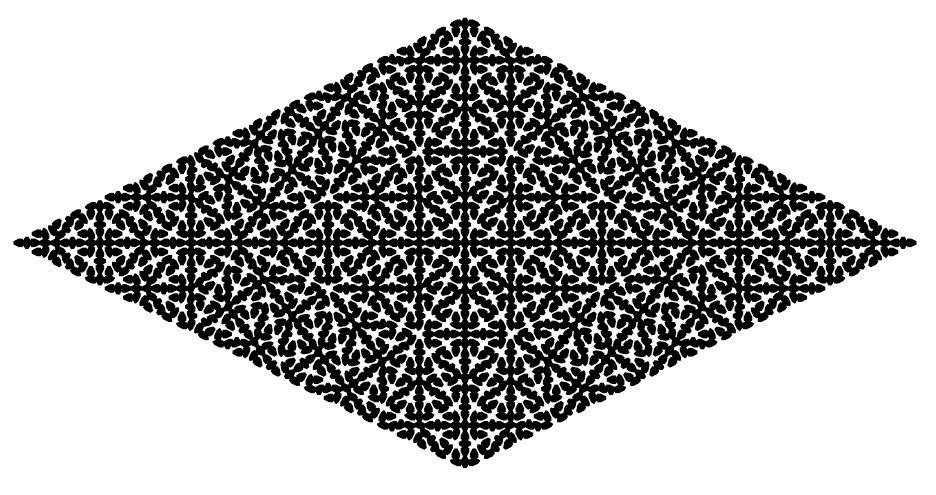}};
%\tikz[remember picture]\node[inner sep=0pt,outer sep=0pt] (c){\includegraphics[scale=0.4]{gasd.png}};
%\tikz[remember picture]\node[inner sep=0pt,outer sep=0pt] (c){\includegraphics[width=\linewidth]{Figures/Pinwheel_Rhombus_Filled_5.png}};

\end{subfigure}
\caption{1st, 3rd and 5th approximants of $F^1_{pin}$ - filled version. The 5th approximant is scaled up for illustration purposes.}
\label{pinwheel_rhombus_filled_approximants}
\end{figure}

\begin{figure}[H]
\centering
\begin{subfigure}[b]{0.4\textwidth}
\centering
\tikz[remember picture]\node[inner sep=0pt,outer sep=0pt] (a){\includegraphics[width=\linewidth]{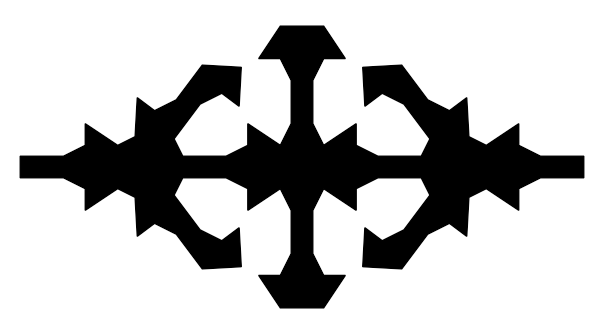}};
\end{subfigure}
\hfill
\begin{subfigure}[b]{0.4\textwidth}
\centering
\tikz[remember picture]\node[inner sep=0pt,outer sep=0pt] (b){\includegraphics[width=\linewidth]{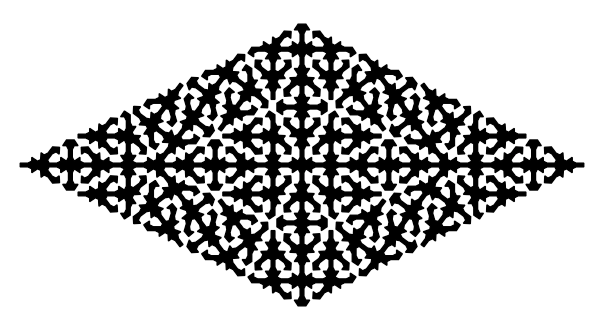}};
\end{subfigure}
\hfill
\tikz[remember picture,overlay]\draw[line width=2pt,-stealth] ([xshift=10pt]a.east) -- ([xshift=30pt]a.east)node[midway,above,text=black,font=\LARGE\bfseries\sffamily] {};
\tikz[remember picture,overlay]\draw[line width=2pt,-stealth] ([xshift=10pt]b.east) -- ([xshift=30pt]b.east)node[midway,below,text=black,font=\LARGE\bfseries\sffamily] {};
\caption{2nd and 4th approximants of $F^1_{pin}$ - filled version.}
\label{pinwheel_rhombus_filled_approximants2}
\end{figure}

\begin{figure}[H]
\centering
\begin{subfigure}[b]{0.4\textwidth}
\centering
\tikz[remember picture]\node[inner sep=0pt,outer sep=0pt] (a){\includegraphics[width=\linewidth]{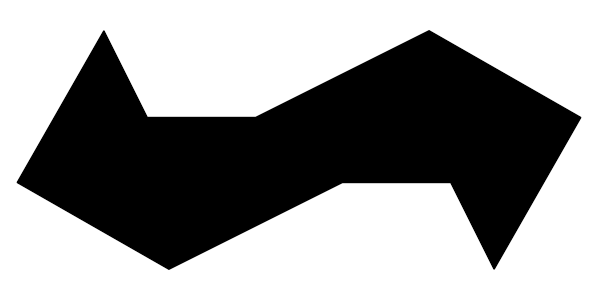}};
\end{subfigure}
\hfill
\begin{subfigure}[b]{0.4\textwidth}
\centering
\tikz[remember picture]\node[inner sep=0pt,outer sep=0pt] (b){\includegraphics[width=\linewidth]{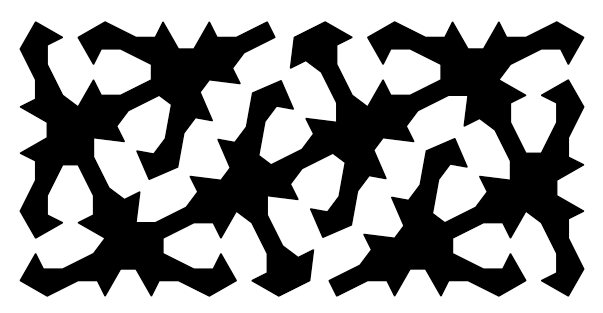}};
\end{subfigure}
\hfill
\tikz[remember picture,overlay]\draw[line width=2pt,-stealth] ([xshift=10pt]a.east) -- ([xshift=30pt]a.east)node[midway,above,text=black,font=\LARGE\bfseries\sffamily] {};
\tikz[remember picture,overlay]\draw[line width=2pt,-stealth] ([xshift=10pt]b.east) -- ([xshift=30pt]b.east)node[midway,below,text=black,font=\LARGE\bfseries\sffamily] {};
\begin{subfigure}[b]{\textwidth}
\centering
\includegraphics[scale=0.35]{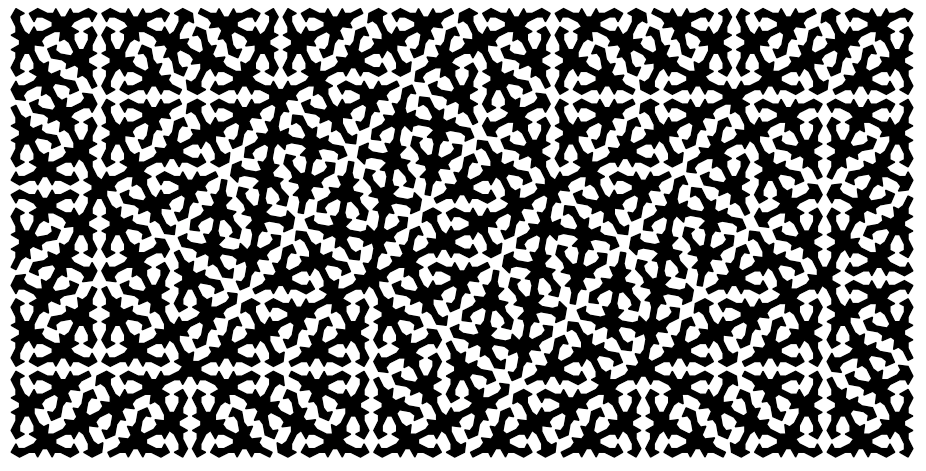}
\end{subfigure}
\caption{1st, 3rd and 5th approximants of $F^2_{pin}$ - filled version. The 5th approximant is scaled up for illustration purposes.}
\label{pinwheel_rectangles_filled_approximants}
\end{figure}

%\begin{figure}[H]
%\centering
%\begin{subfigure}[b]{0.27\textwidth}
%\centering
%\tikz[remember picture]\node[inner sep=0pt,outer sep=0pt] (a){\includegraphics[width=\linewidth]{Figures/Pinwheel_Rectangle_filled_1.png}};
%\end{subfigure}
%\hfill
%\begin{subfigure}[b]{0.27\textwidth}
%\centering
%\tikz[remember picture]\node[inner sep=0pt,outer sep=0pt] (b){\includegraphics[width=\linewidth]{Figures/Pinwheel_Rectangle_filled_3.png}};
%\end{subfigure}
%\hfill
%\begin{subfigure}[b]{0.27\textwidth}
%\centering
%\tikz[remember picture]\node[inner sep=0pt,outer sep=0pt] (c){\includegraphics[width=\linewidth]{Figures/Pinwheel_Rectangle_filled_5.png}};
%\end{subfigure}
%\tikz[remember picture,overlay]\draw[line width=2pt,-stealth] ([xshift=5pt]a.east) -- ([xshift=35pt]a.east)node[midway,above,text=black,font=\LARGE\bfseries\sffamily] {};
%\tikz[remember picture,overlay]\draw[line width=2pt,-stealth] ([xshift=5pt]b.east) -- ([xshift=35pt]b.east)node[midway,below,text=black,font=\LARGE\bfseries\sffamily] {};
%\caption{1st, 3rd and 5th approximants of $F^2_{pin}$ - filled version}
%\label{pinwheel_rectangles_filled_approximants}
%\end{figure}

\begin{figure}[H]
\centering
\begin{subfigure}[b]{0.4\textwidth}
\centering
\tikz[remember picture]\node[inner sep=0pt,outer sep=0pt] (a){\includegraphics[width=\linewidth]{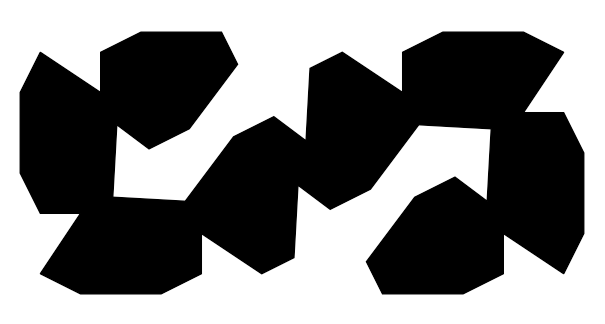}};
\end{subfigure}
\hfill
\begin{subfigure}[b]{0.4\textwidth}
\centering
\tikz[remember picture]\node[inner sep=0pt,outer sep=0pt] (b){\includegraphics[width=\linewidth]{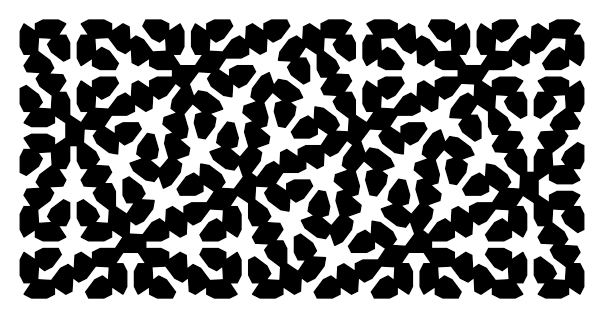}};
\end{subfigure}
\hfill
\tikz[remember picture,overlay]\draw[line width=2pt,-stealth] ([xshift=10pt]a.east) -- ([xshift=30pt]a.east)node[midway,above,text=black,font=\LARGE\bfseries\sffamily] {};
\tikz[remember picture,overlay]\draw[line width=2pt,-stealth] ([xshift=10pt]b.east) -- ([xshift=30pt]b.east)node[midway,below,text=black,font=\LARGE\bfseries\sffamily] {};
\begin{subfigure}[b]{\textwidth}
\centering
\tikz[remember picture]\node[inner sep=0pt,outer sep=0pt] (c){\includegraphics[scale=0.4]{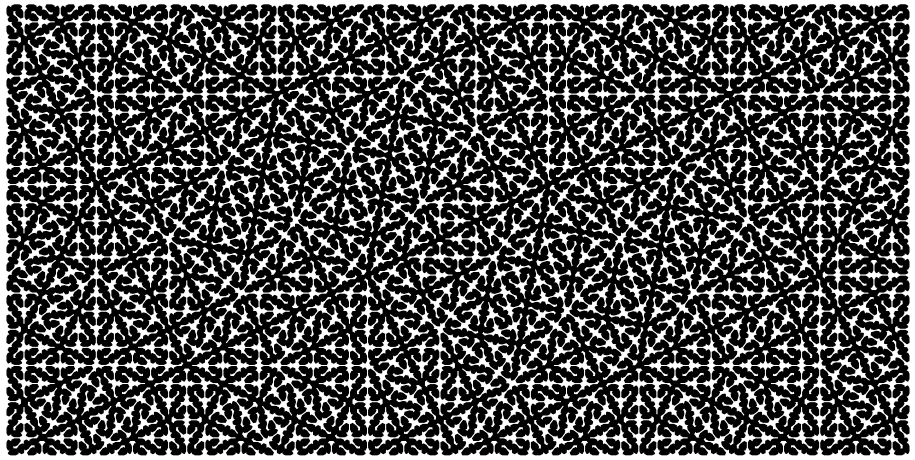}};
%\tikz[remember picture]\node[inner sep=0pt,outer sep=0pt] (c){\includegraphics[width=\linewidth]{asfff.png}};
%\tikz[remember picture]\node[inner sep=0pt,outer sep=0pt] (c){\includegraphics[width=\linewidth]{Figures/Pinwheel_Rectangle_filled_6.png}};

\end{subfigure}
\caption{2nd, 4th and 6th approximants of $F^2_{pin}$ - filled version. The 6th approximant is scaled up for illustration purposes.}
\label{pinwheel_rectangles_filled_approximants2}
\end{figure}

%\begin{figure}[H]
%\centering
%\begin{subfigure}[b]{0.27\textwidth}
%\centering
%\tikz[remember picture]\node[inner sep=0pt,outer sep=0pt] (a){\includegraphics[width=\linewidth]{Figures/Pinwheel_Rectangle_filled_2.png}};
%\end{subfigure}
%\hfill
%\begin{subfigure}[b]{0.27\textwidth}
%\centering
%\tikz[remember picture]\node[inner sep=0pt,outer sep=0pt] (b){\includegraphics[width=\linewidth]{Figures/Pinwheel_Rectangle_filled_4.png}};
%\end{subfigure}
%\hfill
%\begin{subfigure}[b]{0.27\textwidth}
%\centering
%\tikz[remember picture]\node[inner sep=0pt,outer sep=0pt] (c){\includegraphics[width=\linewidth]{Figures/Pinwheel_Rectangle_filled_6.png}};
%\end{subfigure}
%\tikz[remember picture,overlay]\draw[line width=2pt,-stealth] ([xshift=5pt]a.east) -- ([xshift=35pt]a.east)node[midway,above,text=black,font=\LARGE\bfseries\sffamily] {};
%\tikz[remember picture,overlay]\draw[line width=2pt,-stealth] ([xshift=5pt]b.east) -- ([xshift=35pt]b.east)node[midway,below,text=black,font=\LARGE\bfseries\sffamily] {};
%\caption{2nd, 4th and 6th approximants of $F^2_{pin}$ - filled version}
%\label{pinwheel_rectangles_filled_approximants2}
%\end{figure}

\end{example}

\begin{example}[Penrose-Robinson-variant]
Start with the substitution given in Figure \ref{penrose}. Its domain consists of $12$ tiles and their rotations, which are congruent copy of two different shapes, 
an isosceles triangle with side lengths $1,1,(1+\sqrt{5})/2$ and an isosceles triangle with side lengths $1,1,(\sqrt{5}-1)/2$. The expansion factor for this substitution is the golden mean $(1+\sqrt{5})/2$. It is a variation of \emph{Penrose-Robinson\ substitution} \cite{Web_Tiling_Encyclopedia}.
Define the total orders described in Figure \ref{penrose_order}. Using the described total orders we generate space filling curves $F_{star}, F_{deca}$ that fill the supports of the patches shown in Figure \ref{regions}, respectively, such that  $F_{star}(0)=F_{star}(1)$ and $F_{deca}(0)=F_{deca}(1)$. The associated approximants of $F_{star}$ and $F_{deca}$ are shown in Figure \ref{star_approximants}, Figure
\ref{tengon_approximants} and Figure \ref{comparison_star_deca}.
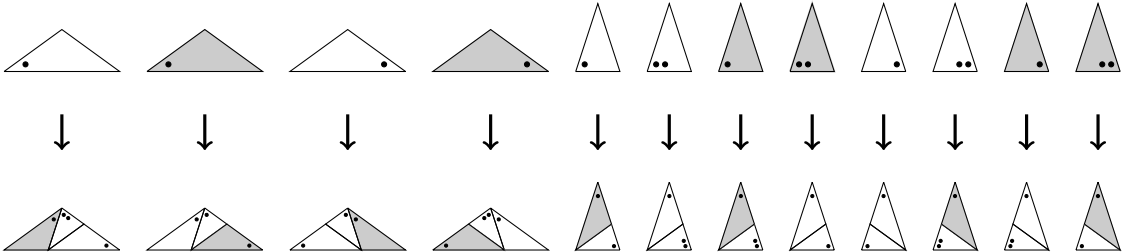
\begin{figure}[H]
\centering
\begin{tikzpicture}[scale=0.95]
\centering
\pensla{0}{0}{0}{1}
\penslb{2}{0}{0}{1}
\pensra{5.618033988749894848204586834365638117720309179805762862135}{0}{0}{1}
\pensrb{7.618033988749894848204586834365638117720309179805762862135}{0}{0}{1}

\pentla{8}{0}{0}{1}
\pentlb{9}{0}{0}{1}
\pentlc{10}{0}{0}{1}
\pentld{11}{0}{0}{1}
\pentra{12.618033988749894848204586834365638117720309179805762862135}{0}{0}{1}
\pentrb{13.618033988749894848204586834365638117720309179805762862135}{0}{0}{1}
\pentrc{14.618033988749894848204586834365638117720309179805762862135}{0}{0}{1}
\pentrd{15.618033988749894848204586834365638117720309179805762862135}{0}{0}{1}

\foreach \ina in {0,2,4,6}{
\draw[very thick,->] (\ina+0.8090169943749474241022934171828190588601545899028814310675,-0.6)--(\ina+0.8090169943749474241022934171828190588601545899028814310675,-1.1);
}
\foreach \ina in {8,9,10,11,12,13,14,15}{
\draw[very thick,->] (\ina+0.3090169943749474241022934171828190588601545899028814310675,-0.6)--(\ina+0.3090169943749474241022934171828190588601545899028814310675,-1.1);
}
\penssla{0}{-2.5}{0}{0.618033988749894848204586834365638117720309179805762862135}
\pensslb{2}{-2.5}{0}{0.618033988749894848204586834365638117720309179805762862135}
\penssra{5.618033988749894848204586834365638117720309179805762862135}{-2.5}{0}{0.618033988749894848204586834365638117720309179805762862135}
\penssrb{7.618033988749894848204586834365638117720309179805762862135}{-2.5}{0}{0.618033988749894848204586834365638117720309179805762862135}

\penttla{8}{-2.5}{0}{0.618033988749894848204586834365638117720309179805762862135}
\penttlb{9}{-2.5}{0}{0.618033988749894848204586834365638117720309179805762862135}
\penttlc{10}{-2.5}{0}{0.618033988749894848204586834365638117720309179805762862135}
\penttld{11}{-2.5}{0}{0.618033988749894848204586834365638117720309179805762862135}
\penttra{12.618033988749894848204586834365638117720309179805762862135}{-2.5}{0}{0.618033988749894848204586834365638117720309179805762862135}
\penttrb{13.618033988749894848204586834365638117720309179805762862135}{-2.5}{0}{0.618033988749894848204586834365638117720309179805762862135}
\penttrc{14.618033988749894848204586834365638117720309179805762862135}{-2.5}{0}{0.618033988749894848204586834365638117720309179805762862135}
\penttrd{15.618033988749894848204586834365638117720309179805762862135}{-2.5}{0}{0.618033988749894848204586834365638117720309179805762862135}
\end{tikzpicture}
\caption{Penrose-Robinson-variant substitution. 1-supertiles are scaled down for demonstration.}
\label{penrose}
\end{figure}

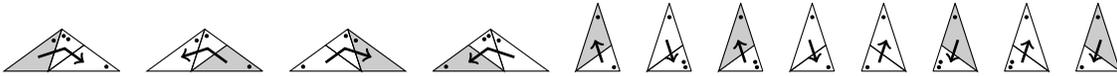
\begin{figure}[H]
\centering
\begin{tikzpicture}[scale=0.95]
\centering

\penssla{0}{-2.5}{0}{0.618033988749894848204586834365638117720309179805762862135}
\pensslb{2}{-2.5}{0}{0.618033988749894848204586834365638117720309179805762862135}
\penssra{5.618033988749894848204586834365638117720309179805762862135}{-2.5}{0}{0.618033988749894848204586834365638117720309179805762862135}
\penssrb{7.618033988749894848204586834365638117720309179805762862135}{-2.5}{0}{0.618033988749894848204586834365638117720309179805762862135}

\penttla{8}{-2.5}{0}{0.618033988749894848204586834365638117720309179805762862135}
\penttlb{9}{-2.5}{0}{0.618033988749894848204586834365638117720309179805762862135}
\penttlc{10}{-2.5}{0}{0.618033988749894848204586834365638117720309179805762862135}
\penttld{11}{-2.5}{0}{0.618033988749894848204586834365638117720309179805762862135}
\penttra{12.618033988749894848204586834365638117720309179805762862135}{-2.5}{0}{0.618033988749894848204586834365638117720309179805762862135}
\penttrb{13.618033988749894848204586834365638117720309179805762862135}{-2.5}{0}{0.618033988749894848204586834365638117720309179805762862135}
\penttrc{14.618033988749894848204586834365638117720309179805762862135}{-2.5}{0}{0.618033988749894848204586834365638117720309179805762862135}
\penttrd{15.618033988749894848204586834365638117720309179805762862135}{-2.5}{0}{0.618033988749894848204586834365638117720309179805762862135}

\draw[line width=1pt, ->] (0.4756836610416140907689600838494857255268212565695480977341666666,0.1959284174308243763895686515463575895325508125477153303574241602-2.5) -- (0.8483616572915790401704890286380317647669243165048023851128738522, 0.3170188387650511907054797777931273811352328780419167408157685481-2.5) -- (1.1180339887498948482045868343656381177203091798057628621354486227, 0.1210904213342268143159111262467697916026820654942014104583443878-2.5);
\draw[line width=1pt, <-] (2+0.4756836610416140907689600838494857255268212565695480977341666666,0.1959284174308243763895686515463575895325508125477153303574241602-2.5) -- (2+0.8483616572915790401704890286380317647669243165048023851128738522, 0.3170188387650511907054797777931273811352328780419167408157685481-2.5) -- (2+1.1180339887498948482045868343656381177203091798057628621354486227, 0.1210904213342268143159111262467697916026820654942014104583443878-2.5);

\draw[line width=1pt, <-] (5.618033988749894848204586834365638117720309179805762862135-0.4756836610416140907689600838494857255268212565695480977341666666,0.1959284174308243763895686515463575895325508125477153303574241602-2.5) -- (5.618033988749894848204586834365638117720309179805762862135-0.8483616572915790401704890286380317647669243165048023851128738522, 0.3170188387650511907054797777931273811352328780419167408157685481-2.5) -- (5.618033988749894848204586834365638117720309179805762862135-1.1180339887498948482045868343656381177203091798057628621354486227, 0.1210904213342268143159111262467697916026820654942014104583443878-2.5);
\draw[line width=1pt, ->] (2+5.618033988749894848204586834365638117720309179805762862135-0.4756836610416140907689600838494857255268212565695480977341666666,0.1959284174308243763895686515463575895325508125477153303574241602-2.5) -- (2+5.618033988749894848204586834365638117720309179805762862135-0.8483616572915790401704890286380317647669243165048023851128738522, 0.3170188387650512140680802986025810241699218750000000000000000000-2.5) -- (2+5.618033988749894848204586834365638117720309179805762862135-1.1180339887498948482045868343656381177203091798057628621354486227, 0.1210904213342268143159111262467697916026820654942014104583443878-2.5);

\draw[line width=1pt, <-] (8+0.2696723314583158080340978057276063529533848633009604770225747704,0.4381092600992780050213909040398971727379149435361181512741129360-2.5) -- (8+0.3726779962499649494015289447885460392401030599352542873784828742, 0.1210904213342268143159111262467697916026820654942014104583443878-2.5);
\draw[line width=1pt, ->] (1+8+0.2696723314583158080340978057276063529533848633009604770225747704,0.4381092600992780050213909040398971727379149435361181512741129360-2.5) -- (1+8+0.3726779962499649494015289447885460392401030599352542873784828742, 0.1210904213342268143159111262467697916026820654942014104583443878-2.5);
\draw[line width=1pt, <-] (2+8+0.2696723314583158080340978057276063529533848633009604770225747704,0.4381092600992780050213909040398971727379149435361181512741129360-2.5) -- (2+8+0.3726779962499649494015289447885460392401030599352542873784828742, 0.1210904213342268143159111262467697916026820654942014104583443878-2.5);
\draw[line width=1pt, ->] (2+1+8+0.2696723314583158080340978057276063529533848633009604770225747704,0.4381092600992780050213909040398971727379149435361181512741129360-2.5) -- (2+1+8+0.3726779962499649494015289447885460392401030599352542873784828742, 0.1210904213342268143159111262467697916026820654942014104583443878-2.5);

\draw[line width=1pt, <-] (12+0.618033988749894848204586834365638117720309179805762862135-0.2696723314583158080340978057276063529533848633009604770225747704,0.4381092600992780050213909040398971727379149435361181512741129360-2.5) -- (12+0.618033988749894848204586834365638117720309179805762862135-0.3726779962499649494015289447885460392401030599352542873784828742, 0.1210904213342268143159111262467697916026820654942014104583443878-2.5);
\draw[line width=1pt, ->] (1+12+0.618033988749894848204586834365638117720309179805762862135-0.2696723314583158080340978057276063529533848633009604770225747704,0.4381092600992780050213909040398971727379149435361181512741129360-2.5) -- (1+12+0.618033988749894848204586834365638117720309179805762862135-0.3726779962499649494015289447885460392401030599352542873784828742, 0.1210904213342268143159111262467697916026820654942014104583443878-2.5);
\draw[line width=1pt, <-] (2+12+0.618033988749894848204586834365638117720309179805762862135-0.2696723314583158080340978057276063529533848633009604770225747704,0.4381092600992780050213909040398971727379149435361181512741129360-2.5) -- (2+12+0.618033988749894848204586834365638117720309179805762862135-0.3726779962499649494015289447885460392401030599352542873784828742, 0.1210904213342268143159111262467697916026820654942014104583443878-2.5);
\draw[line width=1pt, ->] (2+1+12+0.618033988749894848204586834365638117720309179805762862135-0.2696723314583158080340978057276063529533848633009604770225747704,0.4381092600992780050213909040398971727379149435361181512741129360-2.5) -- (2+1+12+0.618033988749894848204586834365638117720309179805762862135-0.3726779962499649494015289447885460392401030599352542873784828742, 0.1210904213342268143159111262467697916026820654942014104583443878-2.5);
\end{tikzpicture}
\caption{Total orders over 1-supertiles of the Penrose-Robinson-variant substitution}
\label{penrose_order}
\end{figure}

\begin{figure}[H]
\centering
\begin{tikzpicture}[scale=1.7]%[scale=1.4]
\centering
\foreach \inn in {-5}{
\pensla{\inn}{0}{18}{1}
\pensra{\inn}{0}{198}{1}
\pensla{\inn}{0}{18+72}{1}
\pensra{\inn}{0}{198+72}{1}
\pensla{\inn}{0}{18+2*72}{1}
\pensra{\inn}{0}{198+2*72}{1}
\pensla{\inn}{0}{18+3*72}{1}
\pensra{\inn}{0}{198+3*72}{1}
\pensla{\inn}{0}{18+4*72}{1}
\pensra{\inn}{0}{198+4*72}{1}
}

\foreach \in in {0}{
\pentrd{\in}{-1.04}{-18}{1}
\pentlc{\in}{-1.04}{18}{1}
\pentrd{\in+ 0.9510565162951535721164393333793821434056986341257502224473056444}{-1-0.3090169943749474241022934171828190588601545899028814310677243113}{-18+72}{1}
\pentlc{\in+ 0.9510565162951535721164393333793821434056986341257502224473056444}{-1-0.3090169943749474241022934171828190588601545899028814310677243113}{18+72}{1}
\pentrd{\in- 0.9510565162951535721164393333793821434056986341257502224473056444}{-1-0.3090169943749474241022934171828190588601545899028814310677243113}{-18-72}{1}
\pentlc{\in- 0.9510565162951535721164393333793821434056986341257502224473056444}{-1-0.3090169943749474241022934171828190588601545899028814310677243113}{18-72}{1}

\pentrd{\in+ 0.5877852522924732}{-1+0.8090169943749473}{-18+2*72}{1}
\pentlc{\in+ 0.5877852522924732}{-1+0.8090169943749473}{18+2*72}{1}
\pentrd{\in- 0.5877852522924732}{-1+0.8090169943749473}{-18-2*72}{1}
\pentlc{\in- 0.5877852522924732}{-1+0.8090169943749473}{18-2*72}{1}
}

\end{tikzpicture}
\caption{A star and a decagon.}
\label{regions}
\end{figure}
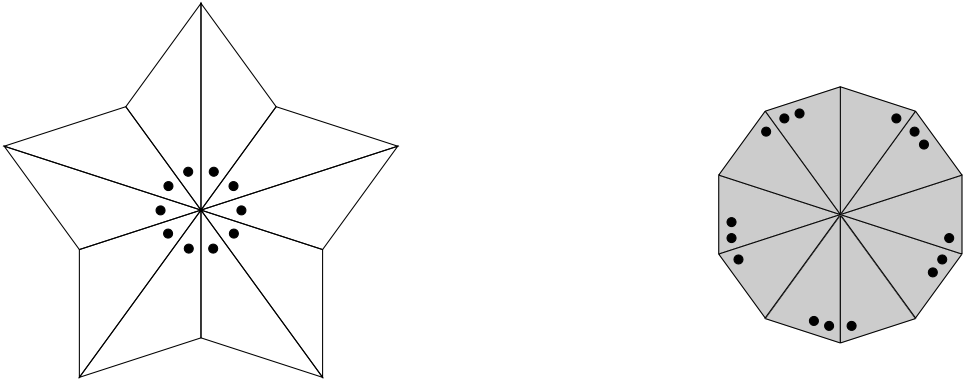

\begin{figure}[H]
\centering
\begin{subfigure}[b]{0.18\textwidth}
\centering
\tikz[remember picture]\node[inner sep=0pt,outer sep=0pt] (a){\includegraphics[width=\linewidth]{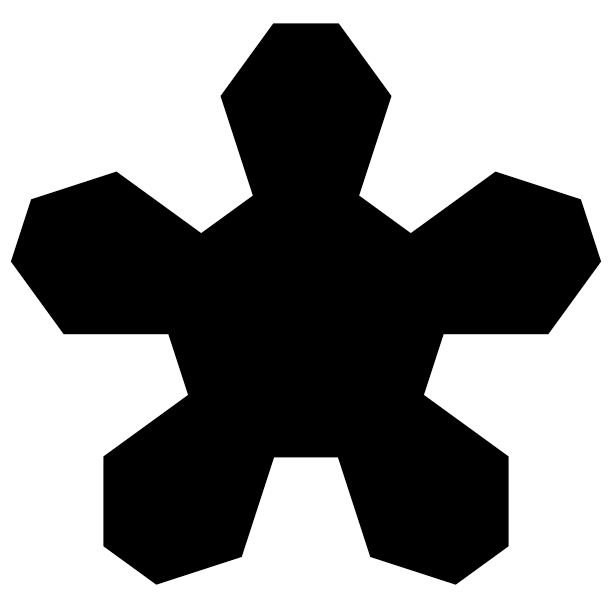}};
\end{subfigure}
\hfill
\begin{subfigure}[b]{0.18\textwidth}
\centering
\tikz[remember picture]\node[inner sep=0pt,outer sep=0pt] (b){\includegraphics[width=\linewidth]{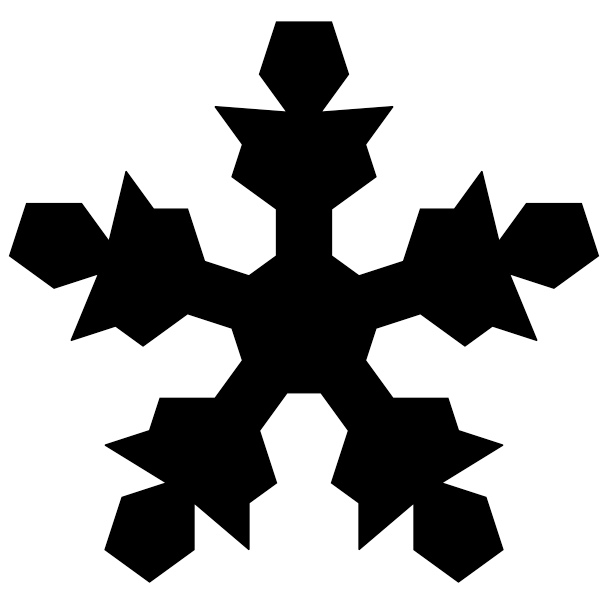}};
\end{subfigure}
\hfill
\begin{subfigure}[b]{0.18\textwidth}
\centering
\tikz[remember picture]\node[inner sep=0pt,outer sep=0pt] (c){\includegraphics[width=\linewidth]{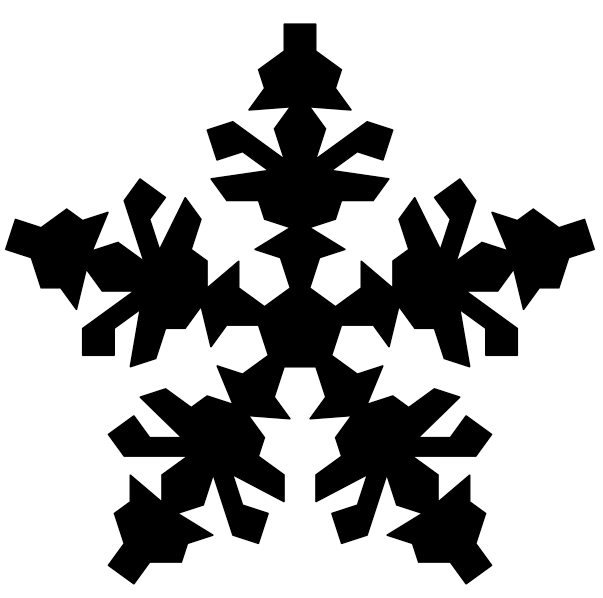}};
\end{subfigure}
\hfill
\begin{subfigure}[b]{0.18\textwidth}
\centering
\tikz[remember picture]\node[inner sep=0pt,outer sep=0pt] (d){\includegraphics[width=\linewidth]{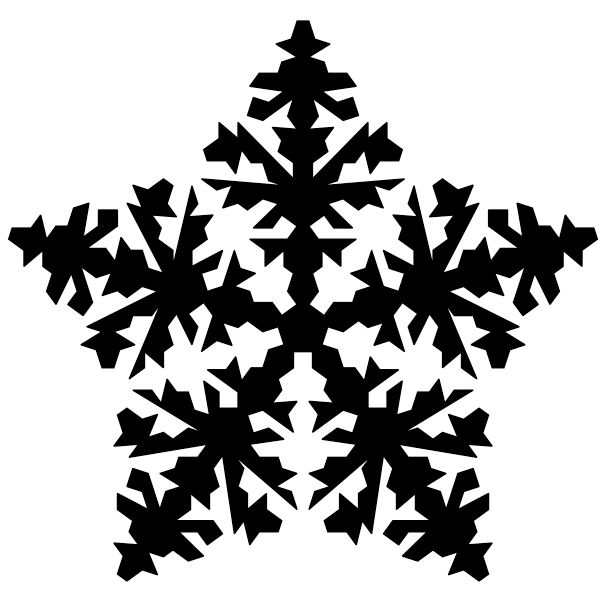}};
\end{subfigure}
\tikz[remember picture,overlay]\draw[line width=2pt,-stealth] ([xshift=5pt]a.east) -- ([xshift=30pt]a.east)node[midway,below,text=black,font=\LARGE\bfseries\sffamily] {};
\tikz[remember picture,overlay]\draw[line width=2pt,-stealth] ([xshift=5pt]b.east) -- ([xshift=30pt]b.east)node[midway,above,text=black,font=\LARGE\bfseries\sffamily] {};
\tikz[remember picture,overlay]\draw[line width=2pt,-stealth] ([xshift=5pt]c.east) -- ([xshift=30pt]c.east)node[midway,above,text=black,font=\LARGE\bfseries\sffamily] {};

\caption{First four approximants of $F_{star}$ - filled version}
\label{star_approximants}
\end{figure}
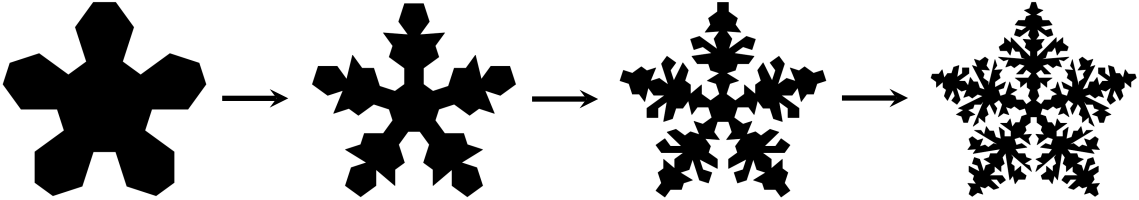

\begin{figure}[H]
\centering
\begin{subfigure}[b]{0.18\textwidth}
\centering
\tikz[remember picture]\node[inner sep=0pt,outer sep=0pt] (a){\includegraphics[width=\linewidth]{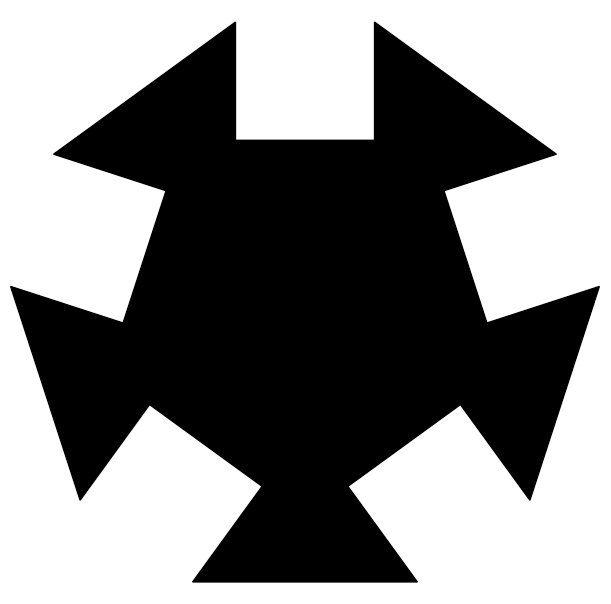}};
\end{subfigure}
\hfill
\begin{subfigure}[b]{0.18\textwidth}
\centering
\tikz[remember picture]\node[inner sep=0pt,outer sep=0pt] (b){\includegraphics[width=\linewidth]{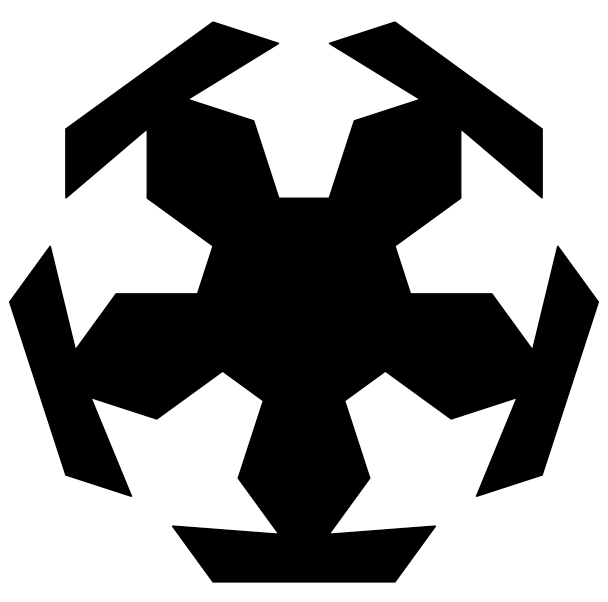}};
\end{subfigure}
\hfill
\begin{subfigure}[b]{0.18\textwidth}
\centering
\tikz[remember picture]\node[inner sep=0pt,outer sep=0pt] (c){\includegraphics[width=\linewidth]{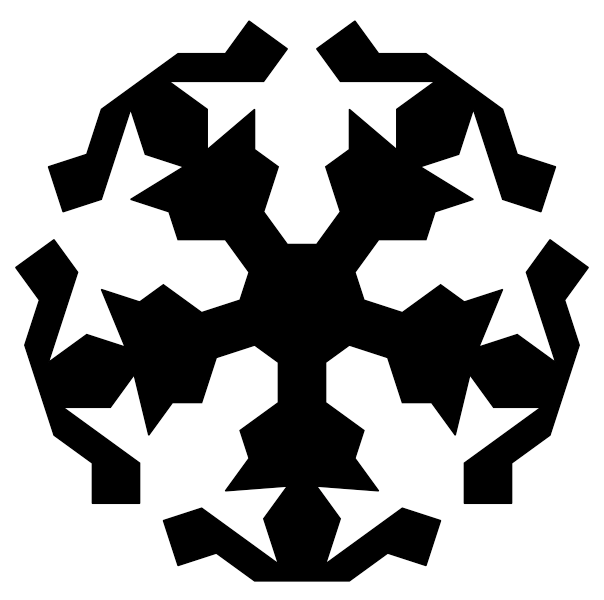}};
\end{subfigure}
\hfill
\begin{subfigure}[b]{0.18\textwidth}
\centering
\tikz[remember picture]\node[inner sep=0pt,outer sep=0pt] (d){\includegraphics[width=\linewidth]{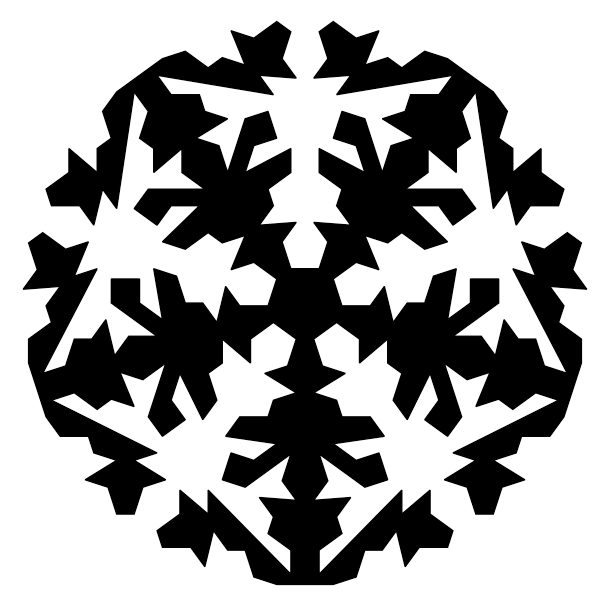}};
\end{subfigure}
\tikz[remember picture,overlay]\draw[line width=2pt,-stealth] ([xshift=5pt]a.east) -- ([xshift=30pt]a.east)node[midway,below,text=black,font=\LARGE\bfseries\sffamily] {};
\tikz[remember picture,overlay]\draw[line width=2pt,-stealth] ([xshift=5pt]b.east) -- ([xshift=30pt]b.east)node[midway,above,text=black,font=\LARGE\bfseries\sffamily] {};
\tikz[remember picture,overlay]\draw[line width=2pt,-stealth] ([xshift=5pt]c.east) -- ([xshift=30pt]c.east)node[midway,above,text=black,font=\LARGE\bfseries\sffamily] {};

\caption{First four approximants of $F_{deca}$ - filled version}
\label{tengon_approximants}
\end{figure}

\begin{figure}[H]
\centering
\begin{subfigure}[b]{0.45\textwidth}
\centering
\includegraphics[width=\textwidth]{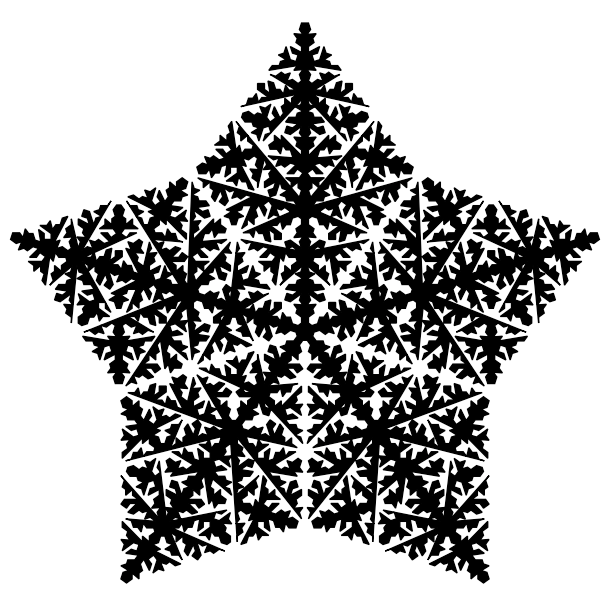}
\end{subfigure}
\hfill
\begin{subfigure}[b]{0.45\textwidth}
\centering
\includegraphics[width=\textwidth]{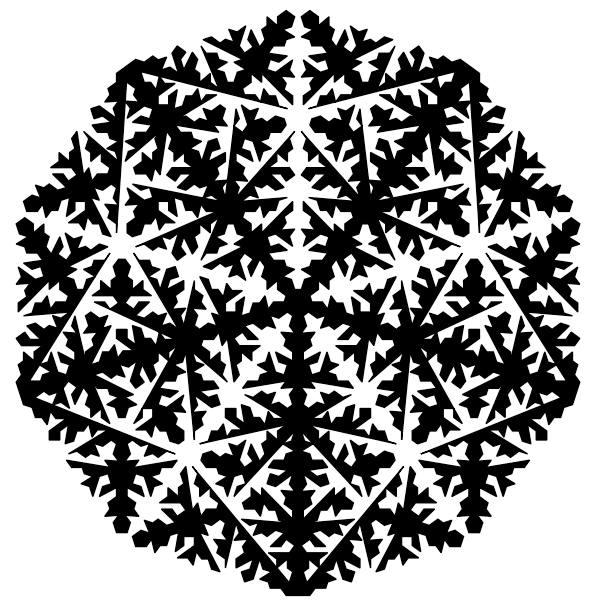}
\end{subfigure}
\caption{6th approximants of $F_{star}$ and $F_{deca}$, from left to right.}
\label{comparison_star_deca}
\end{figure}

\end{example}

\section{Substitutions to Fractal-like Relatively Dense Sets}\label{Section_fractal}

%\subsection{Relatively Dense Fractal-like Sets}

\begin{definition}
A substitution $\omega:\cP\mapsto\cP^*$ is called \emph{primitive} if there exists $k\in\ZZ^+$ such that $\omega^k(p)$ contains a copy of $q$ for every $p,q\in\cP$. 
\end{definition}

Primitive substitutions induce coverings of the plane, called \emph{tilings}, with tiles sharing at most their boundaries. The details of such construction can be found in \cite[Theorem 1.4]{Book_Sadun} or \cite[P:12-13]{Ozkaraca}. We inherit the same idea to produce fractal-like relatively dense sets in the plane.

%\begin{definition}
%A substitution $\omega:\cP\mapsto\cP^*$ is called \emph{primitive} if there exists $k\in\ZZ^+$ such that $\omega^k(p)$ contains a copy of $q$ for every $p,q\in\cP$. In this paper, we further assume that $k=1$.
%\end{definition}

%\begin{definition}
%An inflation rule is ...
%\end{definition}

\begin{proposition}\label{fractal_prop}
Let $\omega:\cP\mapsto\cP^*$ be a given substitution defined over a finite collection $\cP$ with an expansion factor $\lambda>1$ such that $\max\{diam(t):\ t\in\lambda^{-n}\omega^n(p),\ p\in\cP\}\downarrow 0$ as $n\to\infty$, where $diam(t)$ denotes the diameter of $\sss t$. Suppose $F:[0,1]\mapsto\sss p$ is a Lebesgue-type space filling curve generated by the method explained in Theorem \ref{main_theorem} and fills $\sss p$ for some $p\in\cP$. Assume that $F(0)=F(1)$ and $\{F_n:\ n\in\ZZ^+\}$ is the collection of approximants of $F$ $($defined on $\lambda^{-n}\omega^n(p))$ generated by filling the associated closed regions $($i.e. $F_n$'s are sets in $\RR^2)$. Assume further there exists $x\in\RR^2$ and $k\in\ZZ^+$ so that
\begin{enumerate}
    \item[(1)] $p+x \in \omega^n(p+x)$,
    \item[(2)] $\sss (p+x) \cap \bb \omega^n(p+x)=\emptyset$,
    \item[(3)] $F_1+x\subseteq \lambda^n(F_{n+1}+x)$,
    \item[(4)] $F_n$ visits two $($scaled$)$ tiles in $\lambda^{-n}\omega^n(p)$ subsequently only if they share a common edge, for every $n\in\ZZ^+$.
\end{enumerate}
Then
\begin{itemize}
    \item[i.] $F_{k.n+1}+x\subseteq \lambda^{n} (F_{(k+1)n+1}+x)$ for every $k\in\ZZ^+$. In particular,
    \[
    F_{1}+x\subseteq \lambda^n (F_{n+1}+x) \subseteq \lambda^{2n} (F_{2n+1}+x) \subseteq\lambda^{3n} (F_{3n+1}+x)\subseteq\dots
    \]
    \item[ii.] $F=\bigcup\limits_{i=0}^{\infty}\lambda^{in}(F_{in+1}+x)$ is a relatively dense set in the plane. 
    \item[iii.] There exists a collection of sets $\{A_k: k\in\ZZ^+\}$ so that \begin{itemize}
        \item[a.] $\bigcup\limits_{k}A_k=F$,
        \item[b.] $int(A_i)\cap int(A_j)=\emptyset$ whenever $i\neq j$, where $int(A_i),int(A_j)$ are interiors of $A_i$ and $A_j$,respectively,
        \item[c.] There are finite number of indices $k_1,\dots k_m\in\ZZ^+$ for some $m\in\ZZ^+$ such that for each $k\in\ZZ^+$ there exists $n_k\in\{k_1,\dots,k_m\}$ and $x_k\in\RR^2$ with $A_k=A_{n_k}+x_k$.
    \end{itemize}
\end{itemize}
\end{proposition}

\begin{proof}
\begin{itemize}
    \item[i.] We get
    $\{p+x\}\subseteq\omega^n(p+x)\subseteq\omega^{2n}(p+x)\subseteq\dots$, by $(1)$. So, we can conclude by $(3)$ that 
    \[F_{1}+x\subseteq \lambda^n (F_{n+1}+x) \subseteq \lambda^{2n} (F_{2n+1}+x) \subseteq\lambda^{3n} (F_{3n+1}+x)\subseteq\dots.
    \]
    \item[ii.] We have that the collection $T=\bigcup\limits_{i=1}\omega^{in}(p+x)$ is a covering of the plane by $(1)$ and $(2)$ \cite[P:12-13]{Ozkaraca}. Note that $F=\bigcup\limits_{i=0}^{\infty}\lambda^{in}(F_{in+1}+x)$ visits every tile in $T$. Thus, $F$ is relatively dense in $\RR^2$ because there are only finitely many tiles in $\cP$.
    \item[iii.] For each $t\in T$, define $A_t=F\cap\sss t$. Then $F=\bigcup\limits_{t\in T} A_t$. Furthermore, $int(A_{t_i})\cap int(A_{t_j})=\emptyset$ whenever $t_i\neq t_j$, where $int(A_{t_i}),int(A_{t_j})$ are interiors of $A_{t_i},A_{t_j}$, respectively. By $(4)$ and the fact that $|\cP|$ is finite, there are only finitely many indices $t_1,\dots,t_m$ for some $m\in\ZZ^+$ such that for each $t\in T$, there exists $s_t\in\{t_1,\dots,t_m\}$ and $x_t\in\RR^2$ with $A_t=A_{s_t}+x_t$. \qedhere
\end{itemize}
\end{proof}

\begin{remark}
The condition $c$ in the proposition is akin to the definition of fractals by Mandelbrot \cite{Mandelbrot}. In particular, condition $c$ assures that there are finite number dissections of $F$, whose collection is denoted as $\cP$ (analogue to tile set $\cP$), such that $F$ can be written as a countable union of sets, each of which is a translational copy of the defined dissections of $F$ (i.e. each of which is a congruent copy of a tile in $\cP$). 
\end{remark}

\paragraph{An Example (Equithirds):}
Consider the substitution $\omega_{eq}$ given in Figure \ref{Equithirds}. Let $p$ denote the prototile with label $B^+$. Observe that there exists $x\in\RR^2$ such that $p_{B^+}+x\in\omega_{eq}^4(p_{B^+}+x)$ and 
$\sss p_{B^+}+x\cap\bb\omega_{eq}^4(p_{B^+}+x)=\emptyset$. Let $\{F_{eq,k}^{B^+}:\ k\in\ZZ^+\}$ denote the set of approximants of $F_{eq}^{B^+}$, first 8 of which are depicted in Figure \ref{equi_filled_iterations} and Figure \ref{equi_filled_iterations__2}. Note also that 
\[
{\displaystyle F_{eq,1}^{B^+}+x\subseteq \sqrt{3}^4F_{eq,4}^{B^+}+x,}
\]
as demonstrated in Figure \ref{l_fractal_like_set_example}. Hence, 
\[
{\displaystyle F=\bigcup\limits_{i=0}^{\infty}\sqrt{3}^{4i}F_{eq,4i+1}^{B^+}+x}
\]
is a relatively dense fractal-like set in the plane, by Proposition \ref{fractal_prop}.

\begin{figure}[H]
\centering
\begin{tikzpicture}
\node[anchor=south west,inner sep=0] at (0,0) {\includegraphics[scale=0.5]{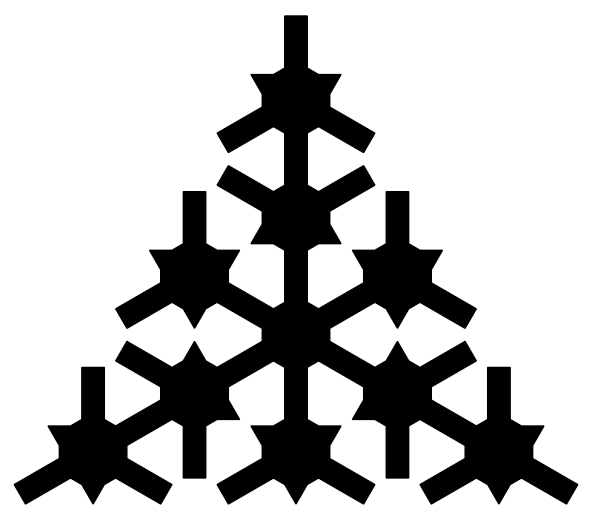}};
\tri{1.62}{2.16}{0}{0.33}
\trii{1.62}{2.16}{0}{0.33}

\trii{-1.62}{6.16}{0}{0.33}

\draw[ultra thick,->] (0.5,7.3) -- (2.5,4.6);
\end{tikzpicture}
\caption{}
\label{l_fractal_like_set_example}
\end{figure}

\bigskip\noindent
{\bf Acknowledgements.}
This work was supported by the Ministry of National Education of Turkey.

\bigskip\noindent
{\bf Data Availability Statement.}
Data sharing not applicable to this article as no datasets were generated or analysed during the current study.

\Addresses

\end{document}